\documentclass[a4paper, 10pt, american]{amsart}
\usepackage{geometry}
\usepackage{times,latexsym,amssymb}
\usepackage{amsmath,amsthm,bm}
\usepackage{color}
\usepackage[colorlinks,pdfpagelabels,pdfstartview = FitH,bookmarksopen
= true,bookmarksnumbered = true,linkcolor = blue,plainpages =
false,hypertexnames = false,citecolor = red,pagebackref=false]{hyperref}
\usepackage{soul}
\usepackage{microtype} 
\usepackage{amsbsy}
\usepackage{amstext}
\usepackage{amssymb}
\usepackage{esint}
\usepackage{stmaryrd}
\usepackage[pdftex]{graphicx}
\usepackage{floatflt}
\usepackage{comment}
\usepackage{cancel}

\allowdisplaybreaks
\newtheorem{mytheorem}{Theorem}
\newtheorem{mylem}[mytheorem]{Lemma}
\newtheorem{mydef}[mytheorem]{Definition}

\newtheorem{mycor}[mytheorem]{Corollary}
\newtheorem{remark}[mytheorem]{Remark}
\numberwithin{mytheorem}{section}
\numberwithin{equation}{section}

\SetSymbolFont{stmry}{bold}{U}{stmry}{m}{n}

\usepackage{schemata}
\usepackage{tikz}
\usetikzlibrary{decorations.pathreplacing,calc,tikzmark}
\usetikzlibrary{matrix,decorations.pathreplacing}

\usepackage{blindtext}

\newcommand{\uproman}[1]{\uppercase\expandafter{\romannumeral#1}}
\usepackage[utf8]{inputenc}
\usepackage{mathtools}

\DeclareMathOperator\inn{Int}
\DeclareMathOperator*{\esssup}{ess\,sup}

\DeclareMathOperator*{\essosc}{ess\,osc}

\DeclareMathOperator\divv{div}

\DeclareMathOperator\loc{loc}

\usepackage{mathrsfs}
\usepackage{yfonts}
\usepackage{braket}
\usepackage{bbm}

\usepackage{paralist}
\usepackage{pifont}
\usepackage{enumerate}

\usepackage{graphicx}
\newcommand{\bigchi}{\scalebox{1.3}{$\chi$}}

\makeatletter
\DeclareRobustCommand*{\bfseries}{%
  \not@math@alphabet\bfseries\mathbf
  \fontseries\bfdefault\selectfont
  \boldmath
}
\makeatother

\newcommand{\foo}[1]{\mathbf{#1}}

\allowdisplaybreaks
\newtheorem{myproposition}[mytheorem]{Proposition}
\numberwithin{mytheorem}{section}
\numberwithin{equation}{section}

\hypersetup{linkcolor=blue,linktoc=page}
\usepackage{titletoc}

\def\Yint#1{\mathchoice
    {\YYint\displaystyle\textstyle{#1}}%
    {\YYint\textstyle\scriptstyle{#1}}%
    {\YYint\scriptstyle\scriptscriptstyle{#1}}%
    {\YYint\scriptscriptstyle\scriptscriptstyle{#1}}%
      \!\iint}
\def\YYint#1#2#3{{\setbox0=\hbox{$#1{#2#3}{\iint}$}
    \vcenter{\hbox{$#2#3$}}\kern-.51\wd0}}
\def\longdash{{-}\mkern-3.5mu{-}} 
\def\fiint{\Yint\longdash}

\def\Xint#1{\mathchoice
{\XXint\displaystyle\textstyle{#1}}%
{\XXint\textstyle\scriptstyle{#1}}%
{\XXint\scriptstyle\scriptscriptstyle{#1}}%
{\XXint\scriptscriptstyle\scriptscriptstyle{#1}}%
\!\int}
\def\XXint#1#2#3{{\setbox0=\hbox{$#1{#2#3}{\int}$ }
\vcenter{\hbox{$#2#3$ }}\kern-0.555\wd0}}

\def\fint{\Xint-}

\newcommand{\vertiii}[1]{{\left\vert\kern-0.25ex\left\vert\kern-0.25ex\left\vert #1 
    \right\vert\kern-0.25ex\right\vert\kern-0.25ex\right\vert}}

\DeclareMathOperator\dist{dist}

\renewcommand{\d}{\mathrm{d}}
\newcommand{\dx}{\mathrm{d}x}

\newcommand{\dt}{\mathrm{d}t}
\newcommand{\ds}{\mathrm{d}s}

\newcommand{\tildexi}{\Tilde{\xi}}

\newcommand{\R}{\mathbb{R}}
\newcommand{\N}{\mathbb{N}}
\renewcommand{\epsilon}{\varepsilon}

\renewcommand{\d}{\mathrm{d}}
\newcommand{\F}{\mathcal{F}}
\newcommand{\G}{\mathcal{G}}
\newcommand{\B}{\mathcal{B}}
\newcommand{\h}{\mathcal{H}}

\renewcommand{\epsilon}{\varepsilon}


\begin{document}
\title[Gradient regularity for degenerate parabolic equations]{Gradient regularity for widely degenerate parabolic partial differential equations}
\date{\today}


\author[M. Strunk]{Michael Strunk}
\address{Michael Strunk \\
Fachbereich Mathematik, Universit\"at Salzburg\\
Hellbrunner Str. 34, 5020 Salzburg, Austria}
\email{michael.strunk@plus.ac.at}


\begin{abstract}
In this paper, we are interested in the regularity of weak solutions~$u\colon\Omega_T\to\mathbb{R}$ to parabolic equations of the type
\begin{equation*}
    \partial_t u - \mathrm{div} \nabla \mathcal{F}(x,t,Du) = f\qquad\mbox{in~$\Omega_T$},
\end{equation*}
where~$\mathcal{F}$ is only elliptic for values of~$Du$ outside a bounded and convex set~$E\subset \mathbb{R}^n$ with the property that~$0\in \mathrm{Int}{E}$. Here,~$\Omega_T \coloneqq\Omega\times(0,T)\subset\mathbb{R}^{n+1}$ denotes a space-time cylinder taken over a bounded domain~$\Omega\subset\mathbb{R}^n$ for some finite time~$T>0$. The function~$\mathcal{F} \colon \Omega_T\times\mathbb{R}^n \to\mathbb{R}_{\geq 0}$ present in the diffusion is assumed to satisfy: the partial mapping~$\xi\mapsto \mathcal{F}(x,t,\xi)$ is regular whenever~$\xi$ lies outside of~$E$, and vanishes entirely whenever~$\xi$ lies within this set. Additionally, the datum~$f$ is assumed to be of class~$L^{n+2+\sigma}(\Omega_T)$ for some parameter~$\sigma > 0$. As our main result we establish that
 \begin{equation*}
     \mathcal{K}(Du)\in C^0(\Omega_T)
 \end{equation*}
for any continuous function~$\mathcal{K}\in C^0(\mathbb{R}^n)$ that vanishes on~$E$. This article aims to extend the~$C^1$-regularity result for the elliptic case presented in~\cite{elliptisch} to the parabolic setting.
\end{abstract}


\maketitle
\vspace{-0.5cm}
\tableofcontents


\section{Introduction and main results}

 We investigate the gradient regularity of weak solutions to widely degenerate parabolic equations of the type
\begin{equation} \label{pde}
    \partial_t u - \divv \nabla \F(x,t,Du) = f\qquad\text{in~$\Omega_T$}.
\end{equation}
Here,~$\Omega$ denotes a bounded domain in~$\R^n$ ($n\geq 2$). The datum~$f$ is assumed to be of class~$L^{n+2+\sigma}(\Omega)$ for some parameter~$\sigma>0$. For every~$(x,t) \in \Omega_T$, the function~$\F\colon \Omega_T \times \R^n \to \R_{\geq 0}$ satisfies certain regularity and structural conditions; however, these conditions only apply to gradient values that lie outside some fixed bounded and convex subset~~$E\subset\R^n$ with the property that~$0\in\inn E$, while~$\F$ entirely vanishes within this set. We refrain from stating the precise conditions imposed on~$\F$ for the moment and rather refer the reader to Section~\ref{subsec:assumptions} for a detailed explanation. Outside the closure of the set~$E$ we assume~$\F$ to be non-degenerate with respect to the~$\xi$ variable, with an ellipticity constant that depends on the distance to~$E$. The notable feature of equation~\eqref{pde} is its extensive degeneracy, which places it in the category of~\textit{widely degenerate equations}. This situation poses challenges, as one can generally expect less regularity for weak solutions to~\eqref{pde} in comparison to equations that degenerate at only a single point, such as the well-known~$p$-Laplacian operator when~$1 < p < \infty$ and~$p \neq 2$. We recall that if~$p=2$, the~$p$-Laplacian simplifies to the Laplacian, which is not only linear but also elliptic across the whole of~$\R^n$. In particular, it cannot be expected that weak solutions will satisfy more than Lipschitz regularity in general. To exemplify this point, we consider a prototype example of widely degenerate equations, which serves as a fundamental reference. This example can be expressed as follows:
\begin{equation} \label{prototype}
    \partial_t u - \divv \left( a(x,t) \frac{(|Du|-1)^{p-1}_+ Du}{|Du|} \right) = f \quad \text{in } \Omega_T,
\end{equation}
where~$p > 1$ represents an arbitrary growth exponent, and~$x \mapsto a(x,t) \in W^{1,\infty}(\Omega)$ for any~$t\in(0,T)$ are coefficients that satisfy Lipschitz continuity with respect to the spatial variable~$x$. Additionally, these coefficients satisfy the condition~$0 < C^{-1} \leq a(x,t) \leq C$ for every~$(x,t) \in \Omega_T$ and some positive constant~$C > 0$. For equation~\eqref{prototype} the set of degeneracy~$E$ is given by the closed unit ball, i.e.~$E=\{\xi\in\R^n:|\xi|\leq 1\}\eqqcolon \overline{B}_1(0)$. It can be quickly verified that any time-independent~$1$-Lipschitz function satisfies the homogeneous variant of equation~\eqref{prototype} weakly. Hence, in general no more than Lipschitz regularity of weak solutions to~\eqref{pde} is to be expected.


\subsection{Literature overview}
In a few scenarios, certain higher regularity has been shown for the composition of a suitable function that vanishes within the set of degeneracy with the gradient~$Du$. For example, in~\cite{ambrosio2024regularity} it has been shown that the following higher Sobolev regularity
\begin{equation} \label{higherdiff}
    \frac{(|Du|-1)_+}{|Du|}Du \in L^2_{\loc}\big(0,T;W^{1,2}_{\loc}(\Omega)\big)
\end{equation}
holds true, provided the datum~$f$ itself also satisfies higher spatial differentiability. In case the datum does not exhibit any weak differentiability and is for example only locally square integrable in~$\Omega_T$, one can still infer the very same higher regularity as provided in the statement of~\eqref{higherdiff} for a different nonlinear function vanishing within any set that is slightly larger than the set of degeneracy, cf.~\cite{gentile2023higher}. A similar state of affairs pertains when investigating any continuity regularity of the gradient~$Du$. In the stationary setting, it has first been shown by Santambrogio and Vespri~\cite{santambrogio2010continuity} in the scalar case that the composition~$\mathcal{K}(Du)$ is continuous in~$\Omega$ for any continuous function~$\mathcal{K}\in C^0(\R^n)$ that vanishes inside the closed unit ball, provided that~$n=2$. This regularity result has been generalized by Colombo and Figalli in~\cite{colombo2017regularity} to the case where~$n\geq 2$ and the set of degeneracy is given by some bounded convex set~$E\subset\R^n$, not necessarily being the closed unit ball, with the property that the origin belongs to the interior of~$E$, i.e.~$0\in \mbox{Int$(E)$}$. We observe that these assumptions regarding the set of degeneracy~$E$ by Colombo and Figalli align with those outlined in this manuscript. Indeed, the results of Santambrogio and Vespri and Colombo and Figalli hold true in a very general setting, where they considered locally Lipschitz continuous solutions to equations of the type
\begin{equation} \label{pdeelliptisch}
    \divv \nabla \F(Du) = f \qquad\text{in~$\Omega$}
\end{equation}
for a convex and non-negative function~$\F\in C^2(\R^n\setminus \overline{E},\R_{ \geq 0})$ satisfying the respective stationary ellipticity estimate~\eqref{voraussetzung}. Recently, it has been derived by the author in~\cite{elliptisch} that the very same regularity results from~\cite{colombo2017regularity,santambrogio2010continuity} remain in fact true in case the function~$\F$ additionally exhibits a dependence on the spatial variable~$x$, provided~$\nabla\F = \nabla\F(x,\xi)$ is locally Lipschitz continuous with respect to~$x$ for any value~$\{\xi\in\R^n: |\xi|_E>1\}$. In the vectorial case, this regularity result has also been obtained for the model equation in~\cite{bogelein2023higher} by Bögelein, Duzaar, Giova, and Passarelli di Napoli. Mons~\cite{mons2023higher} established the same regularity result for vector valued solutions within a broader class of highly degenerate equations, where the vector field included in the diffusion term of the system is assumed to satisfy standard~$p$-growth conditions in the region where~$|\xi| > 1$. Furthermore, Grimaldi~\cite{grimaldi2024higher} considered a more general structural framework, where the usual Euclidean norm appearing in the prototype equation~\eqref{prototype} is replaced by a norm induced by a bounded, symmetric, and coercive bilinear form defined on the space~$\mathbb{R}^{Nn}$, and recovered the regularity result.

\,\\
In the parabolic setting, it has recently been shown by the very same authors in~\cite{bogelein2024gradient} that the above regularity result continues to hold true for the model equation, even if the vectorial case is considered. To our knowledge, a respective result has not been derived in the general setting of~\cite{colombo2017regularity,santambrogio2010continuity,elliptisch}, which is our aim in this very article. 


\subsection{Structure conditions}\label{subsec:assumptions} 
We now state the precise regularity assumptions made on the function~$\F$. As before, let~$E\subset\R^n$ denote a fixed bounded and convex set with~$0\in\inn E$. By~$|\cdot|_E\colon\R^n\to\R_{\geq 0}$, we denote the~\textit{Minkowski functional} on~$\R^n$, which is defined as
\begin{align} \label{minkowski}
    |\xi|_E \coloneqq \inf\{ t>0:\xi\in tE \},\qquad\xi\in\R^n,
\end{align}
where~$tE\coloneqq\{tx:x\in E\}$. This expedient mapping has already been considered in comparable works~\cite{colombo2017regularity,elliptisch} and will also turn out to be crucial in the course of this article. We refer the reader for the moment to Section~\ref{subsec:minkowski}, where a more detailed explanation on convex sets and the Minkowski functional~$|\cdot|_E$ is provided. We assume
$$\F\colon\Omega_T \times\R^n\to\R_{\geq 0}$$
to be a non-negative function that is subject to the following conditions
\begin{align} \label{fregularity}
        \begin{cases}
    \F(x,t,\xi) = 0  & \mbox{for any~$(x,t)\in\Omega_T$,~$\xi\in E$}, \\
   \xi \mapsto \F(x,t,\xi)~\mbox{is convex}  &\mbox{for any~$(x,t)\in\Omega_T$},  \\
    \xi \mapsto\F(x,\xi)\in C^1(\R^n) \cap C^2(\R^n\setminus \overline{E} ) & \mbox{for any~$x\in\Omega$},  \\ 
   |\partial_{x_i} \nabla\F(x,t,\xi)| \leq C(L) &\mbox{for any~$(x,t)\in\Omega_T$,~$|\xi|_E \leq L$} 
        \end{cases}
  \end{align}
for any~$i=1,\ldots,n$ and any~$L\geq 1$, where~$C=C(L)$ denotes a positive Lipschitz constant that depends on~$L\geq 1$. According to conditions~\eqref{fregularity}, we assume that~$\F$ admits convexity with respect to the gradient variable~$\xi$ and entirely vanishes whenever~$\xi\in E$. Moreover, the partial mapping~$\xi \mapsto \nabla^2\F(x,t,\xi)$ is assumed to be continuous outside the closure of the set~$E$ for any~$(x,t)\in\Omega_T$. In particular, the~$C^2$-regularity may break down at the boundary~$\partial E$. On the other hand,~$\nabla\F$ is assumed to be locally Lipschitz continuous with respect to~$x$ for any~$(x,t)\in\Omega_T$. With the notation~$\nabla$ and~$\nabla^2$, we consistently refer to derivatives with respect to the gradient variable~$\xi$, whereas~$\partial_x $ respectively~$\partial_t$ represent the derivative with respect to the spatial variable~$x$ and the time variable~$t$ respectively. \,\\

Finally, the following ellipticity condition for~$\F$ is assumed, which states that~$\F$ is elliptic for any value of~$\xi$ that lies outside the degeneracy set~$E$, i.e. for any~$\delta>0$ there exist constants~$0<\lambda(\delta)\leq\Lambda(\delta)$, such that there holds
\begin{equation} \label{voraussetzung}
  \qquad  \lambda(\delta) |\eta|^2 \leq \langle \nabla^2 \F(x,t,\xi) \eta,\eta \rangle \leq \Lambda(\delta) |\eta|^2\qquad\text{for any~$1+\delta \leq |\xi|_E \leq \delta^{-1}$}
\end{equation}
for all~$(x,t)\in \Omega_T$. The ellipticity constant~$\lambda(\delta)$ vanishes in the limit~$\delta\downarrow 0$, leading to a lack of ellipticity in general. Moreover, the upper bound for the Hessian is only local, with~$\Lambda(\delta)$ possibly becoming unbounded when~$\delta\downarrow 0$. \,\\


Let us now provide an example of a function~$\F$ that actually satisfies the structural conditions~\eqref{fregularity} and~\eqref{voraussetzung}. In a similar fashion to~\cite{elliptisch}, consider the model example
\begin{equation*}
    \F_p (x,t,\xi) = \frac{a(x,t)}{p}(|\xi|-1)^p_+ ,
\end{equation*}
where $x\mapsto a(x,t)\in W^{1,\infty}(\Omega,\R_{> 0})$ denote Lipschitz  coefficients with $|\partial_x a(x,t)| \leq A$ for all~$(x,t)\in \Omega_T$ that satisfy~$0<C_1 \leq a(x,t) \leq C_2<\infty$ for any~$(x,t) \in \Omega_T$. We note that in this particular setting, the set of degeneracy of~$\F_p$ is given by the closed unit ball. Then, it is easy to check that the Minkowski functional~$|\cdot|_E$ reduces to the standard Euclidean norm on~$\R^n$ and we have~$|\xi|_E=|\xi|$ for any~$\xi\in\R^n$. By following the very same approach taken in~\cite{elliptisch}, we infer that both structure conditions~\eqref{fregularity} as well as~\eqref{voraussetzung} are satisfied. In particular, we note that the eigenvalues of the Hessian~$\nabla^2\F_p$ become unbounded for large~$|\xi|$ in the case where~$p>2$, while in the sub-quadratic case~$1<p<2$ the largest eigenvalue of~$\nabla^2\F_p$ blows up when approaching the closed unit ball from outside. Only in the quadratic case where~$p=2$, the eigenvalues of~$\nabla^2\F_p$ are bounded universally by~$\Lambda$. For example, when~$p\geq 2$, there holds the quantitative estimate
\begin{equation} \label{prototypeexplizitell}
C_1 \frac{(|\xi|- 1)^{p-1}_+ }{|\xi|}|\eta|^2 \leq \langle \nabla^2\F_p(x,t,\xi)\eta,\eta\rangle  \leq C_2 (p-1) (|\xi|-1)^{p-2}_+ |\eta|^2
\end{equation}
for a.e.~$(x,t)\in\Omega_T$ and~$\xi,\eta\in\R^n$. For an arbitrary parameter~$0<\delta<1$, consider~$\xi\in\R^n$ with~$1+\delta\leq|\xi|\leq \delta^{-1}$. Consequently, the preceding quantitative estimate~\eqref{prototypeexplizitell} verifies that the ellipticity and boundedness condition~\eqref{voraussetzung} is satisfied with constants~$\lambda(C_1,\delta,p) = \lambda \delta^p$ and~$\Lambda(C_2,\delta,p)= \Lambda (p-1) \big(\frac{1-\delta}{\delta}\big)^{p-2}$ in the case where~$p\geq 2$.
The ellipticity constant~$\lambda$ vanishes in the limit~$\delta \downarrow 0$, leading to a lack of ellipticity in the closed unit ball~$\overline{B_1}$, while the largest eigenvalue of the Hessian becomes unbounded in the case where~$p>2$ when~$|\xi|$ increases. In the case~$p=2$, the constant~$\Lambda$ only depends on~$C_2$ but is independent of~$\delta$, while~$\Lambda$ depends on~$\delta$ and~$C_2$ when~$p>2$. By following a similar reasoning in the sub-quadratic setting~$1<p<2$, a different state of affairs to the super-quadratic case~$p>2$ applies, where the smallest eigenvalue of~$\nabla^2\F_p(x,t,\xi)$ vanishes for large~$|\xi|$ and the largest eigenvalue becomes unbounded whenever~$\xi$ approaches the unit sphere~$\{\xi\in\R^n:|\xi|=1\}$ from outside. 


\subsection{Definition of weak solution} \label{sec:weaksolution}

At this point, we present our notion of a weak solution to equation~\eqref{pde}. 


\begin{mydef} \label{defweakform}
   Let~$\F$ satisfy the structure conditions~\eqref{fregularity} and~\eqref{voraussetzung}. A measurable function
   $$u\in C^0\big(0,T;L^2_{\loc}(\Omega) \big) \cap L^\infty_{\loc}\big(0,T;W^{1,\infty}_{\loc}(\Omega)\big)$$
   is a local weak solution of equation~\eqref{pde}, if the integral identity
\begin{equation} \label{weakform}
\iint_{\Omega_T}  \big( - u \partial_t \phi + \langle \nabla\F(x,t,Du), D\phi \rangle \big)\,\dx\dt =  \iint_{\Omega_T} f\phi\,\dx\dt
\end{equation}
holds true for any test function~$\phi\in C^{\infty}_0(\Omega_T)$. 
\end{mydef}


\begin{remark} \label{lipschitzannahme} \upshape
    Without any additional conditions imposed on~$\F$, for instance~$p$-growth, we are forced to consider weak solutions that are \textit{a priori} locally Lipschitz continuous in~$\Omega_T$. Indeed, for the prototype example~\eqref{prototype} and more generally, also for functions satisfying standard~$p$-growth assumptions, with a right-hand side~$f$ belonging to a suitable Orlicz space, it has been established in~\cite{ambrosio2024gradient} that weak solutions exhibit a locally bounded spatial gradient in~$\Omega_T$ in the super-critical case~$p > \frac{2n}{n+2}$. Their result continues to hold true in the sub-critical range~$1<p\leq \frac{2n}{n+2}$ if weak solutions are assumed to be locally bounded resp. integrable to a certain exponent in~$\Omega_T$, cf.~\cite{dibenedetto1993degenerate}. Due to the general form of~$\F$ and the structure conditions~\eqref{voraussetzung}, without additional information on the growth rate of~$\xi \mapsto \mathcal{\F}(x,t,\xi)$, such as~$p$-growth, we cannot derive quantitative energy estimates for weak solutions to~\eqref{pde}. Already to ensure convergence of the integral involving the solution~$u$ in \eqref{weakform}, further knowledge about the qualitative behavior of~$\F$ would be necessary. Consequently, we focus on weak solutions that are \textit{a priori} of class
    $$ C^0\big(0,T;L^2_{\loc}(\Omega) \big) \cap L^\infty_{\loc}\big(0,T;W^{1,\infty}_{\loc}(\Omega)\big),$$
    which aligns with the treatment of the elliptic setting~\cite{colombo2017excess,colombo2017regularity,santambrogio2010continuity,elliptisch}.
\end{remark}


\subsection{Main result}

In this section, we present the main result of our paper, along with a subsequent corollary.


\begin{mytheorem} \label{hauptresultat}
Let~$n\geq 2$,~$f\in L^{n+2+\sigma}(\Omega_T)$ for some~$\sigma>0$, and
$$u \in C^0\big(0,T;L^2_{\loc}(\Omega) \big) \cap L^\infty_{\loc}\big(0,T;W^{1,\infty}_{\loc}(\Omega)\big)$$
be a local weak solution to~\eqref{pde} under structure conditions~\eqref{fregularity} and~\eqref{voraussetzung}, where~$E\subset \R^n$ denotes a bounded and convex set with~$0\in\inn E$ on which~$\F$ degenerates. Then, there holds
\begin{equation*}
    \mathcal{K}(Du)\in C^0(\Omega_T)
\end{equation*}
for any~$\mathcal{K}\in C^0(\R^n)$ with~$\mathcal{K}\equiv 0$ on~$E$.
\end{mytheorem}


As an immediate consequence of Theorem~\ref{hauptresultat}, we obtain the following corollary.

\begin{mycor} \label{corollaryzwei}
    Let the assumptions of Theorem~\ref{hauptresultat} hold true. If there additionally holds
    $$\nabla\F \in C^0(\Omega_T\times \R^n ,\R^n),$$
    then we have
    \begin{equation*}
        \nabla \F(x,t,Du) \in C^0(\Omega_T,\R^n).
    \end{equation*}
\end{mycor}


\subsection{Strategy of the proof} \label{strategy}
In this section, we delineate the fundamental steps leading to our main regularity result, that is Theorem~\ref{hauptresultat}. We start by introducing a cutoff for the functional~$\F$ when the values of~$|\xi|$ exceed a certain threshold~$L < \infty$. We then redefine~$\F$ to~$\Tilde{\F}$, ensuring that the mapping~$\xi \mapsto \Tilde{\F}(x,t,\xi)$ remains constant for points where~$|\xi| \geq L$, a property that holds true for any~$(x,t) \in \Omega_T$. The threshold~$L$ in fact corresponds to the essential supremum of~$|Du|$ within a compactly contained cylinder~$Q_R = Q_R(y_0,s_0) \Subset \Omega_T$, where
$$u \in L^\infty_{\loc}(0,T;W^{1,\infty}_{\loc}(\Omega))$$
is taken to denote an arbitrary local weak solution to equation~\eqref{pde}, as defined in Definition~\ref{defweakform}. Given our definition of weak solution, indeed it follows that~$L = \|Du\|_{L^\infty(Q_R)} < \infty$. With this construction in place, any local weak solution~$u$ to the equation~\eqref{pde} can be regarded as a local weak solution to the redefined equation
$$\partial_t u - \divv \nabla\Tilde{\F}(x,t,Du) = f \qquad \text{in } Q_R.$$
 In general however, the truncated function~$\Tilde{\F}$ fails to be convex, which we require for the structure conditions~\eqref{voraussetzung} to hold true. This problem is overcome by adding a suitable convex function~$\Phi$ to~$\Tilde{\F}$, such that the sum~$\hat{\F}\coloneqq \Tilde{\F}+\Phi$ is once again convex. Furthermore, we enhance the diffusion term of the original equation~\eqref{pde} by adding the scaled Laplacian term~$\epsilon \Delta u$, with~$\epsilon \in (0,1]$ serving as a small approximation parameter. This modification results in a more regularized version of~\eqref{pde}, where we denote the corresponding functional as~$\hat{\F}_\epsilon$, which is elliptic across all of~$\R^n$ with respect to~$\xi$ and avoids degeneracy within the closed unit ball. It is important to note that the ellipticity constant generally depends on the parameter~$\epsilon\in(0,1]$. Moreover, the obtained mapping~$\xi \mapsto \hat{\F}_\epsilon(x,t,\xi)$ exhibits quadratic growth across~$\R^n$ for any~$(x,t)\in\Omega_T$. In the next step, we examine the unique weak solution $$u_\epsilon \in L^2\big(\Lambda_R(s_0),u+W^{1,2}_0(B_R(y_0))\big)$$ to the weak formulation of the approximating equation on the cylinder~$Q_R=Q_R(y_0,s_0) \Subset \Omega_T$, where we impose a Cauchy-Dirichlet condition on~$\partial_p Q_R$, that involves the weak solution~$u$ of equation~\eqref{pde}. The uniform ellipticity of the regularized equation enables us to establish the existence of second-order weak spatial derivatives, leading to the regularity
$$ u_\epsilon \in L^\infty_{\loc}\big(\Lambda_R(s_0);W^{1,\infty}_{\loc}(B_R(y_0)) \big) \cap L^2_{\loc}\big(\Lambda_R(s_0);W^{2,2}_{\loc}(B_R(y_0)) \big)$$
for the approximating solutions. Additionally, owing to the quadratic growth of~$\xi \mapsto \hat{\F}_\epsilon(x,t,\xi)$ for any~$(x,t) \in \Omega_T$, we derive a quantitative local~$L^\infty$-gradient bound for~$Du_\epsilon$, with the~$L^\infty$-gradient norm essentially controlled from above by the~$L^2$-gradient norm, up to constants and additive terms. This expedient matter of fact allows us to establish a local uniform bound for~$\|Du_\epsilon\|_{L^\infty}$ in the cylinder~$Q_R$ with respect to the approximating parameter~$\epsilon \in (0,1]$.\\
The primary challenge that follows is to prove Theorem~\ref{holdermainresult}, which asserts the local Hölder continuity in~$Q_R$ of the function
$$\G_\delta(Du_\epsilon) \coloneqq \frac{(|Du_\epsilon|_E - (1+\delta))_+}{|Du_\epsilon|_E} Du_\epsilon $$
for any~$\delta,\epsilon \in (0,1]$. The Hölder exponent~$\alpha_\delta \in (0,1)$ and the positive Hölder constant~$C_\delta$ both depend on the data and the parameter~$\delta \in (0,1]$, yet remain independent of~$\epsilon \in (0,1]$. The mapping~$\G_\delta$ vanishes within the slightly larger set~$\{\xi\in\R^n:|\xi|_E\leq 1+\delta \}$, while maintaining positivity outside the set~$\{\xi \in \R^n : |\xi|_E > 1+\delta\}$. As noted previously in \cite{bogelein2023higher}, it is beneficial to examine the function ~$\G_\delta(Du_\epsilon)$ because in the region where~$\G_\delta$ is positive, the equation described by~\eqref{pde} exhibits uniform ellipticity, with the ellipticity constant only depending on the parameter~$\delta\in(0,1]$ but not on~$\epsilon\in(0,1]$. The proof of Theorem~\ref{holdermainresult} requires a meticulous approach to both the non-degenerate and degenerate regime. In the non-degenerate regime, the set of points where~$\partial_{e^*}u_\epsilon$ is close to its supremum is large in measure for at least one element~$e^*\in\partial E^*$, while conversely, the degenerate regime encompasses a subset of points where~$\partial_{e^*}u_\epsilon$ is rather far away from its supremum for any~$e^*\in\partial E^*$. The set~$E^*\subset\R^n$ denotes the unit ball in the dual norm induced by the convex set~$E\subset\R^n$. The reader is provided with further information in Section~\ref{subsec:minkowski}. In the non-degenerate regime, our objective is to deduce a lower bound for ~$|Du_\epsilon|_E$ within the smaller cylinder~$Q_{\frac{\rho}{2}}(z_0) \Subset Q_{2\rho}(z_0) \Subset Q_R \Subset \Omega_T$. This result is accomplished by leveraging the measure-theoretic information that characterizes the non-degenerate regime, which is assumed to hold on~$Q_\rho(z_0)$. By differentiating the regularized equation, we deduce that for each~$e^*\in\partial E^*$, the function~$\partial_{e^*}u_\epsilon$ serves as a weak solution to a linear parabolic equation that exhibits elliptic and bounded coefficients. Importantly, the ellipticity constant for this equation is independent of~$\epsilon \in (0,1]$ and relies solely on the parameter~$\delta \in (0,1]$ as well as the provided data. This matter of fact enables us to apply established excess-decay estimates for parabolic De Giorgi classes, as referenced in \cite{dibenedetto2023parabolic}. In the degenerate regime, we exploit the fact that 
$$ v_\epsilon \coloneqq (\partial_{e^*}u_\epsilon - (1+\delta))_+^2,$$ 
serves as a sub-solution to a linear parabolic equation that incorporates elliptic and bounded coefficients, which again allows us to derive a parabolic De Giorgi class-type estimate. This strategy yields a quantitative reduction of the supremum of~$\G_\delta(Du_\epsilon)$ over a smaller cylinder, as delineated in Proposition~\ref{degenerateproposition}. By intertwining the results of both Propositions~\ref{nondegenerateproposition} and~\ref{degenerateproposition} through a sequence of shrinking nested cylinders, we adeptly navigate between the non-degenerate and degenerate regimes, adjusting our approach based on the measure-theoretic information in each step. The transition between these regimes is crucial, shaping the subsequent analysis. If the non-degenerate regime is applicable at a certain step, this condition persists. However, moving to a smaller cylinder within the degenerate regime yields uncertainty regarding whether the subsequent step will fall under the non-degenerate or degenerate regime. Once Hölder continuity of~$\G_\delta(Du_\epsilon)$ has been established for any~$\delta,\epsilon\in(0,1]$, we can invoke Arzelà-Ascoli's theorem and first pass to the limit~$\epsilon\downarrow 0$ to obtain that~$\G_\delta(Du)$ is also locally Hölder continuous in~$Q_R$, with the Hölder exponent and constant dependent on~$\delta$ and the given data. Moreover, by subsequently also passing to the limit~$\delta \downarrow 0$, we conclude that the limit function 
$$\G(Du)= \frac{(|Du|_E-1)_+}{|Du|_E} Du$$ 
is locally uniformly continuous in~$Q_R$. At this point, however, the quantitative control over the Hölder exponent and constant is lost, leaving uncertainty about the optimal modulus of continuity of the limit function~$\G(Du)$. Ultimately, the main regularity Theorem~\ref{hauptresultat} can be derived from these findings. For further insight, we encourage the reader to consult the reasoning employed in \cite{bogelein2024gradient,elliptisch} for the proofs of both Theorem~\ref{holdermainresult} and Theorem~\ref{hauptresultat}. 


\subsection{Novelties and Significance} \label{novelty} 
In comparison to known~$C^1$-regularity results for widely degenerate parabolic equations, where the article by Bögelein, Duzaar, Giova, and Passarelli di Napoli~\cite{bogelein2024gradient} is the only available result to our knowledge, our main result, Theorem~\ref{hauptresultat}, is new even in the case where the diffusion term in the considered equation~\eqref{pde} does not depend on the space-time variable~$(x,t) \in \Omega_T$. Nevertheless, the flexibility of our proof strategy, that is drawn from the pioneering works of De Giorgi, DiBenedetto, Friedman, and Uhlenbeck, also allows us to treat the more general case where the diffusion term does depend on~$(x,t)$ under the structure assumptions~\eqref{fregularity} - \eqref{voraussetzung}. Moreover, in contrast to the parabolic result of~\cite{bogelein2024gradient}, where the degeneracy set is given as the unit ball, while for gradient values~$\{|Du|>1 \}$ their equation satisfies standard~$p$-growth conditions, we establish Theorem~\ref{hauptresultat} for equations that not only degenerate on a general bounded convex subset~$E\subset \R^n$, with the property that~$0\in\inn E$, but also may exhibit arbitrary growth for gradient values~$\{ |Du|_E>1 \}$. This generality was previously achieved in the elliptic setting by Colombo and Figalli~\cite{colombo2017regularity}, and subsequently by the author in~\cite{elliptisch} with equations allowed to depend on the spatial variable~$x\in\Omega$. \,\\ 
The core difficulty of the present article lies in proving Theorem~\ref{holdermainresult}, which establishes the local Hölder continuity of the approximating mappings~$\G_\delta(Du_\epsilon)$ within~$Q_R(y_0,s_0)\Subset \Omega_T$ for any~$\delta,\epsilon \in (0,1]$. Our main contribution, however, resides in the non-degenerate regime treated in Section~\ref{sec:nondegenerate}, where we begin by establishing a lower bound for~$|Du_\epsilon|_E$ in Proposition~\ref{lowerboundprop}. This is the main ingredient for the proof of Proposition~\ref{nondegenerateproposition}. The proof of Proposition~\ref{lowerboundvs} mirrors the strategy developed by Kuusi and Mingione in~\cite{kuusi2013gradient}, built upon a De Giorgi–type iteration argument. In the setting of widely degenerate equations, a careful choice of test function is crucial in order to obtain an appropriate Caccioppoli-type estimate: although the approximating equation is non-degenerate due to the presence of the scaled Laplacian term~$\epsilon \Delta u_\epsilon$, one must avoid introducing dependence on the parameter~$\epsilon\in(0,1]$ in the quantitative constants and in the data. Therefore, our test function in Proposition~\ref{lowerboundprop} is designed not only to suppress small gradient values, but also to exclude excessively large gradient values. This leads to an~\textit{a priori} insufficient measure-theoretic information expressed in~\eqref{est:meascondinterscaled}, which does not guarantee a lower bound for~$|Du_\epsilon|_E$ in general. It is precisely at this juncture that an additional argument becomes necessary, one that leverages the expedient regularity property 
$$ (\partial_{e^*}u_\epsilon-(1+\delta))^2_+ \in C^0_{\loc} \big(\Lambda_R(s_0) ;L^2_{\loc}(B_R(y_0)) \big) $$
for any~$e^*\in \partial E^*$,~$\delta,\epsilon\in(0,1]$, which is inferred in Lemma~\ref{lem:stetigkeitinzeit} a similar manner to that used by DiBenedetto~\cite{dibenedetto1993degenerate}. 


\subsection{Plan of the paper} \label{planpaper}
Our paper is structured in the following way: in Section~\ref{sec:preliminaries}, we will begin by introducing the notation and framework, along with supplementary material required for later sections, as well as our definition of a weak solution to equation~\eqref{pde}. Following this, in Section~\ref{sec:holder}, we will present the proof of the main regularity result that is Theorem~\ref{hauptresultat}. Thereby, an intermediate regularity result will be established in Theorem~\ref{holdermainresult}, which is inferred from Proposition~\ref{nondegenerateproposition} as well as Proposition~\ref{degenerateproposition}. Section~\ref{sec:energyestimates} serves as an intermediate passage, during which a series of energy estimates and other conclusions will be derived by differentiating the equation. Eventually, the proofs of both Propositions~\ref{nondegenerateproposition} and~\ref{degenerateproposition} will be addressed in the subsequent Sections~\ref{sec:nondegenerate} and~\ref{sec:degenerate}, concluding the article. 


\subsection*{Acknowledgements}
The author would like to sincerely thank Professor Verena Bögelein for her guidance during the development of this article. \\
This research was funded in whole or in part by the Austrian Science Fund (FWF) [10.55776/P36295]. For open access purposes, the author has applied a CC BY public copyright license to any author accepted manuscript version arising from this submission. \,\\

\textbf{Conflict of Interest.} The author declares that there is no conflict of interest. \,\\

\textbf{Data availability.} This manuscript has no associated data.


\section{Preliminaries} \label{sec:preliminaries}
\subsection{Notation and setting}

Throughout this paper,~$\Omega_T\coloneqq\Omega\times(0,T)$ denotes a space-time cylinder that consists of some bounded domain~$\Omega\subset\R^n$, where~$2\leq n\in\N$, while the positive and finite time~$0<T<\infty$ represents the height of the cylinder~$\Omega_T$. The parabolic boundary of~$\Omega_T$ will be indicated in the usual way by
$$\partial_p\Omega_T \coloneqq \big(\overline{\Omega}\times\{0\}\big)\cup\big(\partial\Omega\times(0,T)\big),$$
where~$\partial\Omega \subset \R^{n-1}$ denotes the spatial boundary. Given a function~$f\in L^1(\Omega_T)\cong L^1(0,T;L^1(\Omega))$, we shall also sometimes write~$f(t)$ instead of~$f(\cdot,t)$ whenever it turns out convenient. The standard scalar product on~$\R^n$ will be denoted by~$\displaystyle\langle \cdot\rangle$ and the dyadic product of two vectors~$\xi,\eta \in\R^n$ by~$\xi \otimes \eta$. The positive part of a real quantity~$a\in\R$ is denoted as~$a_+ = \max\{a,0\}$, while the negative part is denoted as~$a_- = \max\{-a,0\}$. Constants are consistently represented as~$C$ or~$C(\cdot)$ and their dependence is described solely on their variables, not their specific values. However, it is possible for constants to vary from one line to another without further clarification. 
\,\\

We define the open ball in~$\mathbb{R}^n$ with radius~$\rho>0$ and center~$x_0\in\mathbb{R}^n$ as 
$$B_\rho(x_0) \coloneqq\{x\in\mathbb{R}^n\colon~|x-x_0|<\rho\}.$$
By a point~$z_0\in\R^{n+1}$, we always refer to~$z_0 = (x_0,t_0)$ where~$x_0\in\R^n$ and~$t_0\in\R_{\geq 0}$. If~$x_0=0$, it will be common to simplify notation by writing~$B_\rho$ instead of~$B_\rho(x_0)$.  
Additionally, we may omit the vertex of the cylinder for convenience. The \textit{general (backward) parabolic cylinder} with vertex~$z_0\in\R^{n+1}$ is defined as 
$$Q_{R,S}(z_0) \coloneqq B_R(x_0)\times (t_0-S,t_0],$$
 with the standard (backward) parabolic cylinder with vertex~$z_0\in\R^{n+1}$ given by
$$Q_\rho(z_0) = B_R(x_0)\times\Lambda_\rho(t_0) \coloneqq B_\rho(x_0)\times (t_0-\rho^2,t_0].$$
Next, we will introduce the \textit{parabolic distance} between~$z_1$ and~$z_2$, where~$z_1=(x_1,t_1)$,~$z_2 = (x_2,t_2) \in \Omega_T$, that is given by
$$d_{p}(z_1,z_2) \coloneqq |x_1-x_2| + \sqrt{|t_1-t_2|}.$$
For any~$f\in L^1(\Omega_T,\R^k)$ and~$k\in\N$, we use the common abbreviation for the average integral of~$f$. Let~$A\subset\Omega$ such that~$\mathcal{L}^n(A)>0$. Then, we define the slice-wise mean~$(f)_A\colon (0,T)\to\R^n$ as follows
\begin{align*} (f)_A(t) \coloneqq \displaystyle\fint_A f(\cdot,t)\,\dx \quad\mbox{for a.e.~$t\in(0,T)$}.
\end{align*}
Since many of the functions in this paper are continuous with respect to the time variable, that is,~$f\in C^0(0,T;L^1(\Omega))$, the slice-wise mean value is well-defined for all~$t\in(0,T)$. In a similar fashion, we denote the mean value of~$f$ on some measurable set~$E\subset \Omega_T$ with~$\mathcal{L}^{n+1}(E)>0$ by
\begin{align*}
     (f)_E \coloneqq \displaystyle\fiint_E f(x,t)\,\dx\dt.
\end{align*}
If it is evident from the context, the vertex~$z_0$ may sometimes be omitted to simplify notation. \,\\

The operators~$\nabla \F$ and~$\nabla^2\F$ denote the first and second order derivatives of the function~$\F(x,t,\xi)$ with respect to~$\xi$ in the region~$\mathbb{R}^n \setminus \overline{E}$. Meanwhile,~$\partial_x \nabla \F(x,t,\xi)$ signifies the derivative of~$\nabla \F$ with respect to the spatial variable~$x$. Occasionally, we will refer to~$\nabla\F$ simply as~$\h$, and both notations will be used interchangeably throughout the discussion. The terms \textit{spatial gradient}, \textit{gradient}, and \textit{spatial derivative} will all be used interchangeably without any distinctions. To avoid confusion and to clearly differentiate between derivatives pertaining to weak solutions~$u$ of equation~\eqref{pde} and derivatives of the function~$\F$ incorporated in the diffusion term, we will adopt the following conventions: weak partial derivatives of~$u$ will be denoted by~$D_i u$ for~$i \in \{1, \ldots, n\}$, directional derivatives~$v\in\R^n$ by~$\partial_v u\coloneqq \langle Du,v\rangle$ whenever this quantity exists, and its weak gradient will be denoted by~$Du$. In contrast, we will use~$\nabla \F = \h$ for the first order derivative of~$\F$ and~$\nabla^2 \F$ for its Hessian. To distinguish between spatial and time derivatives, we will indicate the latter as~$\partial_t u$, provided that~$u\colon \Omega_T \to \mathbb{R}$ possesses a (weak) derivative with respect to the time variable~$t$. It is important to note that no distinction will be made between classical and weak derivatives, with context providing clarity regarding the intended meaning. For any other differentiable mapping~$f\colon \mathbb{R}^n \to \mathbb{R}^m$, where~$n, m \in \mathbb{N}$, we may use either~$Df$ or~$\nabla f$ to represent its derivative, depending on the context. \,\\ 

The following function will turn out expedient in the further course of this article and is utilized in a similar manner to~\cite{bogelein2023higher,bogelein2024gradient,elliptisch}.  For~$\delta\geq 0$, we define
\begin{equation} \label{gdeltafunk}
    \G_\delta(\xi) \coloneqq \frac{(|\xi|_E - (1+\delta))_+ }{|\xi|_E} \xi \qquad\text{for~$\xi\in\R^n$}.
\end{equation}
In the case where~$\delta=0$, we shall simply write
\begin{equation} \label{gfunk}
    \G(\xi) \coloneqq \G_0(\xi) = \frac{(|\xi|_E-1)_+}{|\xi|_E}\xi \qquad\text{for~$\xi\in\R^n$}.
\end{equation}
The following bilinear form on~$\R^n$ naturally arise by differentiating the weak form~\eqref{weakform} with respect to the spatial variable~$x$. For arbitrary~$\eta,\zeta\in\R^n$, we let
 \begin{equation} \label{bilinear}
            \mathcal{B}(\cdot,\xi)(\eta,\zeta)\coloneqq \langle D^2\F(\cdot,\xi)\eta, \zeta \rangle \qquad\text{for~$(x,\xi)\in \Omega_T \times(\R^n\setminus \overline{E})$}.
        \end{equation}


\subsection{Convex sets and the Minkowski functional} \label{subsec:minkowski}

In this article, we consider a bounded and convex set~$E\subset\R^n$ with~$0\in\inn{E}$, which serves as the set of degeneracy for the equation presented in \eqref{pde}. The convexity of~$E$ is crucial as it induces a map on~$\R^n$, which will be particularly important for deriving energy estimates in Sections \ref{sec:holder} -- \ref{sec:degenerate}. Furthermore, the inclusion of~$0\in\inn E$ is essential for the proof of Lemma \ref{gdeltalem}, as it ensures the Lipschitz continuity of the Minkowski functional~$|\cdot|_E\colon\R^n\to\R_{\geq 0}$ on~$\R^n$, defined by
\begin{align*}
    |\xi|_E \coloneqq \inf\{ t>0:\xi\in tE \},\qquad\xi\in\R^n,
\end{align*}
where~$tE\coloneqq\{tx:x\in E\}$. For further insights into the properties of the Minkowski functional, we direct the reader to \cite{minkowski}, particularly pages 109 and 119. Since the set~$E$ consists of points satisfying~$|\xi|_E\leq 1$, it corresponds to the unit ball with respect to the Minkowski functional. We also introduce its dual map, defined as follows:
\begin{align} \label{dualnorm}
    |\xi|_{E'}\coloneqq \sup\limits_{|e|_E \leq 1} \langle\xi,e \rangle = \sup\limits_{e \in E} \langle\xi,e \rangle
\end{align}
along with the unit ball corresponding to this dual map:
\begin{align} \label{unitballdualnorm}
    E^*\coloneqq \overline{B}_1^{|\cdot|_{E'}} = \{\xi\in\R^n:|\xi|_{E'}\leq 1\} = \{\xi\in\R^n:\langle\xi,e\rangle \leq 1~\forall e\in E\}.
\end{align}
As established in \cite[Section~3]{colombo2017regularity}, we obtain the following alternative representation formula for the Minkowski functional:
\begin{align} \label{minkowskialternativ}
    |\xi|_E = \underset{\substack{\vspace{-0.05cm}\\ e^*\in E^{*}}}{\sup} 
\langle\xi,e^*\rangle  = \underset{\substack{\vspace{-0.05cm}\\ e*\in \partial E^{*}}}{\sup} \langle\xi,e^*\rangle
\end{align}
for any~$\xi\in\R^n$. A key implication of this representation is that, in several energy estimates, we will weakly differentiate a given equation in the direction~$e^*$, where~$e^*\in \partial E^*$. This allows us to derive information about the directional derivatives~$\partial_{e^*} u = \langle Du,e^*\rangle$. Since~$e^*\in E^*$ can be chosen arbitrarily, we may pass to the supremum over all~$e^*\in E^*$ to obtain estimates for the Minkowski functional~$|Du|_E$ of the (weak) gradient of~$u$. We denote the (weak) directional derivative in any direction~$v\in\R^n$ as~$\partial_v u$, where we use the representation formula
\begin{align} \label{representation}
     v=\sum\limits_{i=1}^n v_i e_i, 
\end{align}
with~$(e_i)_{i=1,\ldots,n}$ being the standard base of~$\R^n$ and~$v_i$ denoting the respective coefficients of~$v$. Thus, for a.e.~$v\in\R^n$, we have~$\partial_v u = \langle Du,v \rangle$.
It is also evident from the definition that~$|\cdot|_E$ is homogeneous for any~$\lambda>0$, satisfying~$|\lambda\xi|_E=\lambda|\xi|_E$ for any~$\xi\in\R^n$. Additionally, we have~$|0|_E=0$, and since~$0\in\inn{E}$, it follows that~$|\xi|_E<\infty$ for all~$\xi\in\R^n$. Given that~$0\in\inn{E}$ and~$E$ is bounded, there exist radii~$0<r_E\leq R_E<\infty$, such that
\begin{align} \label{mengeeradii}
    B_{r_E}(0)\subset E\subset B_{R_E}(0),
\end{align}
where~$r_E>0$ is chosen as the largest radius and~$R_E>0$ denotes an arbitrary but fixed radius with this property. Consequently, from the definition of \eqref{minkowski}, we can deduce that
\begin{align} \label{betragminkowski}
    \frac{|\xi|}{R_E} \leq |\xi|_E \leq \frac{|\xi|}{r_E}
\end{align}
for any~$\xi\in\R^n$. Finally, for any~$\delta>0$, we define the outer parallel set of~$E$ at a distance~$\delta>0$ as:
\begin{align} \label{Edelta}
    E_\delta\coloneqq \{\xi\in\R^n:|\xi|_E\leq 1+\delta\}.
\end{align}
   

\subsection{Mollification in time} \label{subsec:mollificationintime} 
According to our definition, weak solutions may not necessarily exhibit weak differentiability with respect to the time variable. In order to overcome this lack of regularity, we introduce the standard regularization technique involving Steklov-means, where we refer to~\cite{steklovproperties} for some properties and additional information. Given a function~$f \in L^1(\Omega_T)$ and~$0<h<T$, we define its {\it Steklov-average}~$[f]_{ h}$ by 
\begin{equation}\label{def-stek-right}
	[f]_{h}(x,t) 
	\coloneqq
\begin{cases}
    \displaystyle{\frac{1}{h} \int_{t}^{t+h} f(x,\tau) \,\d\tau } 
		& \mbox{for~$t\in (0,T-h)$},  \\ 
		0 
		& \mbox{for~$t\in [T-h,T)$}.
\end{cases}
\end{equation}
Rewriting the identity~\eqref{weakform} in terms of Steklov-means~$[u]_h$ of~$u$ and integrating by parts, yields 
\begin{align}
    \int_{\Omega\times\{t\}}\big(\partial_t [u]_h \phi + \langle [\nabla\F(x,t,Du)]_h, D\phi\rangle\big)\,\dx
   = \int_{\Omega\times\{t\}} [f]_h \phi\,\dx \label{lösung-steklov}
\end{align}
for any test function~$\phi \in C^{\infty}_0(\Omega)$ and any~$t\in(0,T)$. 


\subsection{Auxiliary material}
In this section, a series of algebraic inequalities and further additional material is summarized. \,\\
The following results, Lemma~\ref{lem:minkowskitriangle} -- Lemma~\ref{gdeltalem}, are taken from~\cite[Lemma~2.1 -- Lemma~2.7]{elliptisch}. The first one presents a triangle inequality and a generalized inverse triangle inequality for the Minkowski functional~$|\cdot|_E$.


\begin{mylem} \label{lem:minkowskitriangle}
    For any~$\xi,\eta\in\R^n$ there holds
    \begin{align} \label{est:minkowskitriangle}
        |\xi+\eta|_E\leq |\xi|_E+|\eta|_E.
    \end{align}
    In particular, there holds the following reverse triangle inequality
    \begin{align} \label{est:minkowskireversetriangle}
        ||\xi|_E-|\eta|_E| \leq \max\{|\xi-\eta|_E,|\eta-\xi|_E\}.
    \end{align}
\end{mylem}

\begin{remark} \label{remarkzutriangle} \upshape
 In general, however, the reverse triangle inequality~\eqref{est:minkowskitriangle} stated in Lemma~\ref{est:minkowskireversetriangle} does not simplify to the standard reverse triangle inequality applicable to the Euclidean norm, as the Minkowski functional~$|\cdot|_E$ is not necessarily symmetric.
\end{remark}


Due to the fact that~$0\in\inn{E}$, we derive the following Lipschitz estimate for~$|\cdot|_E$.

\begin{mylem} \label{lem:minkowskilipschitz}
For any~$\xi,\eta\in\R^n$ there holds the Lipschitz estimate
\begin{align} \label{est:lipschitz}
    ||\xi|_E - |\eta|_E | \leq \mbox{$\frac{1}{r_E}$} |\xi-\eta|.
\end{align}
Here,~$r_E>0$ denotes the radius introduced in~\eqref{mengeeradii}.
\end{mylem}


The next lemma states an algebraic inequality similar to respective estimates for the Euclidean Norm~\cite[Lemma~2.1]{bogelein2023higher}.

\begin{mylem} \label{lem:algineq}
    For any~$\xi,\eta\in\R^n\setminus\{ 0\}$ there holds
    \begin{equation*}
        \bigg|\frac{\xi}{|\xi|_E}-\frac{\eta}{|\eta|_E}\bigg|_E \leq \frac{R_E}{r_E} \frac{2}{|\xi|_E} |\xi-\eta|_E.
    \end{equation*}
\end{mylem}


The subsequent lemma states a relation between the mapping~$\G_\delta$ introduced in~\eqref{gdeltafunk} and the Euclidean norm, which is an important tool for the derivation of the intermediate regularity result Theorem~\ref{holdermainresult} and, subsequently, also Theorem~\ref{hauptresultat}. 

\begin{mylem} \label{gdeltalem}
    Let~$\delta\geq 0$ and~$\xi,\eta\in\R^n$. Then, there holds 
    \begin{equation} \label{est:gdeltalipschitzeins}
        |\G_\delta(\xi)-\G_\delta(\eta)| \leq  3 \Big(\frac{R_E}{r_E}\Big)^2 |\xi-\eta|.
    \end{equation}
    Additionally, if~$\delta>0$ and~$|\xi|_E \geq 1+\delta$, we have
    \begin{equation} \label{est:gdeltalipschitzzwei}
        |\xi-\eta| \leq 3 \Big(\frac{R_E}{r_E}\Big)^2 \Big(1+ \frac{1}{\delta} \Big) |\G(\xi)-\G(\eta)|.
    \end{equation}
\end{mylem}


The subsequent iteration lemma allows a technique of reabsorbing certain quantities and stems from~\cite[Lemma~6.1]{giusti2003direct}.


\begin{mylem} \label{iterationlem}
       Let $\phi$ be a bounded, non-negative function on~$0\leq R_0 < R_1$ and assume that for $R_0\leq \rho < r \leq R_1$ there holds
      $$\phi(\rho) \leq \eta \phi(r) + \frac{A}{(r-\rho)^{\alpha}} + \frac{B}{(r-\rho)^{\beta}} + C$$
      for some constants $A,B,C,\alpha\geq \beta\geq 0$, and $\eta\in(0,1)$. Then, there exists a constant $\Tilde{C}=\Tilde{C}(\eta,\alpha)$, such that for all $R_0\leq \rho_0<r_0\leq R_1$ there holds
      $$\phi(\rho_0) \leq \Tilde{C}\bigg( \frac{A}{(r_0-\rho_0)^{\alpha}} + \frac{A}{(r_0-\rho_0)^{\alpha}} + C \bigg).$$
\end{mylem}


Next, another expedient and well-known lemma is stated, which concerns the geometric convergence of sequences. For a proof, the reader is referred to~\cite[Chapter~I, Lemma~4.1]{dibenedetto1993degenerate}.

\begin{mylem} \label{geometriclem}
Let~$(Y_i)_{i\in\N_0}\subset \R_{\geq 0}$ be a sequence of non-negative numbers, satisfying the recursive inequality
$$Y_{i+1} \leq C b^i Y^{1+\kappa}_i$$
for any~$i\in\N_0$, where~$C\geq 0,\kappa >0$, and~$b>1$ denote positive constants. If there holds
$$Y_0 \leq C^{-\frac{1}{\kappa}}b^{-\frac{1}{\kappa^2}},$$
    then we have~$Y_i \to 0$ as~$i\to\infty$.
\end{mylem}


We recall the following parabolic Poincar\'{e}-type inequality, which is inferred from~\cite[Chapter~I, Corollary~3.1]{dibenedetto1993degenerate}.

\begin{mylem} \label{poincarelem}
    Let~$\Omega_T=\Omega\times(0,T)\subset\R^{n+1}$ be a bounded space-time cylinder and
    $$ u\in L^\infty\big(0,T;L^2(\Omega)\big) \cap L^2\big(0,T; W^{1,2}_0(\Omega)\big). $$
    Then, there exists a positive constant~$C=C(n)$, such that
    $$\iint_{\Omega_T}u^2\,\dx\dt \leq C |\{(x,t)\in \Omega_T: |u(x,t)|>0\}|^{\frac{2}{n+2}}\bigg( \esssup\limits_{t\in(0,T)}\int_{\Omega\times\{t\}}u^2\,\dx + \iint_{\Omega_T}|Du|^2\,\dx\dt \bigg). $$
\end{mylem}


The next lemma states a parabolic version of the classical elliptic Sobolev embedding and is taken from~\cite[Chapter I, Proposition 3.1]{dibenedetto1993degenerate}.

\begin{mylem} \label{sobolevembedding}
    Let~$\Omega_T=\Omega\times(0,T)\subset\R^{n+1}$ be a bounded space-time cylinder and
    $$u\in L^\infty(0,T;L^2(\Omega))\cap L^2(0,T;W^{1,2}_0(\Omega)).$$
    Then, there holds
    $$\iint_{\Omega_T}|u|^l\,\dx\dt \leq C\iint_{\Omega_T}|Du|^2\,\dx\dt \, \bigg(\esssup\limits_{t\in(0,T)}\int_{\Omega\times\{t\}}|u|^2\,\dx \bigg)^{\frac{2}{n}},$$
    where~$l\coloneqq \frac{2(n+2)}{n}$ and~$C=C(n)$.
\end{mylem}


Finally, we recall a sort of discrete isoperimetric inequality that stems from~\cite[Chapter 10.5.1, (5.4)]{DiBenedetto2009}.

\begin{mylem} \label{isoperimetric}
Let~$l,m\in\R$ with~$l<m$ and~$u\in W^{1,1}(B_R(x_0))$ for some~$B_R(x_0)\subset\R^n$. Then, there holds
$$(m-l)|B_R(x_0)\cap\{u<l\}| |B_R(x_0)\cap\{u>m\}| \leq C R^{n+1}  \int_{B_R(x_0)\cap\{l<u<m\}}|Du|\,\dx$$
with constant~$C=C(n)$.
\end{mylem}


\section{Proof of Theorem~\ref{hauptresultat}} \label{sec:holder}
In this section, we aim to give the proof of the main regularity result, that is Theorem~\ref{hauptresultat}. This objective will be achieved by performing multiple intermediate steps. 


\subsection{Regularizing the equation} \label{sec:regularizing}
As performed in a similar fashion for the elliptic setting in~\cite[Proof of Theorem~1.1]{colombo2017regularity} and~\cite[Section~3.1]{elliptisch}, we commence by introducing a regularizing procedure of our equation~\eqref{pde}, such that the redefined function~$\Tilde{\F}(x,t,\xi)$ satisfies quadratic growth with respect to the gradient variable~$\xi$, for any~$(x,t)\in\Omega_T$. Moreover, a further modification leads to a redefined function~$\hat{\F}_\epsilon$, where~$\epsilon\in(0,1]$, that not only exhibits quadratic growth but is additionally elliptic on the whole of~$\R^n$, with an ellipticity constant depending on the parameter~$\epsilon \in (0,1]$ in general. We consider a cylinder~$Q_R=Q_R(y_0,s_0)\Subset\Omega_T$ and let
$$u\in L^\infty_{\loc}(0,T;W^{1,\infty}_{\loc}(\Omega))$$
denote an arbitrary weak solution to~\eqref{pde}. In what follows, we will always omit the vertex~$(y_0,s_0)$ of the cylinder~$Q_R=Q_R(y_0,s_0)$ in any given quantity. For the further course of this section, we denote
\begin{equation} \label{schrankeu}
    K \coloneqq \|Du\|_{L^\infty(Q_R)} <\infty.
\end{equation}
 Moreover, we set
\begin{equation} \label{schrankeF}
    \Tilde{K} \coloneqq \|\F\|_{L^\infty(Q_R \times (B_{K+2R_E}))}<\infty, \qquad L\coloneqq \Tilde{K} + 1 <\infty, 
\end{equation}
where~$R_E>0$ is given in~\eqref{mengeeradii}. The additional constant quantity~$1$ is chosen arbitrarily and could be replaced with any positive real number. Next, let~$\Psi\in C^2(\R_{> 0},\R_{> 0})$ denote a smooth cut-off function that satisfies
\begin{equation*}
    \Psi(s)=s \qquad \mbox{for~$s \in[0,\Tilde{K}]$}, \qquad \Psi(s)=L \qquad \mbox{for~$s\in[L,\infty)$}
\end{equation*}
as well as
\begin{equation*}
    |\psi'(s)|+|\psi''(s)| \leq C_\Psi \qquad \mbox{for all~$s\in\R_{>0}$},
\end{equation*}
where~$C_\Psi>0$. For any~$(x,t)\in Q_R$ and~$\xi\in\R^n$, we consider the redefined function
\begin{equation} \label{tildef}
    \Tilde{\F}(x,t,\xi) \coloneqq \Psi(\F(x,t,\xi)).
\end{equation}
Due to this construction, there holds
\begin{equation} \label{tildefeeval}
    \F(x,t,Du(x,t))=\Psi(\F(x,t,Du(x,t))) = \Tilde{\F}(x,t,Du(x,t))
\end{equation}
for any~$(x,t) \in Q_R$ and~$\xi\in\R^n$, and, moreover, we have
\begin{align} \label{approxderivative}
    \nabla \Tilde{\F}(x,t,\xi) &=  \nabla\F(x,t,\xi) \Psi'(\F(x,t,\xi)),\quad \xi\in\R^n, \\
    \nabla^2\Tilde{\F}(x,t,\xi) &=  (\nabla\F(x,t,\xi)\otimes\nabla\F(x,t,\xi)) \Psi''(\F(x,t,\xi)) +  \nabla^2\F(x,t,\xi) \Psi'(\F(x,t,\xi)), \quad \xi\in\R^n\setminus \overline{E} \nonumber
\end{align}
for any~$(x,t) \in Q_R$. Since~$\nabla^2\F(x,t,\xi)$ is positive definite for any~$\xi\in \R^n\setminus \overline{E}$ and any~$(x,t) \in Q_R$, the mapping~$\xi \mapsto \F(x,t,\xi)$ on~$\R^n\setminus \overline{E}$ is strictly convex for any~$(x,t)\in Q_R$. Therefore,~$\xi \mapsto \F(x,t,\xi)$ is coercive on~$\R^n$ for any~$(x,t)\in Q_R$, such that for any~$(x,t)\in Q_R$ the truncated function~$\xi\mapsto\Tilde{\F}(x,t,\xi)$ remains constant on the set~$Q_R\times (\R^n \setminus B_{N(x,t)} )$ for some
$$ N(x,t) \geq K+2R_E. $$
We define
\begin{align} \label{N}
  N\coloneqq\sup\limits_{(x,t)\in Q_R}N(x,t) \geq K+2R_E  
\end{align}
and observe that~$N<\infty$ due to assumption~\eqref{voraussetzung} that holds for any~$(x,t)\in Q_R$. Thus, according to the preceding reasoning,~$\Tilde{\F}(x,t,\xi)$ is constant for any~$(x,t,\xi)\in Q_R\times (\R^n\setminus B_N)$. In particular, for the truncated function~$\Tilde{\F}$ there hold the following estimates
\begin{align} \label{derivativebounds}
    \|\nabla \Tilde{\F}\|_{L^\infty(Q_R \times \R_{N})} &\leq  
     C_\Psi \|\nabla \F\|_{L^\infty(Q_R \times B_{N})}, \\
    \|\nabla^2 \Tilde{\F}\|_{L^\infty(Q_R \times (\R_{N}\setminus E_\delta))} &\leq C_\Psi \|\nabla \F\|^2_{L^\infty(Q_R \times B_{N}) } + C_\Psi \|\nabla^2 \F\|_{L^\infty(Q_R \times (B_{N}\setminus E_\delta))} \nonumber
\end{align}
for any~$\delta>0$. Here,~$E_\delta$ denotes the outer parallel set of~$E$ as defined in~\eqref{Edelta}. Therefore, the preceding~$L^\infty$-estimates for both~$\nabla\Tilde{\F}$ and~$\nabla^2\Tilde{\F}$ only depend on the constant~$C_\Psi$ and the~$L^\infty$-bounds for~$\nabla\F$ and~$\nabla^2\F$ respectively. Next, we set 
\begin{align} \label{schrankehessian}
    C_{\F} \coloneqq \max\{ \|\nabla^2 \Tilde{\F}\|_{L^\infty(Q_R \times (\R_N \setminus B_{K+R_E}))}, 1 \} < \infty
\end{align}
and consider a convex function~$\Phi\in C^2(\R^n)$ that satisfies
\begin{align} \label{convexfunction}
       \begin{cases}
       \Phi(\xi)=0  & \mbox{for any~$|\xi| \leq K+R_E$}, \\
       |\nabla\Phi(\xi)| \leq (2C_{\F} + 1)|\xi|  & \mbox{for any~$\xi\in \R^n$}, \\
   \langle \nabla^2 \Phi (\xi) \eta,\eta \rangle \geq (C_{\F}+1) |\eta|^2 & \mbox{for any~$|\xi|\geq K+2 R_E$}, \\
        |\nabla^2\Phi(\xi)| \leq 2C_\F+1 & \mbox{for any~$\xi \in \R^n$} 
        \end{cases}
  \end{align}
 for any~$\eta\in\R^n$. Moreover, we set
\begin{align*}
    \hat{\F}(x,t,\xi) \coloneqq \Tilde{\F}(x,t,\xi) + \Phi(\xi) \qquad \mbox{for~$(x,t)\in Q_R,\, \xi\in\R^n$}
\end{align*}
and note that~$\R^n\ni\xi \mapsto \hat{\F}(x,t,\xi)$ is convex across~$\R^n$ due to the preceding construction of~$\Tilde{\F}$ and~$\Phi$. In fact, on the set~$Q_R\times B_{K+2R_E}$ we have~$\Tilde{\F}(x,t,\xi)=\F(x,t,\xi)$, such that, since~$\Phi$ is convex,~$\hat{\F}$ is convex across~$\R^n$. On the complementary set
$$ Q_R\times (\R^n\setminus B_{K+2R_E}), $$
by exploiting~\eqref{schrankehessian} as well as~\eqref{convexfunction}, there holds the ellipticity estimate
\begin{align} \label{hatapproxellipticity}
    \langle \nabla^2 \hat{\F}(x,t,\xi) \eta,\eta \rangle &= \langle \nabla^2 \Tilde{\F}(x,t,\xi) \eta,\eta \rangle + \langle \nabla^2 \Phi(\xi) \eta,\eta \rangle \geq |\xi|^2 
\end{align}
for any~$\eta\in\R^n$. For the Hessian~$\nabla^2\hat{\F}$, we have the estimate
\begin{align} \label{hatapproxbound}
    |\nabla^2\hat{\F}(x,t,\xi)| \leq 3C_{\F}+1 \qquad \mbox{for~$(x,t)\in Q_R,\, \xi \in \R^n\setminus B_{K+R_E} $}.
\end{align}
As in~\cite[Section~3.1]{elliptisch}, we further infer that the redefined mapping~$\hat{\F}$ satisfies the set of structure conditions~\eqref{fregularity}, where condition~$\eqref{fregularity}_4$ holds true with a quantitative bound
\begin{align} \label{hatflipschitz}
    |\partial_{x_i}\nabla\hat\F(x,\xi)| &\leq C(C_\Psi,N,R_E) + C_\Psi \underbrace{\| \partial_x \F\|_{L^\infty(Q_R\times B_{N})}}_{\leq C(C_\Psi,N,r_E)} \underbrace{\| \nabla \F\|_{L^\infty(Q_R\times B_{N})}}_{\eqqcolon \hat{C}_\F} \\
    &\leq C(\hat{C}_\F,C_\Psi,N,r_E) \nonumber
\end{align}
for any~$i=1,\ldots,n$, any~$(x,t)\in Q_R$, and~$\xi\in\R^n$. In turn, we used the mean value theorem to estimate the quantity involving~$\partial_x\F$ further above. In particular, the structure constant~$C=C(\hat{C}_\F,C_\Psi,N,r_E)$ does not depend on the size of~$|\xi|_E$. Additionally, the ellipticity condition~\eqref{voraussetzung} for~$\hat{\F}$ holds true. For any~$\delta>0$ we find positive constants~$\hat{\lambda}=\hat{\lambda}(C_\F,\delta)$ and~$\hat{\Lambda}=\hat{\Lambda}(C_\F,\delta)$ with~$0<\hat{\lambda}\leq\hat{\Lambda}$, such that
\begin{align} \label{hatfvoraussetzung}
     \qquad \hat{\lambda}(C_\F,\delta) |\eta|^2 \leq \langle \nabla^2 \hat{\F}(x,t,\xi) \eta,\eta \rangle \leq \hat{\Lambda}(C_\F,\delta) |\eta|^2\qquad\text{for any~$1+\delta \leq |\xi|_E \leq \delta^{-1}$}
\end{align}
for any~$(x,t)\in Q_R$. Since~$u$ is a weak solution to equation~\eqref{pde},~$u$ is in particular a weak solution to
\begin{equation*}
    \partial_t u - \divv \nabla\hat{\F}(x,t,Du) = f  \qquad\mbox{in~$Q_R$}.
\end{equation*}
This is a direct consequence of~\eqref{schrankeF} as we have that~$\hat\F=\F$ in~$Q_R\times B_K$. However,~$\hat{\F}$ does not admit uniform ellipticity on the whole of~$\R^n$, necessitating an additional approximation in order to obtain a truly elliptic equation on~$\R^n$. For an approximating parameter~$\epsilon\in[0,1]$, we set
\begin{equation} \label{psifapprox}
    \hat{\F}_\epsilon(x,t,\xi) \coloneqq \hat{\F}(x,t,\xi) + \epsilon \textstyle{\frac{1}{2}} |\xi|^2 \qquad\mbox{for~$(x,t)\in Q_R,\,\xi\in\R^n$}
\end{equation}
and note that~$\hat{\F}_\epsilon$ is indeed elliptic with respect to~$\xi$ on the whole of~$\R^n$ in the case where~$\epsilon\in(0,1]$, with the ellipticity constant depending on the parameter~$\epsilon\in(0,1]$ in general. Here, the quantity~$\frac{1}{2}|\xi|^2$ denotes the convex function of the well-known~$2$-energy. Consequently, we set
\begin{equation} \label{Gapprox}
    \hat{\h}_\epsilon(x,t,\xi)\coloneqq \nabla \hat{\F}_\epsilon(x,t,\xi) = \big( \hat{\h} (x,t,\xi)+\epsilon\xi\big)\in\R^n\qquad\mbox{for any~$(x,t)\in\Omega_T$,~$\xi\in\R^n$}
\end{equation}
as well as
\begin{equation*} 
    \nabla^2\hat{\F}_\epsilon(x,t,\xi)=\big(\nabla^2 \hat{\F} (x,t,\xi)+\epsilon I_n \big) \in \R^{n\times n}\qquad\mbox{for any~$(x,t)\in\Omega_T$,~$\xi\in\R^n\setminus \overline{E}$}.
\end{equation*}
In the spirit of~\eqref{bilinear}, we set
\begin{equation} \label{bilinearapprox}
            \hat{\mathcal{B}}_\epsilon(x,t,\xi)(\eta,\zeta)\coloneqq  \langle \nabla^2 \hat{\F}(x,t,\xi)\eta,\zeta \rangle +\epsilon \langle \eta, \zeta\rangle  \qquad\mbox{for any~$(x,t)\in \Omega_T$,~$\xi\in\R^n\setminus \overline{E}$}
        \end{equation}
for any~$\eta,\zeta\in\R^n$.
\,\\

Additionally, we have that~$\hat{\F}_\epsilon$ satisfies quadratic growth, i.e. there holds
\begin{align} \label{quadraticgrowth}
    |\nabla \hat{\F}_\epsilon(x,t,\xi)| \leq (\epsilon+C_\Psi\|\nabla \F \|_{L^\infty(Q_R \times B_N)}) (1+|\xi|) \leq C(C_\F,C_\Psi,\hat{C}_\F)(1+|\xi|)
\end{align}
for any~$(x,t)\in Q_R$ and~$\xi\in\R^n$, resp.
\begin{align} \label{quadraticgrowthohnekonst}
    |\nabla \hat{\F}_\epsilon(x,\xi)| \leq (2C_\F+1) |\xi|. 
\end{align}
for any~$(x,t)\in Q_R$ and~$|\xi|>K+2 R_E$. The latter is a consequence of the definition of~$\hat{\F}$. In order to simplify notation, we shall omit the constant~$C_\Psi$ in any quantity depending on~$C_\Psi$ from this point on. \,\\

We will now state certain structure estimates that will turn out essential for the subsequent discussion. \,\\
The following lemma serves as a crucial preliminary tool, providing a monotonicity estimate for~$\hat{\h}_\epsilon$. This estimate is a direct consequence of the ellipticity condition~\eqref{hatfvoraussetzung}, which holds for any~$1+\delta \leq |\xi|_E$ and for all~$(x,t) \in \Omega_T$.


\begin{mylem} \label{monotonicityapprox}
   Let~$\epsilon\in[0,1]$ and~$\delta>0$. For any~$(x,t)\in\Omega$ and any~$\Tilde{\xi},\xi\in\R^n$ with~$|\xi|_E \geq 1+\delta $ there holds
   \begin{equation} \label{est:monotonicityapprox}
       \langle \hat{\h}_\epsilon(x,t,\Tilde{\xi})-\hat{\h}_\epsilon(x,t,\xi),\Tilde{\xi}-\xi\rangle \geq \bigg(\epsilon +  C(C_\F,\delta)\frac{r_E(2|\xi|_E -(2+\delta))}{2|\xi|_E(R_E+r_E)} \bigg)|\Tilde{\xi}-\xi|^2. 
   \end{equation}
      Moreover, for any~$\tildexi,\xi\in\R^n$ the quantity on the left-hand side of~\eqref{est:monotonicityapprox} is non-negative.
\end{mylem}


As a consequence of Lemma~\ref{gdeltalem} and Lemma~\ref{monotonicityapprox}, we obtain the next auxiliary result.

\begin{mylem} \label{lemgdeltakvgz}
    Let~$\epsilon\in[0,1]$ and~$\delta>0$. For any~$(x,t)\in\Omega$ and~$\Tilde{\xi}, \xi \in\R^n$ there holds
    \begin{align*}
        \epsilon|\Tilde{\xi}-\xi|^2 + |\G_\delta(\tildexi) - \G_\delta(\xi)|^2 \leq  C(C_\F,\delta,R_E,r_E) \langle \hat{\h}_\epsilon(x,\tildexi)- \hat{\h}_\epsilon(x,\xi), \tildexi-\xi\rangle.
    \end{align*}
\end{mylem}


\begin{mylem} \label{bilinearelliptic}
    Let~$\epsilon\in[0,1]$ and~$\delta>0$. For any~$(x,t)\in \Omega_T$ and any~$\xi,\eta\in\R^n$ with~$1+\delta \leq |\xi|\leq \delta^{-1} $ there holds
    \begin{equation*}
        (\epsilon+\lambda(C_\F,\delta))|\eta|^2 \leq \hat{\mathcal{B}}_\epsilon(x,t,\xi)(\eta,\eta) \leq (\epsilon + \Lambda(C_\F,\delta) )|\eta|^2.
    \end{equation*}
\end{mylem}
\begin{proof}
    The estimate is an immediate consequence of the ellipticity assumption~\eqref{hatfvoraussetzung} and the definition of the bilinear form~$\hat{\B}_\epsilon$ according to~\eqref{bilinearapprox}. 
\end{proof}


\begin{remark} \upshape
 The estimates from Lemma~\ref{monotonicityapprox} and Lemma~\ref{bilinearelliptic} also to hold true in the case where~$|\xi|_E\leq 1$ without the constants~$C(C_\F,\delta)$,~$\lambda(C_\F,\delta)$,~$\Lambda(C_\F,\delta)$.    
\end{remark}


Next, let us consider the unique weak solution
$$ u_\epsilon \in C^0\big(\Lambda_R(s_0),L^{2}(B_R(y_0))\big) \cap  L^2\big(\Lambda_R(s_0),u+W^{1,2}_0(B_R(y_0))\big)$$
to the Cauchy-Dirichlet problem
\begin{align} \label{approx}
    \begin{cases}
       \partial_t u_\epsilon - \divv \hat{\h}_\epsilon(x,t,Du_\epsilon) = f & \mbox{in~$Q_R$}, \\
        \,\,\,\,\,\,\,\, \quad\qquad\qquad\quad\qquad u_\epsilon =u & \mbox{on~$\partial_p Q_R$}.
    \end{cases}
\end{align}
 Since the redefined equation satisfies quadratic growth according to~\eqref{quadraticgrowth}, it is indeed sufficient to consider approximating solutions~$u_\epsilon$ thar are~\textit{a priori} assumed to be of class
$$ C^0\big(\Lambda_R(s_0),L^{2}(B_R(y_0))\big) \cap L^2\big(\Lambda_R(s_0),u+W^{1,2}_0(B_R(y_0))\big).$$
The respective weak formulation of~\eqref{approx} results in
\begin{equation} \label{weakformapprox}
    \iint_{Q_R} (-u_\epsilon \partial_t\phi  + \langle \hat{\h}_\epsilon(x,t,Du_\epsilon), D\phi\rangle )  \,\dx\dt=  \iint_{Q_R} f \phi\,\dx\dt
\end{equation}
for any test function~$\phi\in C^{\infty}_0(Q_R)$. The ellipticity property of~$\hat{\F}_\epsilon$ enables us to obtain higher regularity of the form
$$ u_\epsilon \in L^{\infty}_{\loc}\big(\Lambda_R(s_0);W^{1,\infty}_{\loc}(B_R(y_0)) \big) $$
for the approximating solutions~$u_\epsilon$. Additionally, we obtain local quantitative estimates for~$Du_\epsilon$ in~$Q_R$, both of which are stated in form of Proposition~\ref{approxregularityeins} and Proposition~\ref{approxregularityzwei} respectively. In fact, we immediately infer from~\cite[Chapter~V, Theorem~3.3 \& Remark~3.1]{dibenedetto1993degenerate} that~$u_\epsilon$ is bounded in~$Q_R$ for any~$\epsilon\in(0,1]$, i.e. there holds~$u_\epsilon\in L^\infty(Q_R)$, which is a consequence of the quadratic growth~\eqref{quadraticgrowth} and the fact that the boundary datum~$u\in L^\infty(Q_R)$ itself is bounded within~$Q_R$. Indeed, there holds
\begin{align} \label{uepsbeschränkt}
     \|u_\epsilon\|_{L^\infty(Q_R)} \leq \|u\|_{L^\infty(Q_R)}<\infty
\end{align}
for any~$\epsilon\in(0,1]$.
\,\\

Before stating quantitative local~$L^{\infty}_{\loc}(Q_R,\R^n) \cap L^2_{\loc}\big(\Lambda_R(s_0);W^{1,2}_{\loc}(B_R(y_0),\R^n) \big)$-gradient estimates for~$Du_\epsilon$ on~$Q_{R}$, the subsequent energy estimate turns out expedient, which states a uniform~$L^2$-gradient bound with respect to the parameter~$\epsilon\in(0,1]$ for the approximating solutions~$u_\epsilon$.


\begin{mylem} \label{energybound}
    Let~$\epsilon\in(0,1]$. There holds the uniform energy estimate
    \begin{align*}
        \iint_{Q_R} |Du_\epsilon|^2\,\dx &\leq C \iint_{Q_R} \big( 1+|Du|^2 \big)\,\dx
    \end{align*}
    with a constant $C=C(C_\F,\hat{C}_\F,R_E,r_E)$.
\end{mylem}
\begin{proof}
We subtract the weak forms in Steklov-means of~\eqref{weakform} for~$u$ from the respective weak formulation~\eqref{weakformapprox} in Steklov-means for~$u_\epsilon$. This yields
\begin{equation*} 
     \int_{B_R} \big(\partial_t[u-u_\epsilon]_h \phi  + \langle [\hat{\h}(x,t,Du) -  \hat{\h}_\epsilon(x,t,Du_\epsilon)]_h, D\phi\rangle \big)  \,\dx\dt= 0
\end{equation*}
 for any~$\phi\in C^\infty_0(B_R)$. and any~$t\in\Lambda_R(s_0)$. Subsequently, we integrate the latter with respect to~$t\in\Lambda_R(s_0)$ to obtain
\begin{equation} \label{energyboundsteklov}
     \iint_{Q_R} \big(\partial_t(u-u_\epsilon) \phi  + \langle \hat{\h}(x,t,Du) -  \hat{\h}_\epsilon(x,t,Du_\epsilon), D\phi\rangle \big)  \,\dx\dt= 0
\end{equation}
 for any~$\phi\in C^\infty_0(B_R)$. Now, we test the preceding weak formulation~\eqref{energyboundsteklov} with~$\phi = [u-u_\epsilon]_h$. Note that~$\phi$ is indeed an admissible test function due to the definition of~$u_\epsilon$ and through a standard approximation argument. For the term involving the time derivative, we have
 \begin{align*}
     \iint_{Q_R} \partial_t[u-u_\epsilon]_h \phi\,\dx\dt &= \frac{1}{2} \iint_{Q_R } \partial_t[u-u_\epsilon]^2_h \,\dx\dt \\
     & = \frac{1}{2}  \iint_{B_R \times\{s_0\}} [u-u_\epsilon]^2_h \,\dx\dt.
 \end{align*}
Passing first to the limit~$h\downarrow 0$ yields
\begin{align*}
    \lim\limits_{h\downarrow 0}  \iint_{Q_R} \partial_t[u-u_\epsilon]_h \phi\,\dx\dt = \frac{1}{2}  \iint_{B_R \times\{s_0\}} |u-u_\epsilon|^2\,\dx.
\end{align*}
Letting also~$h\downarrow 0$ in the remaining diffusion term and adding the quantity~$\epsilon \langle Du,Du-Du_\epsilon\rangle$ on both sides of the inequality, we end up with
\begin{align} \label{allgemeinerenergyestimate}
    \frac{1}{2}  \int_{B_R \times\{s_0\}} |u-u_\epsilon|^2\,\dx &+ \iint_{Q_R} \langle \hat{\h}_\epsilon(x,t,Du) -  \hat{\h}_\epsilon(x,t,Du_\epsilon), Du-Du_\epsilon\rangle  \,\dx\dt \\
    &= \epsilon \iint_{Q_R} \langle Du,Du-Du_\epsilon \rangle \,\dx\dt. \nonumber
\end{align}
Since the first term leads to a non-negative contribution, we discard it. Then, there remains
\begin{align} \label{energyterme}
    \iint_{Q_R} \langle \hat{\h}(x,t,Du), Du \rangle\,\dx &+ \iint_{Q_R} \langle \hat{\h}_\epsilon(x,t,Du_\epsilon),Du_\epsilon \rangle\,\dx\dt \\
    &\leq \iint_{Q_R} \big( \langle \hat{\h}_\epsilon(x,t,Du_\epsilon),Du \rangle + \langle \hat{\h}(x,t,Du),Du_\epsilon \rangle \big)\,\dx\dt. \nonumber
\end{align}
The first quantity on the left-hand side of the preceding estimate is again non-negative due to Lemma~\ref{monotonicityapprox} and is discarded. We recall that~$\hat{\F}$ satisfies~\eqref{hatfvoraussetzung}. The second term on the left-hand side is also non-negative for the same reason, allowing us to pass to the subset of points~$Q_R\cap\{|Du_\epsilon|_E\geq \frac{3}{2}\}$ and bound the term below by employing Lemma~\ref{monotonicityapprox} to obtain
\begin{align*}
    \iint_{Q_R} \langle \hat{\h}_\epsilon(x,t,Du_\epsilon),Du_\epsilon \rangle\,\dx\dt &\geq \iint_{Q_R} \bigg( \hat{C} \frac{2|Du_\epsilon|_E-\frac{5}{2}}{|Du_\epsilon|_E} |Du_\epsilon|^2 \bigchi_{ \{ |Du_\epsilon|_E \geq \frac{3}{2} \} } +\epsilon|Du_\epsilon|^2 \bigg) \,\dx\dt \\
    &\geq \iint_{Q_R} \bigg(\hat{C} \frac{2|Du_\epsilon|_E-\frac{5}{2}}{|Du_\epsilon|_E} (|Du_\epsilon|_E-\mbox{$\frac{3}{2}$})^2_+  +\epsilon|Du_\epsilon|^2 \bigg) \,\dx\dt \\
    &\geq   \iint_{Q_R}  \hat{C}(|Du_\epsilon|_E-\mbox{$\frac{3}{2}$})^2_+\,\dx\dt + \epsilon \iint_{Q_R} |Du_\epsilon|^2 \,\dx\dt,
\end{align*}
with a positive constant~$\hat{C}>0$ that depends on~$C_\F,R_E,r_E$. In turn, we also used~\eqref{betragminkowski}. Next, we bound the first term contained in the integral on the right-hand side of~\eqref{energyterme} further above by exploiting the quadratic growth of~$\hat{\h}_\epsilon$ according to~\eqref{quadraticgrowth} and also Young's inequality, which yields
\begin{align*}
    \iint_{Q_R} \langle \hat{\h}_\epsilon(x,t,Du_\epsilon),Du \rangle \,\dx\dt &\leq  \iint_{Q_R} \big( C(1+|Du_\epsilon|)|Du| + \epsilon|Du_\epsilon| |Du|\big) \,\dx\dt \\
    &\leq  \iint_{Q_R} \bigg( \frac{\hat{C}}{4} (|Du_\epsilon|_E-\mbox{$\frac{3}{2}$})^2_+\, +  C|Du|^2 + C\bigg)\,\dx\dt \\
    &\quad + \frac{\epsilon}{2} \iint_{Q_R} \big(|Du_\epsilon|^2 + |Du|^2 \big)\,\dx\dt
\end{align*}
with a constant~$C=C( C_\F,\hat{C}_\F,R_E,r_E)$. The remaining second integral term on the right-hand side of~\eqref{energyterme} is estimated in a similar way. After reabsorbing all quantities involving~$(|Du_\epsilon|_E-\mbox{$\frac{3}{2}$})^2$ into the left-hand side of~\eqref{energyterme}, we end up with
\begin{align*}
    \iint_{Q_R} \big( (|Du_\epsilon|_E-\mbox{$\frac{3}{2}$})^2_+ + \epsilon|Du_\epsilon|^2 \big)\,\dx\dt \leq \epsilon C \iint_{Q_R} |Du|^2\,\dx\dt + C  \iint_{Q_R} (1+|Du|^2)\,\dx\dt
\end{align*}
for some constant~$C=C(C_\F,\hat{C}_\F,R_E,r_E)$, from which the claimed energy bound can be readily inferred by another application of~\eqref{betragminkowski}. 
\end{proof}


As a next step, we aim to establish the higher regularity
$$u_\epsilon \in L^\infty_{\loc}\big(\Lambda_R(s_0);W^{1,\infty}_{\loc}(B_R(y_0)) \big) \cap L^2_{\loc}\big(\Lambda_R(s_0);W^{2,2}_{\loc}(B_R(y_0)) \big),$$
where we commence with the treatment of the~$L^2_{\loc}\big(\Lambda_R(s_0);W^{1,2}_{\loc}(B_R(y_0),\R^n) \big)$-gradient estimate.


\begin{myproposition} \label{approxregularityeins}
    Let~$\epsilon\in(0,1]$. Then, there holds
    \begin{equation*}
        Du_\epsilon \in L^2_{\loc}\big(\Lambda_R(s_0);W^{1,2}_{\loc}(B_R(y_0),\R^n) \big) \cap L^\infty_{\loc}\big(\Lambda_R(s_0);L^{2}_{\loc}(B_R(y_0),\R^n) \big).
    \end{equation*}
    Moreover, for any cylinder~$Q_{2\rho}(z_0)\Subset Q_R$ with~$\rho\in(0,1]$, there holds the quantitative estimate
     \begin{align} \label{duw12}
      & \esssup\limits_{t\in \Lambda_{\frac{\rho}{4}}(t_0)}\int_{ B_{\frac{\rho}{4}}(x_0)  \times\{t\}} |Du_\epsilon|^2 \,\dx + \iint_{Q_{\frac{\rho}{4}}(z_0)} |D^2u_\epsilon|^2\,\dx\dt \\
    &\quad \leq \frac{C}{\epsilon^2\rho^2} \Bigg(\iint_{Q_{\rho}(z_0)} (1+|Du_\epsilon|)^2\,\dx\dt + \rho^2 \iint_{Q_{\rho}(z_0)} |f|^2 \,\dx\dt \Bigg) \nonumber
    \end{align}
    with a constant~$C=C(C_\F,\hat{C}_\F,N,n,r_E)$. 
\end{myproposition}


\begin{proof}
We establish the quantitative~$L^2_{\loc}\big(\Lambda_R(s_0);W^{1,2}_{\loc}(B_R(y_0)\big) \cap L^\infty_{\loc}\big(\Lambda_R(s_0);L^{2}_{\loc}(B_R(y_0),\R^n) \big)$-gradient estimate~\eqref{duw12} in the usual way through employing the notion of difference quotients. Let us consider a cylinder~$Q_{2\rho}(z_0)\Subset Q_R$ and a parameter~$\sigma$ small enough, such that there holds~$|\sigma|\in(0,\dist(Q_{2\rho}(z_0),\partial Q_R))$ for any~$z\in Q_{2\rho}(z_0)$. The finite difference of~$u_\epsilon$ with respect to~$x$ in direction~$x_i$ and increment~$\sigma\neq 0$ is given by
$$\tau_{i,\sigma}u_\epsilon(x,t) \coloneqq u_\epsilon(x+\sigma e_i,t) - u_\epsilon(x,t)$$
for a.e.~$(x,t)\in Q_{2\rho}(z_0)$ and any~$i\in\{1,\ldots,n\}$. In order to simplify notation, we shall write~$u=u_\epsilon$ and~$\tau_\sigma u = \tau_{i,\sigma}u_\epsilon$. We test the weak form~\eqref{weakformapprox} with~$\tau_{-\sigma}\phi$ for an arbitrarily chosen test function~$\phi\in C^\infty_0(Q_{2\rho}(z_0))$ and perform a change of variables, leading to
\begin{equation*}
    \iint_{Q_R} (- \tau_\sigma u \,\partial_t\phi  + \langle \tau_\sigma \hat{\h}_\epsilon(x,t,Du), D\phi\rangle )  \,\dx\dt=  - \iint_{Q_R} f \tau_{-\sigma}\phi\,\dx\dt.
\end{equation*}
Integrating by parts in the evolutionary term, we obtain
\begin{equation*}
    \iint_{Q_R} (\partial_t (\tau_\sigma u) \phi  + \langle \tau_\sigma \hat{\h}_\epsilon(x,t,Du), D\phi\rangle )  \,\dx\dt=  \iint_{Q_R} f\, \tau_{-\sigma}\phi\,\dx\dt
\end{equation*}
for any~$\phi\in C^\infty_0(Q_R)$. Indeed, this calculation is only formal and can be made rigorous by a Steklov-average procedure similar to Lemma~\ref{energybound}. In the preceding equation, we choose~$\phi=\zeta \tau_\sigma u $, where~$\zeta=\psi^2(x,t)\eta(t) \in W^{1,\infty}_0(Q_{\frac{\rho}{2}}(z_0),[0,1])$ and~$\psi\in C^2(Q_{\frac{\rho}{2}(z_0)},[0,1])$ denotes a smooth cut-off function vanishing on the parabolic boundary~$\partial_p(Q_{\frac{\rho}{2}}(z_0))$ with~$\psi \equiv 1$ in~$Q_{\frac{\rho}{4}}(z_0)$,~$|D\psi| \leq \frac{8}{\rho}$,~$|D^2\psi|\leq \frac{64}{\rho^2}$, and~$|\partial_t\psi|\leq \frac{32}{\rho^2}$, while~$\eta\in W^{1,\infty}(\R,[0,1])$ is a Lipschitz continuous function in time subject to
\begin{align*}
    \eta(t) \coloneqq \begin{cases}
        1 & \mbox{for~$t\in (-\infty,\kappa]$}, \\
        1-\frac{t-\kappa}{\delta} & \mbox{for~$t\in (\kappa,\kappa+\delta]$}, \\
        0 & \mbox{for~$t\in(\kappa+\delta,\infty)$}
    \end{cases}
\end{align*}
for an arbitrarily chosen~$t_0-\frac{\rho^2}{4}<\kappa<\kappa+\delta<t_0$ with~$\delta>0$. This choice of test function~$\phi$ can be justified by an approximation argument. We begin with the treatment of the term involving the time derivative where we integrate by parts, which yields
\begin{align*}
     \iint_{Q_R} \partial_t (\tau_\sigma u) \phi \,\dx\dt &= \frac{1}{2} \iint_{Q_R} \partial_t (\tau_\sigma u)^2 \zeta \,\dx\dt \\
     &= - \frac{1}{2} \iint_{Q_R}  (\tau_\sigma u)^2 \partial_t\zeta \,\dx\dt \\
     &= \frac{1}{2} \frac{1}{\delta} \iint_{B_{R}\times(\kappa,\kappa+\delta]} (\tau_\sigma u)^2 \psi^2 \,\dx\dt \\
     & \quad - \iint_{Q_R} (\tau_\sigma u)^2 \psi \partial_t \psi \eta,\dx\dt.
\end{align*}
Passing to the limit~$\delta\downarrow 0$, we thus obtain for a.e.~$\kappa \in \Lambda_{\frac{\rho}{2}}(t_0)$ the following
\begin{align} \label{differencekappa}
   \int_{B_{R}\times\{\kappa\}} & (\tau_\sigma u)^2 \psi^2 \,\dx + \iint_{Q_\kappa}\langle \tau_\sigma \hat{\h}_\epsilon(x,t,Du), D[\psi^2 \tau_{\sigma} u] \rangle \,\dx\dt \\
   &= \iint_{Q_\kappa} (\tau_\sigma u)^2 \psi \partial_t \psi \,\dx\dt + \iint_{Q_\kappa} f\, \tau_{-\sigma}[\psi^2 \tau_\sigma u]\,\dx\dt, \nonumber
\end{align}
where we abbreviated~$Q_\kappa \coloneqq B_{R} \times (t_0-\frac{\rho^2}{4},\kappa]$. By exploiting the commutativity of finite differences with weak derivatives, the diffusion term yields two quantities
\begin{align*}
    \iint_{Q_\kappa}\langle \tau_\sigma \hat{\h}_\epsilon(x,t,Du), D[\psi^2 \tau_{\sigma} u] \rangle \,\dx\dt &= \iint_{Q_\kappa}\psi^2  \langle \tau_\sigma \hat{\h}_\epsilon(x,t,Du), \tau_{\sigma} Du \rangle \,\dx\dt \\
    &\quad + 2 \iint_{Q_\kappa}  \langle \tau_\sigma \hat{\h}_\epsilon(x,t,Du),  \tau_{\sigma} u \psi D\psi \rangle \,\dx\dt. 
\end{align*}
In the second term on the right-hand side of the preceding equality, we again integrate by parts through a change of variables and shift the finite difference from the vector field~$\hat{\h}_\epsilon$ onto the quantity~$\tau_\sigma u \psi  D\psi$. This way, we achieve
\begin{align} \label{weakdiffusiondifference}
    \iint_{Q_\kappa}\langle \tau_\sigma \hat{\h}_\epsilon(x,t,Du), D[\psi^2 \tau_{\sigma} u] \rangle \,\dx\dt &= \iint_{Q_\kappa}\psi^2  \langle \tau_\sigma \hat{\h}_\epsilon(x,t,Du), \tau_{\sigma} Du \rangle \,\dx\dt \\
    &\quad - 2 \iint_{Q_\kappa}  \langle \hat{\h}_\epsilon(x,t,Du),  \tau_{-\sigma} [\tau_{\sigma} u \psi D\psi] \rangle \,\dx\dt. \nonumber
\end{align}
The finite difference apparent in the first integral term on the right-hand side above can be rewritten as follows
\begin{align*}
    \tau_\sigma \hat{\h}_\epsilon(x,t,Du) &= \tau_\sigma\hat{\h}(x,t,Du) + \epsilon\tau_\sigma Du(x) \\
    &= \epsilon\tau_\sigma Du(x) + \hat{\h}(x+\sigma e_i,t,Du(x+\sigma e_i,t)) - \hat{\h}(x,t,Du(x+\sigma e_i,t)) \\
    &\quad + \hat{\h}(x,t,Du(x+\sigma e_i,t)) - \hat{\h}(x,t,Du(x,t)) \\
    &\eqqcolon \foo{I} + \foo{II} + \foo{III}
\end{align*}
for a.e.~$(x,t) \in Q_\kappa$. The first term~$\foo{I}$ and the third term~$\foo{III}$ remain unchanged. The second term~$\foo{II}$ is bounded above with aid of the Lipschitz assumption~$\eqref{fregularity}_4$, according to~\eqref{hatflipschitz}, as the latter implies that~$\hat{\h}$ remains constant whenever~$|\xi|\geq L$ for~$\xi\in\R^n$. Thus, also recalling~\eqref{N}, we estimate
$$\foo{II} \leq C(\hat{C}_\F,N,r_E)|\sigma|.$$ 
For the second integral quantity in~\eqref{weakdiffusiondifference}, we use a sort of product rule for finite differences and rewrite the latter as
\begin{align*}
    \tau_{-\sigma}[\tau_\sigma u \psi D\psi](x,t) &= \tau_\sigma u(x,t) \psi(x,t) D\psi(x,t) - \tau_\sigma u(x-\sigma e_i,t) \psi(x-\sigma e_i,t) D\psi(x-\sigma e_i,t) \\
    &= \tau_{-\sigma} \tau_\sigma u(x,t) \psi(x,t) D\psi(x,t) + \tau_{-\sigma}[\psi D\psi](x,t) \tau_\sigma u(x-\sigma e_i,t) \\
    &= \tau_{-\sigma} \tau_\sigma u(x,t) \psi(x,t) D\psi(x,t) - \tau_{-\sigma}[\psi D\psi](x,t) \tau_{-\sigma}u(x,t),
\end{align*}
which holds true for a.e.~$(x,t)\in Q_{2\rho}(z_0)$. This term is bounded above by applying the Cauchy-Schwarz inequality, while the term involving the datum~$f$ in~\eqref{differencekappa} is further estimated by taking its modulus. Combining our calculations, we thus obtain
\begin{align} \label{w22diff}
    &\int_{ B_{R} \times\{\kappa\}} (\tau_\sigma u)^2 \psi^2 \,\dx + \epsilon\int_{Q_\kappa} \psi^2 |\tau_\sigma Du|^2\,\dx\dt \\
    &\quad + \underbrace{\iint_{Q_\kappa} \psi^2 \langle \hat{\h}(x,t,Du(x+\sigma e_i,t))-\hat{\h}(x,t,Du(x,t)), Du(x+\sigma e_i,t)-Du(x,t) \rangle\,\dx\dt}_{\geq 0} \nonumber \\
    & \leq C |\sigma| \iint_{Q_\kappa} \psi^2 |\tau_\sigma Du|\,\dx\dt + \iint_{Q_\kappa} |f| |\tau_{-\sigma}[\psi^2 \tau_{\sigma} u]|\,\dx\dt \nonumber \\
    &\quad + C(\epsilon + 1) \iint_{Q_\kappa} (1+|Du|)|\tau_{-\sigma}[\psi D\psi]| |\tau_{-\sigma}u|\,\dx\dt \nonumber \\
    &\quad + C(\epsilon + 1) \iint_{Q_\kappa} (1+|Du|) \psi|D\psi| |\tau_{-\sigma}\tau_\sigma u| \,\dx\dt + \iint_{Q_\kappa} (\tau_\sigma u)^2 \psi \partial_t \psi,\dx\dt \nonumber \\
    &\eqqcolon \Tilde{\foo{I}} + \Tilde{\foo{II}} + \Tilde{\foo{III}} + \Tilde{\foo{IV}}+\Tilde{\foo{V}} \nonumber
\end{align}
with~$C=C(C_\F,\hat{C}_\F,N,r_E)$, for a.e.~$\kappa\in \Lambda_{\frac{\rho}{2}}(t_0)$. The third term on the left-hand side of the preceding estimate is non-negative due to the second assertion of Lemma~\ref{monotonicityapprox} and thus can be discarded. The quantities~$\Tilde{\foo{I}} - \Tilde{\foo{V}}$ are now treated one after another by bounding them further above with Young's inequality and by exploiting the bounds for~$\psi$. Moreover, we rely on a classical result concerning difference quotients, where we refer the reader to~\cite[Chapter~5.8, Theorem 3]{evans2022partial}. The first term~$\Tilde{\foo{I}}$ is estimated by
\begin{align*}
    \Tilde{\foo{I}} &\leq \frac{\epsilon}{4} \int_{Q_{\kappa}} \psi^2 |\tau_\sigma Du|^2 \,\dx + \frac{C}{\epsilon} |\sigma|^2 |Q_{\frac{\rho}{2}}(z_0)|.
\end{align*}
Next, the second term~ involving the datum~$f$ is bounded further above via Young's inequality, the bounds for~$\psi$ and~$D\psi$ and also that~$\psi$ is supported within~$Q_{\frac{\rho}{2}}(z_0)$, and the result from~\cite[Chapter~5.8, Theorem 3]{evans2022partial}, which yields 
\begin{align*}
    \Tilde{\foo{II}} &\leq \frac{2}{\epsilon}|\sigma|^2 \iint_{Q_{\rho}(z_0)} |f|^2\,\dx\dt + \frac{\epsilon}{8|\sigma|^2} \iint_{B_{\frac{\rho}{2}}(x_0)\times(t_0-\frac{\rho^2}{4},\kappa]} |\tau_{-\sigma}[\psi^2 \tau_\sigma u]|^2\,\dx\dt \\
    &\leq \frac{2}{\epsilon}|\sigma|^2 \iint_{Q_{\rho}(z_0)} |f|^2\,\dx\dt + \frac{\epsilon}{8} \iint_{B_{\frac{\rho}{2}}(x_0)\times(t_0-\frac{\rho^2}{4},\kappa]} |D[\psi^2 \tau_\sigma u]|^2\,\dx\dt \\
    &\leq \frac{2}{\epsilon}|\sigma|^2 \iint_{Q_{\rho}(z_0)} |f|^2\,\dx\dt + \frac{\epsilon}{8} \iint_{B_{\frac{\rho}{2}}(x_0)\times(t_0-\frac{\rho^2}{4},\kappa]} |\psi^2 \tau_\sigma Du + 2\psi D\psi \tau_\sigma u|^2\,\dx\dt \\
    &\leq \frac{2}{\epsilon}|\sigma|^2 \iint_{Q_{\rho}(z_0)} |f|^2\,\dx\dt + \frac{\epsilon}{4} \iint_{Q_{\kappa}} \psi^2 |\tau_\sigma Du|^2 \,\dx\dt + \frac{C\epsilon}{\rho^2}|\sigma|^2 \iint_{Q_{\rho}(z_0)} |Du|^2 \,\dx\dt
\end{align*}
with constant~$C=C(n)$. The third quantity~$\Tilde{\foo{III}}$ is estimated above similarly to before
\begin{align*}
    \Tilde{\foo{III}} &\leq C(\epsilon+1)\rho^2 \iint_{B_{\frac{\rho}{2}}(x_0)\times(t_0-\frac{\rho^2}{4},\kappa]} (1+|Du|)^2 |\tau_{-\sigma}[\psi D\psi]|^2\,\dx\dt \\
    &\quad + C\frac{(\epsilon+1)}{\rho^2} \iint_{B_{\frac{\rho}{2}}(x_0)\times(t_0-\frac{\rho^2}{4},\kappa]} |\tau_{-h}u(x)|^2\,\dx\dt \\
   &\leq C(\epsilon+1) \rho^2|\sigma|^2 \iint_{Q_{\rho}(z_0)} (1+|Du|)^2 |D[\psi D\psi]|^2\,\dx\dt + C\frac{(\epsilon+1)}{\rho^2}|\sigma|^2 \iint_{Q_{\rho}(z_0)} |Du|^2\,\dx\dt \\
   &\leq C(\epsilon+1)\rho^2|\sigma|^2 \iint_{Q_{\rho}(z_0)} (1+|Du|^2) (|D\psi\otimes D\psi|^2 + |D^2\psi|^2)\,\dx\dt \\
   &\quad + C\frac{(\epsilon+1)}{\rho^2}|\sigma|^2 \iint_{Q_{\rho}(z_0)} |Du|^2\,\dx\dt \\
   &\leq C\frac{(\epsilon+1)}{\rho^2}|\sigma|^2 \iint_{Q_{\rho}(z_0)} (1+|Du|)^2\,\dx\dt
\end{align*}
with a constant~$C=C(C_\F,\hat{C}_\F,N,n,r_E)$. For the fourth term~$\Tilde{\foo{IV}}$, we apply Hölder's and Young's inequality, and also exploit the fact that finite differences and weak derivatives commutate. This leads to
\begin{align*}
    \Tilde{\foo{IV}} &\leq C(\epsilon+1) \\
    &\quad\cdot \bigg( \iint_{B_{\frac{\rho}{2}}(x_0)\times(t_0-\frac{\rho^2}{4},\kappa]} (1+|Du|)^2 \psi^2 |D\psi|^2\,\dx\dt \bigg)^{\frac{1}{2}} \bigg( \iint_{B_{\frac{\rho}{2}}(x_0)\times(t_0-\frac{\rho^2}{4},\kappa]} |\tau_{-\sigma}\tau_\sigma u|^2 \,\dx\dt \bigg)^{\frac{1}{2}} \\
    &\leq C(\epsilon+1) |\sigma| \bigg( \iint_{Q_{\frac{\rho}{2}}(z_0)} (1+|Du|)^2 \psi^2 |D\psi|^2\,\dx\dt \bigg)^{\frac{1}{2}} \bigg( \iint_{B_{\rho}(x_0)\times(t_0-\frac{\rho^2}{4},\kappa]} |\tau_\sigma Du|^2 \,\dx\dt \bigg)^{\frac{1}{2}} \\
    &\leq \frac{C}{\rho^2} \frac{\epsilon+1}{\epsilon} |\sigma|^2 \iint_{Q_{\frac{\rho}{2}}(z_0)} (1+|Du|)^2 \,\dx + \frac{\epsilon}{4} \iint_{Q_{\kappa}} |\tau_\sigma Du|^2 \,\dx\dt
\end{align*}
Finally, the last term~$\Tilde{\foo{V}}$ is estimated by
\begin{align*}
    \Tilde{\foo{V}} &\leq \frac{C}{\rho^2} |\sigma|^2 \iint_{Q_{\rho}(z_0)} |Du|^2\,\dx\dt. 
\end{align*}
 At this point, we reabsorb all quantities involving the finite difference~$\tau_\sigma Du$ into the left-hand side of~\eqref{w22diff} and use the bounds~$\epsilon,\rho\in(0,1]$, to obtain
\begin{align*} 
    &\int_{ B_{R} \times\{\kappa\}} (\tau_\sigma u)^2 \psi^2 \,\dx + \int_{Q_\kappa} \psi^2 |\tau_\sigma Du|^2\,\dx\dt \\
    & \quad \leq \frac{C|\sigma|^2}{\epsilon^2\rho^2}\Bigg(\iint_{Q_{\rho}(z_0)} (1+|Du|)^2\,\dx\dt + \rho^2 \iint_{Q_{\rho}(z_0)} |f|^2\,\dx\dt \Bigg)
\end{align*} 
for a.e.~$\kappa \in \Lambda_{\frac{\rho}{2}}(t_0)$, with a constant~$C=C(C_\F,\hat{C}_\F,N,n,r_E)$. For the remaining two quantities on the left-hand side, we exploit the fact that~$\psi\equiv 1$ on~$B_{\frac{\rho}{4}(x_0)}$. Finally, we take the essential supremum with respect to~$\kappa\in \Lambda_{\frac{\rho}{2}}(t_0)$, which overall leads to
\begin{align*} 
    &\esssup\limits_{t\in \Lambda_{\frac{\rho}{4}}(t_0)}\int_{ B_{\frac{\rho}{4}}(x_0)  \times\{t\}} (\tau_\sigma u)^2 \,\dx + \iint_{Q_{\frac{\rho}{4}}(z_0)} |\tau_\sigma Du|^2\,\dx\dt \\
    & \quad \leq \frac{C|\sigma|^2}{\epsilon^2\rho^2}\Bigg(\iint_{Q_{\rho}(z_0)} (1+|Du|)^2\,\dx\dt + \rho^2 \iint_{Q_{\rho}(z_0)} |f|^2\,\dx\dt\Bigg)
\end{align*}
with~$C=C(C_\F,\hat{C}_\F,N,n,r_E)$. The term in brackets on the right-hand side is bounded uniformly with respect to~$\sigma$, allowing us to utilize another classical result on difference quotients, cf.~\cite[Chapter~5.8,~Theorem 3]{evans2022partial}. Since~$i\in\{1,\ldots,n\}$ was chosen arbitrarily, we obtain the claimed gradient estimate~\eqref{duw12}.
\end{proof}


Proposition~\ref{approxregularityeins} enables us to differentiate the weak form~\eqref{weakformapprox}, leading to the following local~$L^\infty$-gradient bound for~$Du_\epsilon$ in~$Q_R$, which builds upon parabolic De Giorgi classes and stems from~\cite[Chapter~12, Theorem~2.1]{dibenedetto2023parabolic}.


\begin{myproposition} \label{approxregularityzwei}
  Let~$\epsilon\in(0,1]$. Then, there holds
  $$ u_\epsilon\in L^\infty_{\loc}\big(\Lambda_R(s_0);W^{1,\infty}_{\loc}(B_R(y_0)) \big) \cap L^2_{\loc}\big(\Lambda_R(s_0);W^{2,2}_{\loc}(B_R(y_0)) \big) . $$
  Moreover, for any cylinder~$Q_{2\rho}(z_0)\Subset Q_R$ with~$\rho\in(0,1]$, there holds the quantitative local~$L^\infty$-gradient estimate
    \begin{equation} \label{dulinfty}
        \esssup\limits_{Q_{\rho}(z_0)} |Du_\epsilon| \leq C \Bigg( \bigg(\fiint_{Q_{2\rho}(z_0)}  |Du_\epsilon|^2  \,\dx\dt \bigg)^{\frac{1}{2}} +1  \Bigg)
    \end{equation} 
   with~$C=C(C_\F,\hat{C}_\F,\|f\|_{L^{n+2+\sigma}(Q_R)},N,n,R,R_E,r_E,\sigma)$. 
\end{myproposition}

\begin{proof}
According to Proposition~\ref{approxregularityeins}, second order weak spatial derivatives of~$u_\epsilon$ exist and there holds
$$ u_\epsilon\in L^2\big(\Lambda_{2\rho}(t_0);W^{2,2}(B_{2\rho}(x_0)) \big). $$
 Formally differentiating the weak form~\eqref{weakformapprox} for~$u_\epsilon$ by testing the latter with~$\partial_{e^*}\phi$, where~$e^*\in\partial E^*$ is taken arbitrary and~$\phi\in C^\infty_0(Q_{2\rho}(z_0)\cap\{|Du_\epsilon|_E\geq 1+\delta\})$, we obtain for~$u^{*}_\epsilon\coloneqq \partial_{e^*} u_\epsilon$ the identity
\begin{equation*} 
      \iint_{Q_{2\rho}(z_0)} \big(-u^{*}_\epsilon\, \partial_t \phi +  \langle \partial_{e^*}[\nabla \hat{\F}_\epsilon(x,t,Du_\epsilon)], D\phi\big) \,\dx\dt =  - \iint_{Q_{2\rho}(z_0)} f \partial_{e^*}\phi \,\dx\dt 
\end{equation*}
for any~$\phi\in C^\infty_0(Q_{2\rho}(z_0)\cap\{|Du_\epsilon|_E\geq 1+\delta\})$. Similarly to~\cite[Chapter~12, Section~1]{dibenedetto2023parabolic}, we test the preceding weak form with~$\phi=\zeta^2(\partial_{e^*}u_\epsilon-k)_{+}$, where~$\zeta \in C^\infty_0(Q_{2\rho}(z_0),[0,1])$ and~$k>0$. This choice of test function is admissible up to a Steklov-average procedure and an approximation argument. Importantly, we assume that
$$ k\geq \mbox{$\frac{K+R_E}{r_E}$}, $$ 
which implies~$|Du_\epsilon|> K+R_E$ whenever~$\phi$ is positive according to~\eqref{betragminkowski},
such that the equation is non-degenerate in this set of points. We note that the previous assumption made on~$k$ is not restrictive for the derivation of the claimed~$L^\infty$-gradient estimate~\eqref{dulinfty}, since the latter is anyway chosen large enough during the proof of~\cite[Chapter~12, Section~1]{dibenedetto2023parabolic} in dependence on the given data. Since the diffusion term in the weak form above is non-degenerate whenever~$\phi$ is positive, we are allowed to differentiate the diffusion term and accordingly obtain
\begin{align*} 
      \iint_{Q_{2\rho}(z_0)} &\big(- u^{*}_\epsilon\, \partial_t \phi +  \hat{\B}_\epsilon(x,t,Du_\epsilon)(D u^{*}_\epsilon,D\phi) \big) \,\dx\dt \\
      &=  - \iint_{Q_{2\rho}(z_0)} \bigg(f \sum\limits_{i=1}^n v_i D_i\phi + \sum\limits_{i,j=1}^n v_i g^i_j D_j \phi \bigg) \,\dx\dt, 
\end{align*}
for the preceding choice of~$\phi$. Here, the coefficients~$g^i=\colon Q_{2\rho}(z_0) \to \R^n$ for~$i=1,\ldots,n$ are given by
$$ |g^i(x,t)| \coloneqq |\partial_{x_i} \nabla\hat{\F}(x,t,Du_\epsilon)| \leq C(\hat{C}_\F,N,r_E),$$
which immediately follows from the estimate~\eqref{hatflipschitz}. In turn, we also used the representation formula~\eqref{representation}, where~$v_i$ for~$i=1,\ldots,n$ denote the coefficients of~$e^*\in\partial E^*$. Note that there holds~$|v_i|\leq R_E$ for any~$i=1,\ldots,n$. Due to~\eqref{minkowskialternativ} and~\eqref{betragminkowski}, we have
$$ |Du_\epsilon| \geq r_E |Du_\epsilon|_E \geq K+R_E \qquad\mbox{a.e. in~$Q_{2\rho}(z_0)$}. $$
Consequently, Lemma~\ref{bilinearelliptic} is applicable and we obtain that the bilinear form~$\mathcal{\hat{B}}_\epsilon(x,t,Du_\epsilon)$ is elliptic and bounded on
$$ Q_{2\rho}(z_0)\cap \{ u^{*}_\epsilon \geq \mbox{$\frac{K+R_E}{r_E}$}\}, $$
which follows from the set inclusion
$$ \{ (x,t)\in Q_{2\rho}(z_0): u^{*}_\epsilon \geq \mbox{$\frac{K+R_E}{r_E}$} \} \subset \{ (x,t)\in Q_{2\rho}(z_0): |Du_\epsilon|_E \geq \mbox{$\frac{K+R_E}{r_E}$} \}. $$
Moreover, the ellipticity constant is independent of the parameter~$\epsilon\in(0,1]$ and only depends on the data~$C_\F$, that is: there exists a positive constant~$C=C(C_\F)\geq 1$, such that
\begin{align*} 
  C^{-1}|\xi|^2 \leq \hat{\mathcal{B}}_\epsilon(x,t,Du_\epsilon)(\xi,\xi) \leq C|\xi|^2  
\end{align*}
holds true for a.e.~$(x,t)\in Q_{2\rho}(z_0)\cap \{ u^{*}_\epsilon \geq \mbox{$\frac{K+R_E}{r_E}$}\}$ and any~$\xi\in\R^n$. Due to the quadratic growth property~\eqref{quadraticgrowth} and the fact that~$f\in L^{n+2+\sigma}(\Omega_T)$, the conditions~~\cite[Chapter~12, Section~1, (1.1) -- (1.3)]{dibenedetto1985addendum} are satisfied, such that we are in position to apply the result~\cite[Chapter~12, Theorem~2.1]{dibenedetto1985addendum}. In turn, this yields the estimate
   \begin{align*} 
        \esssup\limits_{Q_{\rho}(z_0)} |\partial_{e^*} u_\epsilon| &\leq C \Bigg( \bigg(\fiint_{Q_{2\rho}(z_0)}  |\partial_{e^*} u_\epsilon|^2  \,\dx\dt \bigg)^{\frac{1}{2}} +1  \Bigg) \\
        &\leq C \Bigg( \bigg(\fiint_{Q_{2\rho}(z_0)}  |Du_\epsilon|^2  \,\dx\dt \bigg)^{\frac{1}{2}} +1  \Bigg)
    \end{align*} 
   with~$C=C(C_\F,\hat{C}_\F,\|f\|_{L^{n+2+\sigma}(Q_R)},N,n,R,r_E,\sigma)$, where we also used~\eqref{minkowskialternativ} as well as~\eqref{betragminkowski}. Since~$e^*\in\partial E^*$ was chosen arbitrarily, we again exploit~\eqref{minkowskialternativ} and \eqref{betragminkowski}, which leads to an additional dependence on the radius~$R_E>0$ and yields the claimed~$L^\infty$-gradient estimate~\eqref{dulinfty}.
\end{proof}


For the proof of Theorem~\ref{hauptresultat}, the following lemma, stating the convergence of~$\G_\delta(Du_\epsilon) \to \G_\delta(Du)$ in~$L^2(Q_R,\R^n)$ as~$\epsilon\downarrow 0$, for any~$\delta\in(0,1]$,  will be crucial.


\begin{mylem} \label{approxkvgzinl2}
    Let~$\delta,\epsilon\in(0,1]$ and~$u_\epsilon$ be the unique weak solution to~\eqref{approx}. Then, there holds
    \begin{equation*}
        \G_\delta(Du_\epsilon) \to \G_\delta(Du)\qquad\text{in~$L^2(Q_R,\R^n)$ as~$\epsilon\downarrow 0$.}
    \end{equation*}
\end{mylem}
\begin{proof}
By following the approach taken in the proof of Lemma~\ref{energybound}, we again infer the energy estimate~\eqref{allgemeinerenergyestimate}, i.e. there holds
\begin{align*}
       \frac{1}{2}  \int_{B_R \times\{s_0\}} |u-u_\epsilon|^2\,\dx &+ \iint_{Q_R} \langle \hat{\h}_\epsilon(x,t,Du) -  \hat{\h}_\epsilon(x,t,Du_\epsilon), Du-Du_\epsilon\rangle  \,\dx\dt \\
    &= \epsilon \iint_{Q_R} \langle Du,Du-Du_\epsilon \rangle \,\dx\dt. 
\end{align*}
 An application of Lemma~\ref{lemgdeltakvgz} and also Young's inequality, after discarding the non-negative first term on the left-hand side in the preceding estimate, yields
\begin{align*}
   \epsilon \iint_{Q_R}   |Du_\epsilon &- Du|^2\,\dx\dt + \iint_{Q_R}|\G_\delta(Du_\epsilon) - \G_\delta(Du)|^2\,\dx\dt  \\
    &\leq C(C_\F,\delta,R_E,r_E) \iint_{Q_R} \langle \hat{\h}_\epsilon(x,t,Du)-\hat{\h}_\epsilon(x,t,Du_\epsilon), Du_\epsilon-Du \rangle \,\dx  \\
    &\leq  C(C_\F,\delta,R_E,r_E)\, \epsilon \int_{B_R} |Du||Du_\epsilon-Du|\,\dx  \\
    &\leq \epsilon \int_{B_R} |Du_\epsilon-Du|^2\,\dx + C(C_\F,\delta,R_E,r_E) \, \epsilon \int_{B_R}|Du|^2\,\dx.
\end{align*}
After reabsorbing the first term on the right-hand side of the above inequality into the left-hand side, we note that importantly the constant~$C=C(C_\F,\delta,R_E,r_E)$ does not depend on the parameter~$\epsilon\in(0,1]$, such that the claim follows after passing to the limit~$\epsilon\downarrow 0$.
\end{proof}


\subsection{Hölder continuity of~\texorpdfstring{$\G_\delta(Du_\epsilon)$}{}} \label{sec:holderapprox}
The aim of this section is the proof of the main result, that is Theorem~\ref{hauptresultat}, which is achieved by proving the local Hölder continuity of~$\G_\delta(Du_\epsilon)$ in~$Q_R$ for any~$\delta,\epsilon\in(0,1]$ as an intermediate step. As before, we denote by~$u_\epsilon$ the unique weak solution to the Cauchy-Dirichlet problem~\eqref{approx}. This regularity result is stated in form of Theorem~\ref{holdermainresult}. For this matter, we consider a cylinder~$Q_{2\rho}(z_0)=B_{2\rho}(x_0)\times \Lambda_{2\rho}(t_0)\Subset Q_R= Q_R(y_0,s_0) \Subset \Omega_T$, where~$Q_R$ denotes the cylinder introduced in Section~\ref{sec:regularizing}. We infer from Proposition~\ref{approxregularityeins} and Proposition~\ref{approxregularityzwei} that the approximating family~$(Du_\epsilon)_{\epsilon\in(0,1]}$ is uniformly bounded with respect to the regularizing parameter~$\epsilon\in(0,1]$ in~$Q_{2\rho}(z_0)$ in the sense that there exists a positive constant~$\Tilde{M} < \infty$, which depends on the data
\begin{align*}
    (C_\F,\hat{C}_\F,\|Du\|_{L^{2}(Q_R)},\|f\|_{L^{n+2+\sigma}(Q_R)},N,n,R,R_E,r_E,\sigma),
\end{align*}
but that is independent of~$\epsilon\in(0,1]$, such that
\begin{equation} \label{glmschranke}
 \sup\limits_{\epsilon\in(0,1]} \|Du_\epsilon\|_{L^\infty(Q_{2\rho}(z_0))}\leq \Tilde{M}.
\end{equation}
By employing~\eqref{betragminkowski}, this implies in particular the bound
\begin{equation} \label{minkowskischranke}
 \sup\limits_{\epsilon\in(0,1]} \esssup\limits_{Q_{2\rho}(z_0)} |Du_\epsilon|_{E}\leq \frac{\Tilde{M}}{r_E} \eqqcolon M.
\end{equation}
We assume that~$M\geq 3$. Due to~\eqref{minkowskischranke}, there holds the bound
\begin{equation} \label{schrankeeins}
    \esssup\limits_{Q_{2\rho}(z_0)} |Du_\epsilon|_E \leq 1+\delta+\mu
\end{equation}
for some parameter~$\mu>0$ with
\begin{equation} \label{schrankezwei}
    1+\delta+\mu\leq M.
\end{equation}
Next, for any~$e^{*}\in\partial E^*$, we define the super level set of~$\partial_{e^*} u_\epsilon$ on~$Q_\rho(z_0)$ to some parameter~$\nu\in(0,1)$ by
\begin{equation*}
    E^{\nu}_{e^*,\rho}(z_0) \coloneqq Q_\rho(z_0)\cap \big\{ \partial_{e^*} u_\epsilon - (1+\delta) > (1-\nu)\mu \big\}.
\end{equation*}
Given~$\delta\in(0,1]$, we will employ the following notation
\begin{align} \label{hesseschrankedelta}
    C_\F(\delta) \coloneqq \max\{ \|\nabla^2\hat{\F}\|_{L^\infty(Q_R\times (\R^n\setminus E_\delta))}, 3 C_\F+1 \} < \infty
\end{align}
throughout the following discussion, where~$C_\F$ is the constant from~\eqref{hatapproxbound}. \,\\

As a main result of this section, we obtain the Hölder continuity of~$\G_\delta(Du_\epsilon)$ in~$Q_\rho(z_0)$ with some Hölder exponent~$\alpha_\delta\in(0,1)$ and Hölder constant~$C_\delta$.


\begin{mytheorem} \label{holdermainresult}
    Let~$\delta,\epsilon\in(0,1]$ and~$u_{\epsilon}$ denote the unique weak solution to the Cauchy-Dirichlet problem~\eqref{approx}. Then, there holds
    $$\G_{\delta}(Du_{\epsilon}) \in C^{0,\alpha_\delta,\alpha_\delta /2}(Q_\rho(z_0),\R^n)$$
     with a Hölder exponent~$\alpha_{\delta}\in(0,1)$ and a Hölder constant~$C_{\delta}\geq 1$ depending on the data
     \begin{align*}
         (C_\F(\delta),\hat{C}_\F,\delta,\|f\|_{L^{n+2+\sigma}(Q_R)},M,N,n,R,R_E,r_E,\sigma),
     \end{align*}
    i.e. there holds
    \begin{equation} \label{gdeltaholderquant}
    |\G_{\delta}(Du_{\epsilon}(z_1)) - \G_{\delta}(Du_{\epsilon}(z_2))| \leq C_\delta \Big(|x_1-x_2|+\sqrt{|t_1-t_2|}\Big)^{\alpha_\delta} 
    \end{equation}
for any $z_1=(x_1,t_1),z_2=(x_2,t_2) \in Q_\rho \Subset Q_R$.
\end{mytheorem}


We shall prove Theorem~\ref{holdermainresult} by following an approach in the spirit of DiBenedetto and Friedman's technique~\cite{friedman1984regularity,dibenedetto1985addendum,Friedman1985} (we also refer the reader to Uhlenbeck's work~\cite{uhlenbeck1977regularity}) for parabolic~$p$-Laplacian systems in combination with De Giorgi's beautiful level set approach~\cite{degiorgieins,degiorgizwei}. The idea is to distinguish between the~\textit{non-degenerate regime} and the~\textit{degenerate regime}. In the non-degenerate regime, there holds the measure-theoretic information~$|Q_\rho(z_0)\setminus E^\nu_{e^*,\rho}(z_0)| < \nu |Q_\rho(z_0)|$ for some~$\nu\in(0,\frac{1}{4}]$ for at least one~$e^*\in\partial E^*$, while in the degenerate regime there holds~$|Q_\rho(z_0)\setminus E^\nu_{e^*,\rho}(z_0)| \geq \nu |Q_\rho(z_0)|$ for any~$e^*\in \partial E^*$. Roughly speaking, the subset of points where~$\partial_{e^*}u_\epsilon$ is close to its supremum is large in measure in the non-degenerate regime, whereas in the degenerate regime the subset of points where~$\partial_{e^*}u_\epsilon$ is far from its supremum is large in measure. We simplify notation by setting
\begin{align} \label{beta}
    \beta \coloneqq \frac{\sigma}{n+2+\sigma} = 1-\frac{n+2}{n+2+\sigma} \in (0,1).
\end{align}
 We will denote by
\begin{align} \label{excess}
    \Phi(z_0,\rho) \coloneqq \fiint_{Q_\rho (z_0)} | Du_\epsilon-(Du_\epsilon)_{z_0,\rho} | ^2 \,\dx\dt
\end{align}
the $L^2$-excess of~$Du_\epsilon$ on the cylinder~$Q_\rho(z_0)$, which plays a crucial role in the analysis of the non-degenerate regime.
\,\\


The first proposition deals with the non-degenerate regime.

\begin{myproposition} \label{nondegenerateproposition}
    Let~$\delta,\epsilon\in(0,1]$,~$\mu>0$, and assume that
    \begin{equation} \label{deltamu}
        \delta<\mu.
    \end{equation}
    There exist an exponent~$\alpha \in (0,\beta)$, a constant~$C\geq 1$, a parameter~$\nu\in(0,\frac{1}{4}]$, as well as a radius~$\Tilde{\rho}\in(0,1]$, all subject to the dependencies
    \begin{align*}
      \alpha&=\alpha(C_\F(\delta),\hat{C}_\F,\delta,\|f\|_{L^{n+2+\sigma}(Q_R)},M,N,n,r_E,\sigma), \\
       C&=C(C_\F(\delta),\hat{C}_\F,\delta,\|f\|_{L^{n+2+\sigma}(Q_R)},M,N,n,R,R_E,r_E,\sigma), \\
       \nu&=\nu(C_\F(\delta),\hat{C}_\F,\delta,\|f\|_{L^{n+2+\sigma}(Q_R)},M,N,n,R,R_E,r_E,\sigma), \\
       \Tilde{\rho}&=\Tilde{\rho}(C_\F(\delta),\hat{C}_\F,\delta,\|f\|_{L^{n+2+\sigma}(Q_R)},M,N,n,r_E,\sigma), 
    \end{align*}    
    such that: if the measure condition
    \begin{equation} \label{nondegeneratemeascond}
    |Q_\rho(z_0)\setminus E^{\nu}_{e^*,\rho}(z_0)|<\nu|Q_\rho(z_0)|
    \end{equation}
     is satisfied for at least one~$e^*\in\partial E^*$ on a cylinder $Q_{\rho}(z_0)\subset Q_{2\rho}(z_0)\Subset Q_R$ with $\rho\leq\Tilde{\rho}$, then the limit
    \begin{equation} \label{lebesguerepresentant}
        \Gamma_{z_0}\coloneqq \lim\limits_{r\downarrow 0} (\G_{2\delta}(Du_\epsilon))_{z_0,r}
    \end{equation}
     exists. Furthermore, we have the excess-decay estimate
    \begin{equation} \label{excessgdelta}
        \fiint_{Q_r(z_0)}|\G_{2\delta}(Du_\epsilon)-\Gamma_{z_0}|^2\,\dx\dt \leq C\Big(\frac{r}{\rho} \Big)^{2\alpha} \mu^2
    \end{equation}
    for any $r\in(0,\rho]$, as well as the bound
    \begin{align} \label{lebesguebound}
        |\Gamma_{z_0}|\leq R_E \mu.
    \end{align}
\end{myproposition}


For the following discussion, let~$0<\Tilde{\beta}<\alpha<\beta$, where~$\alpha\in(0,1)$ denotes the parameter from the previous Proposition~\ref{nondegenerateproposition} and~$\beta\in(0,1)$ is defined in~\eqref{beta}. The second proposition treats the degenerate regime.

\begin{myproposition} \label{degenerateproposition}
    Let $\delta,\epsilon\in(0,1]$, $\mu>0$, and $\nu\in(0,\frac{1}{4}]$. Assuming that~\eqref{deltamu} applies, there exist a parameter $\kappa\in\big[2^{-\frac{\Tilde{\beta}}{2}},1\big)$ and a radius~$\hat{\rho}\in(0,1)$, both depending on the data
    \begin{align*}
        (C_\F(\delta),\hat{C}_\F,\delta,\|f\|_{L^{n+2+\sigma}(Q_R)},M,N,n,\nu,R,r_E,\sigma),
    \end{align*}
     such that: if the measure condition
    \begin{equation} \label{degeneratemeascond}
                |Q_\rho(z_0)\setminus E^{\nu}_{e^*,\rho}(z_0)|\geq\nu|Q_\rho(z_0)|
    \end{equation}
     is satisfied for any~$e^*\in\partial E^*$ on a cylinder~$Q_{\rho}(z_0)\subset Q_{2\rho}(z_0)\Subset Q_R$ with~$\rho\leq\hat{\rho}$, then there holds the reduction of supremum 
    \begin{equation} \label{degenerateest}
        \esssup\limits_{Q_{\Tilde{\nu}\rho}(z_0)}\,|\G_{\delta}(Du_\epsilon)|_E \leq \kappa \mu,
    \end{equation}
    where~$\Tilde{\nu}\coloneqq\frac{\sqrt{\nu}}{2}$.
\end{myproposition}


We postpone the proof of Propositions~\ref{nondegenerateproposition} and~\ref{degenerateproposition} for the moment, addressing them at the end of the article in Sections~\ref{sec:nondegenerate} and~\ref{sec:degenerate}.


\begin{proof}[\textbf{\upshape Proof of Theorem~\ref{holdermainresult}}]
The local Hölder continuity result of~$\G_\delta(Du_\epsilon)$ in~$Q_R$ stated in Theorem~\ref{holdermainresult} follows by combining the results of Propositions~\ref{nondegenerateproposition} and~\ref{degenerateproposition}. Due to similarity to the proof of~\cite[Theorem~3.6]{bogelein2023higher} and~\cite[Theorem~3.5]{elliptisch} in the elliptic setting resp.~\cite[Theorem~3.7]{bogelein2024gradient} in the parabolic setting, we refrain from stating the full proof of Theorem~\ref{holdermainresult} and rather refer to the aforementioned articles. For technical reasons it turns out expedient to prove the local Hölder continuity of~$\G_{2\delta}(Du_\epsilon)$, from which the local Hölder continuity of~$\G_{\delta}(Du_\epsilon)$ of course follows, due to the arbitrariness of~$\delta\in(0,1]$. We additionally remark that it is not restrictive to assume~\eqref{deltamu}in both the non-degenerate regime in Proposition~\ref{nondegenerateproposition} and also in the degenerate regime in Proposition~\ref{degenerateproposition}, similarly as in the elliptic counterpart~\cite[Proposition~3.7]{elliptisch}. Indeed, this stands in contrast to~\cite[Proposition~3.5]{bogelein2023higher} and~\cite[Proposition~3.6]{bogelein2024gradient} where this assumption is not required in the degenerate regime but in the non-degenerate regime only. This can be seen as follows: if the degenerate regime applies at some step during the iteration scheme, i.e. on some cylinder~$Q_{\rho_i}(z_0)\Subset Q_R$ with shrinking radii~$\rho_i\in(0,\rho)$ and levels~$\mu_i=\kappa^i\mu$, where~$\kappa\in\big[2^{-\frac{\Tilde{\beta}}{2}},
1\big)$ denotes the parameter from Proposition~\ref{degenerateproposition}, and the assumption~\eqref{deltamu} is not satisfied on this particular cylinder~$Q_{\rho_i}(z_0)$, i.e. there holds~$\mu_i\leq \delta$, then it follows that~$\G_{2\delta}=0$ on ~$Q_{\rho_i}(z_0)$. Moreover, due to~$G_{2\delta}\leq \G_\delta$ a.e. in~$Q_\rho(z_0)$, we infer
\begin{align} \label{iterationsup}
    \esssup\limits_{Q_{\rho_i}(z_0)}|\G_{2\delta}(Du_\epsilon)|_E \leq \mu_i \qquad\mbox{for any~$i\in\N_0$}, 
\end{align} 
which equals~\cite[(3.15)]{bogelein2023higher},~\cite[(3.15)]{bogelein2024gradient},~\cite[3.31]{elliptisch}. Moreover, similarly to the aforementioned articles, in case the degenerate regime applies in each step of the iteration process, we obtain due to~\eqref{iterationsup} that~$\Gamma_{z_0}$ exists and is equal to zero, a consequence of~\eqref{betragminkowski}.
\end{proof}


Finally, we are in position to provide the proof of our main result, that is Theorem~\ref{hauptresultat}

\begin{proof}[\textbf{\upshape Proof of Theorem~\ref{hauptresultat}}]
Similarly to before, let us consider the arbitrarily chosen cylinder~$Q_R=Q_R(y_0,s_0)\Subset\Omega_T$ and let
$$ u_\epsilon \in L^\infty\big(\Lambda_R(s_0),u+W^{1,\infty}_0(B_R(y_0))\big) $$
denote the unique weak solution to the Dirichlet problem~\eqref{approx} for~$\epsilon\in(0,1]$.
Further, let~$\delta\in(0,1]$ and consider a cylinder~$Q_\rho(z_0)\Subset Q_R \Subset \Omega_T$. In this setting, Theorem~\ref{holdermainresult} is applicable, which asserts that~$\G_\delta(Du_\epsilon)$ is Hölder continuous in~$\overline{Q}_\rho(z_0)$ with a Hölder exponent~$\alpha_\delta\in(0,1)$ and a Hölder constant~$C_\delta$ that depend on the data
\begin{align*}
       \alpha_\delta &= \alpha_\delta(C_\F(\delta),\hat{C}_\F,\delta,\|f\|_{L^{n+2+\sigma}(Q_R)},M,N,n,R,r_E,\sigma) \in(0,1), \\
         C_\delta &= C_\delta(C_\F(\delta),\hat{C}_\F,\delta,\|f\|_{L^{n+2+\sigma}(Q_R)},M,N,n,R,R_E,r_E,\sigma)\geq 1.
\end{align*}
According to Lemma~\ref{lemgdeltakvgz} there holds the convergence~$\G_\delta(Du_\epsilon)\to\G_\delta(Du)$ in~$L^2(Q_R)$ as~$\epsilon\downarrow 0$. Thus, there exists a subsequence~$(\epsilon_i)_{i\in\N}\subset (0,1]$ with~$\G_\delta(Du_{\epsilon_i}) \to \G_\delta(Du)$ a.e. in~$Q_R$ as~$\epsilon_i\downarrow 0$. As ~$(\G_\delta(Du_\epsilon))_{\epsilon\in(0,1]}$ is locally uniformly equicontinuous in~$Q_R$ and also locally uniformly bounded in~$Q_R$, we may apply the theorem of Arzelà-Ascoli, allowing us to pass to the limit~$\epsilon\downarrow 0$, which yields that the limit function~$\G_\delta(Du)$ itself is Hölder continuous in~$\overline{Q}_\rho(z_0)$ with Hölder constant~$C_\delta$ and Hölder exponent~$\alpha_\delta\in(0,1)$. 

Furthermore, due to~\eqref{betragminkowski} and Lemma~\ref{lem:algineq}, we estimate
\begin{align*}
    |\G_\delta(Du) - \G(Du)| &= \bigg| \frac{(|Du|_E - (1+\delta))_+}{|Du|_E}Du - \frac{(|Du|_E-1)_+}{|Du|_E}Du \bigg| \\
    &\leq \frac{|Du|}{|Du|_E} \big|(|Du|_E-(1+\delta))_+ - (|Du|_E-1)_+ \big| \\
    &\leq \mbox{$\frac{1}{r_E}$} \big|(|Du|_E-(1+\delta))_+ - (|Du|_E-1)_+ \big|  \\
    &\leq \mbox{$\frac{1}{r_E}$} \delta
\end{align*}
for a.e.~$(x,t) \in \overline{Q}_\rho(z_0)$. This verifies that the family~$\G_\delta(Du)\to\G(Du)$ is uniformly continuous in~$\overline{Q}_\rho(z_0)$. As a consequence, the limit function~$\G(Du)$ itself is uniformly continuous in~$\overline{Q}_\rho(z_0)$. 

For any~$\epsilon >0$ there exists~$\kappa>0$ with
\begin{align} \label{modulusofcontinuity}
    |\G(Du(z_1))-\G(Du(z_2))| <\epsilon \qquad \mbox{for any~$x,y \in \overline{Q}_\rho(z_0)$ with~$d_p(z_1,z_2)<\kappa$}.
\end{align}
Let~$\mathcal{K}\in C^0(\R^n)$ denote an arbitrary continuous function that vanishes inside the bounded and convex set of degeneracy~$E\subset\R^n$, according to the statement of Theorem~\ref{hauptresultat}. Since there holds
$$ u\in L^\infty_{\loc}\big(0,T; W^{1,\infty}_{\loc}(\Omega)), $$
there exists a positive constant~$0<M<\infty$, such that~$|Du| \leq M$ a.e. in~$\overline{Q}_\rho(z_0)$. Then, there exists a modulus of continuity~$\omega_M\colon \R_{\geq 0} \to\R_{\geq 0}$ on~$\overline{B}_M$ with
$$|\mathcal{K}(\xi)-\mathcal{K}(\eta)|\leq \omega_M(|\xi-\eta|) \qquad\mbox{for any~$\xi,\eta \in \overline{B}_M$}.$$
We now distinguish between two cases. First, we assume that~$|Du(x)|_E\leq 1+\sqrt{\epsilon}$. If~$|Du(x)|_E\geq 1$, we use the fact that~$\mathcal{K}\equiv 0$ on~$E$ and the homogeneity of the Minkowski functional~$|\cdot|_E$ , which yield
\begin{align*}
    |\mathcal{K}(Du(x))| &= \bigg| \mathcal{K}(Du(x))-\mathcal{K}\Big(\frac{Du(x)}{|Du(x)|_E}\Big) \bigg| \\
    &\leq \omega_M\bigg(\Big|Du(x)-\frac{Du(x)}{|Du(x)|_E}\Big|\bigg) \\
    &= \omega_M\bigg(\frac{|Du(x)|}{|Du(x)|_E}(|Du(x)|_E-1)\bigg) \\
    &\leq \omega_M(R_E\sqrt{\epsilon}).
\end{align*}
If there holds~$|Du(x)|_E\leq 1$, then the above estimate holds true trivially. The modulus of continuity~\eqref{modulusofcontinuity}, the triangle inequality~\eqref{lem:minkowskitriangle}, and also~\eqref{betragminkowski} yield
\begin{align*}
    (|Du(y)|_E-1)_+ &= |\G(Du(y))|_E \\
    &\leq |\G(Du(y)) - \G(Du(x))|_E + |\G(Du(x))|_E \\
    &< \mbox{$\frac{\epsilon}{r_E}$} + \sqrt{\epsilon} \\
    &\leq 2\max\{\mbox{$\frac{1}{r_E}$},1 \}\sqrt{\epsilon}.
\end{align*}
This implies in particular
$$|Du(y)|_E \leq 1+2\max\{ \mbox{$\frac{1}{r_E}$},1 \}\sqrt{\epsilon}.$$
Thus by following the very same steps as above and exploiting~\eqref{betragminkowski} once more, we obtain~$|\mathcal{K}(Du(y))| \leq \omega_M\big(2 R_E\max\big\{\frac{1}{r_E},1 \big\}\sqrt{\epsilon} \big)$. By combining the estimates for~$|\mathcal{K}(Du(x))|$ and also~$|\mathcal{K}(Du(y))|$, there holds
$$|\mathcal{K}(Du(x))-\mathcal{K}(Du(y))| \leq \omega_M(R_E\sqrt{\epsilon}) + \omega_M\big(2 R_E\max\{ \mbox{$\frac{1}{r_E}$},1 \} \sqrt{\epsilon} \big) \leq 2 \omega_M\big(2 R_E\max\{ \mbox{$\frac{1}{r_E}$},1\} \sqrt{\epsilon}\big), $$
such that~$\mathcal{K}(Du)$ is continuous in~$\overline{Q}_\rho(z_0)$ in the case where~$|Du(x)|_E\leq 1+\sqrt{\epsilon}$. In the remaining case~$|Du(x)|_E>1+\sqrt{\epsilon}$, we may apply Lemma~\ref{gdeltalem}, yielding that
\begin{align*}
    |Du(x)-Du(y)| &\leq \frac{C(R_E,r_E)}{\sqrt{\epsilon}}|\G(Du(x))-\G(Du(y))| \\
    &\leq C(R_E,r_E) \sqrt{\epsilon},
\end{align*}
which further implies
$$|\mathcal{K}(Du(x)) - \mathcal{K}(Du(y))| \leq \omega_M(C(R_E,r_E) \sqrt{\epsilon}).$$
Consequently, we have verified that~$\mathcal{K}(Du)$ is continuous in~$\overline{Q}_\rho(z_0)$. Since all cylinders~$Q_\rho(z_0)\Subset Q_R$ were chosen arbitrarily, the composition~$\mathcal{K}(Du)$ is continuous in the entire space-time domain~$\Omega_T$, finishing the proof.
\end{proof}


\begin{remark} \label{remarknachhauptresultatproof}
    \upshape 
It is noteworthy that during the proof of Theorem~\ref{hauptresultat}, the control of the quantitative Hölder exponent is lost in the limit as~$\delta \downarrow 0$, leading to uniform continuity of the limit function~$\G(Du)$ locally in~$Q_R$ instead of desirable Hölder continuity. 
To our knowledge it remains an open question even for the model equation~\eqref{prototype} what the best possible modulus of continuity of~$\G(Du)$ may be and whether~$\G(Du)$ is actually locally Hölder continuous in~$Q_R$, posing an intriguing question for further exploration in future research, cf.~\cite{bogelein2023higher}. However, in general~$\mathcal{K}(Du)$ fails to be Hölder continuous for certain integrands~$\F$ in~$\Omega_T$ for any Hölder exponent~$\alpha \in (0,1)$. In fact, it has been demonstrated in~\cite[Remark~1.3]{colombo2017regularity} in the elliptic setting that for a certain function~$\F$ there exist Lipschitz continuous functions~$\mathcal{K} \in C^0(\R^n, \R)$ that vanish inside~$\overline{B}_1$, for which~$\mathcal{K}(Du)$ is not Hölder continuous for any Hölder exponent~$\alpha\in (0,1)$.
\end{remark}


\section{Differentiating the equation} \label{sec:energyestimates}

The aim of this section lies in deriving appropriate conclusions by differentiating the equation~\eqref{weakformapprox} that will be crucial for the discussion of the concluding Sections~\ref{sec:nondegenerate} and~\ref{sec:degenerate}. We recall that~$\xi \mapsto\hat{\F}_\epsilon(x,t,\xi)$ is only of class~$C^2(\R^n\setminus \overline{E})$ for any~$(x,t)\in Q_R$ according to its definition in~\eqref{psifapprox}. Therefore, careful consideration is required when selecting test functions to ensure the validity of differentiating the equation. Throughout this section, we let~$\delta,\epsilon\in(0,1]$ and assume that~$u_\epsilon$ denotes the weak solution to~\eqref{approx} on a parabolic cylinder~$Q_{2\rho}(z_0)\Subset Q_R$. \,\\

We begin by establishing that
$$v_\epsilon = (\partial_{e^*}u_\epsilon - (1+\delta))^2_+,$$
is a weak sub-solution to a linear parabolic equation on~$Q_{2\rho}(z_0)\Subset Q_R$, stated in form of the subsequent lemma. 


\begin{mylem} \label{subsollemma}
    Let~$\delta,\epsilon\in(0,1]$,~$e^*\in\partial E^*$, and~$u_\epsilon$ denote the unique weak solution to~\eqref{approx}. The function
    $$v_\epsilon \coloneqq (\partial_{e^*}u_\epsilon-(1+\delta))^2_+$$
    is a weak sub-solution to a linear parabolic equation on~$Q_{2\rho}(z_0)\Subset Q_R$ in the sense that there holds
    \begin{align*}
        \iint_{Q_{2\rho}(z_0)} \big( & -v_\epsilon\partial_t\zeta + \hat{\B}_\epsilon(x,t,Du_\epsilon)(Dv_\epsilon,D\zeta) \big)\,\dx\dt \\
        &\leq C \bigg(\iint_{Q_{2\rho}(z_0)} \zeta (1+|f|^2)\,\dx\dt + \iint_{Q_{2\rho}(z_0)} |D\zeta| (1+|f|)\,\dx\dt \bigg)
    \end{align*}
    for any non-negative test function~$\zeta\in C^1_0(Q_{2\rho}(z_0))$, with constant~$C=C(C_\F(\delta),\hat{C}_\F,\delta,M,N,n,r_E)$. 
\end{mylem}


\begin{proof}
 To abbreviate notation we again write~$u=u_\epsilon$. Let~$e^*\in\partial E^*$ be arbitrary. We test the weak form~\eqref{weakformapprox} with the test function~$\partial_{e^*}\phi_{e^*}$, where~$\phi_{e^*}=\zeta (\partial_{e^*} u-(1+\delta))_+$. Here,~$\zeta\in C^1_0(Q_{2\rho }(z_0))$ denotes an arbitrary smooth cut-off function. Indeed, this choice of test function is admissible due to an approximation argument. Subsequently, we integrate by parts, which yields 
 \begin{align*} 
     \iint_{Q_{2\rho} (z_0)} & \partial_t[\partial_{e^*}u] \phi_{e^*} \,\dx\dt + \iint_{Q_{2\rho} (z_0)} \langle \partial_{e^*} [\hat{\h}_\epsilon(x,t,Du)], D\phi_{e^*}\rangle \,\dx\dt = - \iint_{Q_{2\rho} (z_0)} f \partial_{e^*}\phi_{e^*}\,\dx\dt. 
 \end{align*}
 As before, this procedure can be justified by a rigorous use of Steklov-averages. We abbreviate notation by setting~$a\coloneqq 1+\delta$. A direct calculation verifies that
 \begin{align*}
    \partial_j \phi_{e^*} &= \partial_j \zeta (\partial_{e^*} u-a)_+ +  \zeta \partial_j \partial_{e^*} u \bigchi_{\{ \partial_{e^*} u>a \}}
\end{align*}
holds true for any~$j\in\{1,\ldots,n\}$ and for a.e.~$(x,t)\in Q_{2\rho}(z_0)$. This leads to
 \begin{align} \label{subsolweak}
     \iint_{Q_{2\rho} (z_0)} & \partial_t[\partial_{e^*}u] \zeta (\partial_{e^*}u-a)_+ \,\dx\dt \\
     &+ \iint_{Q_{2\rho} (z_0)} \langle \partial_{e^*} [\hat{\h}_\epsilon(x,t,Du)], D[\zeta (\partial_{e^*} u-a)_+]\rangle \,\dx\dt \nonumber \\
     &= \iint_{Q_{2\rho} (z_0)} f \partial_{e^*}[\zeta (\partial_{e^*} u-a)_+]\,\dx\dt. \nonumber
 \end{align}
We begin by treating the term involving the time derivative in~\eqref{subsolweak}, where we use the fact
\begin{align} \label{videntität}
    \partial_t[\partial_{e^*}u] \zeta (\partial_{e^*}u-a)_+ = \mbox{$\frac{1}{2}$} \partial_t (\partial_{e^*}u-a)^2_+ \zeta \qquad\mbox{a.e. in~$Q_{2\rho}(z_0)$}, 
\end{align} 
which yields
\begin{align} \label{timeterm}
   \iint_{Q_{2\rho} (z_0)} \partial_t[\partial_{e^*}u] \zeta (\partial_{e^*}u-a)_+ \,\dx\dt 
    &= \frac{1}{2} \iint_{Q_{2\rho} (z_0)} \partial_t (\partial_{e^*}u-a)^2_+ \zeta \,\dx\dt \\
    &= - \frac{1}{2} \iint_{Q_{2\rho} (z_0)} (\partial_{e^*}u-a)^2_+ \partial_t\zeta \,\dx\dt \nonumber \\
    &= - \frac{1}{2} \iint_{Q_{2\rho} (z_0)} v_\epsilon \partial_t\zeta \,\dx\dt. \nonumber
\end{align}
Next, the diffusion term in~\eqref{subsolweak} is treated. We obtain two integral terms
\begin{align*}
      = \iint_{Q_{2\rho} (z_0)} & \langle \partial_{e^*} [\hat{\h}_\epsilon(x,t,Du)], D[\zeta (\partial_{e^*}u - a)_+] \rangle \,\dx\dt \\
    &= \iint_{Q_{2\rho} (z_0)} \hat{\B}_\epsilon(x,t,Du)(\partial_{e^*}Du,D[\zeta (\partial_{e^*}u - a)_+]) \,\dx\dt \\
    &\quad + \iint_{Q_{2\rho} (z_0)} \langle \partial_{x_{e^*}} \hat{\h}_\epsilon(x,t,Du),D[\zeta (\partial_{e^*}u - a)_+] \rangle \,\dx\dt \\
    &\eqqcolon \foo{I} + \foo{II}.
\end{align*}
By exploiting the bilinearity of~$\B_\epsilon$ and identity~\eqref{videntität} once more, the first term~$\foo{I}$ yields
\begin{align*}
    \foo{I} &= \iint_{Q_{2\rho} (z_0)} \hat{\B}_\epsilon(x,t,Du)(\partial_{e^*} Du,D\zeta) (\partial_{e^*}u-a)_+ \,\dx\dt \\
    &\quad + \iint_{Q_{2\rho} (z_0)} \zeta \hat{\B}_\epsilon(x,t,Du)(\partial_{e^*} Du,\partial_{e^*} Du) \bigchi_{\{ \partial_{e^*} u>a \}} \,\dx\dt \\
    &= \frac{1}{2} \iint_{Q_{2\rho} (z_0)} \hat{\B}_\epsilon(x,t,Du)(Dv_\epsilon,D\zeta) \,\dx\dt \\
    &\quad + \iint_{Q_{2\rho} (z_0)} \zeta \hat{\B}_\epsilon(x,t,Du)(\partial_{e^*} Du,\partial_{e^*} Du) \bigchi_{\{ \partial_{e^*} u>a \}} \,\dx\dt.
\end{align*}
 We note that for any~$e^*\in\partial E^*$ there holds the set inclusion
 $$ \{ (x,t)\in Q_{2\rho}(z_0): \partial_{e^*}u \geq 1+\delta \} \subset \{ (x,t)\in Q_{2\rho}(z_0): |Du|_E \geq 1+\delta \}. $$
Therefore, whenever the integrand is positive, we are in position to exploit our ellipticity condition of~$\hat{\B}_\epsilon$ in form of Lemma~\ref{bilinearelliptic} for the second integral term in the preceding identity. By further recalling the definition of~$C_\F$ in~\eqref{schrankehessian} and also~\eqref{hesseschrankedelta}, there holds the lower bound
\begin{align} \label{einssummed}
     \foo{I} &\geq \frac{1}{2} \iint_{Q_{2\rho} (z_0)} \hat{\B}_\epsilon(x,t,Du)(Dv_\epsilon,D\zeta) \,\dx\dt \\
    &\quad + (\epsilon+\Tilde{C}(C_\F(\delta),\delta))  \iint_{Q_{2\rho} (z_0)} \zeta |\partial_{e^*} Du|^2 \bigchi_{\{ \partial_{e^*} u>a \}} \,\dx\dt \nonumber 
\end{align}
with~$\Tilde{C}=\Tilde{C}(C_\F(\delta),\delta)$ denoting a positive constant. Let us next treat the second quantity~$\foo{II}$. By employing the Cauchy-Schwarz inequality and the Lipschitz property~\eqref{hatflipschitz}, together with the representation formula~\eqref{representation} and the bound~$|e^*|\leq R_E$, we obtain
\begin{align} \label{est:lundm}
    \foo{II} &\leq C(M) \iint_{Q_{2\rho} (z_0)} |D[\zeta (\partial_{e^*}u - a)_+]|\,\dx\dt \\
    &\leq C \iint_{Q_{2\rho} (z_0)} |D\zeta| (|\partial_{e^*}u|-a)_+ \,\dx\dt + C \iint_{Q_{2\rho} (z_0)} \zeta |\partial_{e^*} Du| \bigchi_{\{ |\partial_{e^*} u|>a \}} \,\dx\dt \nonumber
\end{align}
with a constant~$C=C(\hat{C}_\F,N,n,R_E,r_E)$. Employing Young's inequality, we thus further estimate
\begin{align} \label{zweisummed}
   \foo{II} &\leq C \iint_{Q_{2\rho} (z_0)} |D\zeta| (|\partial_{e^*}u|-a)_+ \,\dx\dt + \mbox{$\frac{1}{4}$} \Tilde{C} \iint_{Q_{2\rho} (z_0)} \zeta |\partial_{e^*}Du|^2 \bigchi_{\{ |\partial_{e^*} u|>a \}} \,\dx\dt \\
     &\quad + C \iint_{Q_{2\rho} (z_0)} \zeta  \bigchi_{\{ |\partial_{e^*} u|>a \}} \,\dx\dt  \nonumber
\end{align}
with a constant~$C=C(C_\F(\delta),\hat{C}_\F,\delta,N,n,R_E,r_E)$. Finally, we treat the term involving the datum~$f$ on the right-hand side of~\eqref{subsolweak}. We estimate this term further above by
\begin{align*}
    \foo{III} &\coloneqq \iint_{Q_{2\rho} (z_0)} f \partial_{e^*}[\zeta(\partial_{e^*}u-a)_+]\,\dx\dt \\
    &\leq \iint_{Q_{2\rho} (z_0)} |f| |D\zeta| (\partial_{e^*}u -a)_+\,\dx\dt + \iint_{Q_{2\rho} (z_0)}  \zeta |f| |\partial_{e^*} Du| \bigchi_{\{ \partial_{e^*}u > a \}} \,\dx\dt.
\end{align*}
Once again, we employ Young's inequality, just as before. This way, there holds
\begin{align} \label{datumsummed}
  \foo{III} &\leq C \iint_{Q_{2\rho} (z_0)} |f| |D\zeta| (\partial_{e^*}u -a)_+\,\dx\dt + \mbox{$\frac{1}{4}$}\Tilde{C} \iint_{Q_{2\rho} (z_0)} \zeta |\partial_{e^*} Du|^2 
\bigchi_{\{ \partial_{e^*}u>a \}} \,\dx\dt  \\ 
    &\quad + C \iint_{Q_{2\rho} (z_0)} |f|^2 \zeta \bigchi_{\{ \partial_{e^*}u>a \}} \,\dx\dt \nonumber
\end{align}
with a constant~$C=C(C_\F(\delta),\delta)$. After combining our results~\eqref{timeterm},~\eqref{einssummed},~\eqref{zweisummed}, and~\eqref{datumsummed}, and reabsorbing the second quantity involving second order weak derivatives of~$u$ in both~\eqref{zweisummed} and also~\eqref{datumsummed} into the left-hand side, we overall end up with
\begin{align} \label{combined}
    &-  \iint_{Q_{2\rho} (z_0)} v_\epsilon \partial_t\zeta \,\dx\dt + \iint_{Q_{2\rho} (z_0)} \B_\epsilon(x,t,Du)(Dv_\epsilon,D\zeta)\,\dx\dt \\
    &\quad + (\epsilon+\Tilde{C}(\delta)) \iint_{Q_{2\rho} (z_0)} \zeta |\partial_{e^*}Du|^2  \bigchi_{\{ \partial_{e^*} u>a \}} \,\dx\dt \nonumber \\
    &\leq C \iint_{Q_{2\rho} (z_0)} |D\zeta| ( 1+|f|) (\partial_{e^*} u-a)_+ \,\dx\dt + C \iint_{Q_{2\rho} (z_0)} \zeta ( 1+|f|^2 ) \bigchi_{\{ \partial_{e^*} u>a \}}\,\dx\dt \nonumber
\end{align}
with~$C=C(C_\F(\delta),\hat{C}_\F,\delta,N,n,r_E)$. We discard the third quantity on the left-hand side of~\eqref{combined} due to a non-negative contribution, while the right-hand side is estimated further above with the bounds~\eqref{schrankeeins} and~\eqref{schrankezwei}. This way, we obtain the claimed estimate, as stated in the lemma. 
\end{proof}


Next, we modify the sub-solution~$v_\epsilon$ from the preceding Lemma~\ref{subsollemma}. \eqref{minkowskialternativ} and the general bound~\eqref{schrankeeins} imply the upper bound
\begin{align*}
    v_\epsilon &= (\partial_{e^*}u_\epsilon-(1+\delta))^2_+ \leq (|Du_\epsilon|_E-(1+\delta))^2_+ \leq \mu^2 \qquad\mbox{a.e. in~$Q_{2\rho}(z_0)$}
\end{align*}
for any~$e^*\in\partial E^*$. Thus, by setting
\begin{align} \label{tildev}
    \hat{v}_\epsilon\coloneqq \frac{1}{\mu^2} v_\epsilon = \frac{(\partial_{e^*}u_\epsilon-(1+\delta))^2_+}{\mu^2} \qquad\mbox{a.e. in~$Q_{2\rho}(z_0)$},
\end{align}
 we have that~$\hat{v}_\epsilon \leq 1$ a.e. in~$Q_{2\rho}(z_0)$. Moreover, we rescale~$\hat{v}_\epsilon$,~$\hat{\h}_\epsilon$,~$\hat{\B}_\epsilon$, and also~$f$ to the cylinder~$Q_{2\rho}\coloneqq Q_{2\rho}(0)$ with vertex~$0$ by considering
\begin{align}  \label{tildevrescaled}
      \begin{cases} 
          \Tilde{v}_\epsilon(x,\tau) \coloneqq \hat{v}_\epsilon(x_0+x,t_0+\tau)  & \mbox{for~$(x,\tau)\in Q_{2\rho}$}, \\
     \Tilde{\h}_\epsilon(x,\tau,\xi) \coloneqq \hat{\h}_\epsilon(x_0+x,t_0+\tau,\xi)  &\mbox{for~$(x,\tau)\in Q_{2\rho}$,~$\xi\in\R^n$},  \\
     \Tilde{\B}_\epsilon(x,\tau,\xi) \coloneqq \hat{\B}_\epsilon(x_0+x,t_0+\tau,\xi) &\mbox{for~$(x,\tau)\in Q_{2\rho}$~$\xi\in\R^n$},  \\
      \Tilde{f}(x,\tau) \coloneqq f(x_0+x,t_0+\tau) &\mbox{for~$(x,\tau)\in Q_{2\rho}$}. 
       \end{cases}
     \end{align}
   Due to the assumption~\eqref{deltamu} and by also rescaling the test functions~$\zeta\in C^1_0(Q_{2\rho}(z_0))$ in the weak form of~\ref{subsollemma} to~$Q_{2\rho}$ by considering
 $$\Tilde{\zeta}(x,\tau)\coloneqq \zeta(x_0+x,t_0+\tau)\in C^1_0(Q_{2\rho}),$$
 we immediately obtain the following corollary. 


\begin{mycor} \label{subsoltildev}
     Let~$\delta,\epsilon\in(0,1]$ and~$u_\epsilon$ be a weak solution to~\eqref{approx}. The function~$\Tilde{v}=\Tilde{v}_\epsilon$ defined in~\eqref{tildevrescaled}
    is a weak sub-solution to a linear parabolic equation in the sense that there holds
    \begin{align*}
        \iint_{Q_{2\rho}} \big( & -\Tilde{v} \partial_t \Tilde{\zeta} + \Tilde{\B}_\epsilon(x,\tau,Du_\epsilon(x_0+x,t_0+\tau))(D\Tilde{v},D\Tilde{\zeta}) \big)\,\dx\dt \\
        &\leq C \bigg(\iint_{Q_{2\rho}} \Tilde{\zeta} (1+|f|^2)\,\dx\dt + \iint_{Q_{2\rho}} |D\Tilde{\zeta}| (1+|\Tilde{f}|)\,\dx\dt \bigg)
    \end{align*}
    for any non-negative test function~$\Tilde{\zeta}\in C^1_0(Q_{2\rho})$, with constant~$C=C(C_\F(\delta),\hat{C}_\F,\delta,M,N,n,r_E)$.
\end{mycor}


For convenience, we revert to denoting the time variable by~$t$ instead of~$\tau$. Since the bilinear form~$\Tilde{\B}_\epsilon(x,t,Du_\epsilon)(D\Tilde{v},D\Tilde{\zeta})$ vanishes on the set of points
$$Q_{2\rho}\cap \{\partial_{e^*}u_\epsilon(x_0+x,t_0+t) \leq 1+\delta\},$$
we may redefine~$\Tilde{\B}_\epsilon$ to the identity operator on the set of points where~$\Tilde{v}=0$ and obtain that~$\Tilde{v}$ is a sub-solution to a linear parabolic equation with elliptic and bounded coefficients. Consequently, two classical energy estimates of De Giorgi class-type can be inferred. 


\begin{mylem} [De Giorgi class-type estimates for a sub-solution] \label{degiorgilem}
    Let~$\delta,\epsilon\in(0,1]$ and~$\Tilde{v}=\Tilde{v}_\epsilon $ denote the weak sub-solution defined in~\eqref{tildevrescaled}. There exists~$C=C(C_\F(\delta),\hat{C}_\F,\delta,M,N,n,r_E)$, such that for any~$k>0$,~$\tau\in(0,1)$, and any cylinder~$Q_{r,s}(0,\tau_0)=B_r\times(\tau_0-s,\tau_0] \subset Q_{2\rho}$, there hold
    \begin{align} \label{degiorgieins}
        \esssup\limits_{\tau\in(\tau_0-\tau s,\tau_0]} & \int_{B_{\tau r}\times\{\tau\}} (\Tilde{v}-k)^2_+\,\dx + \iint_{Q_{\tau r,\tau s}(0,\tau_0)} |D(\Tilde{v} - k)_+|^2\,\dx\dt \\
        &\leq C \iint_{Q_{r,s}(0,\tau_0)} \bigg(\frac{(\Tilde{v}-k)^2_+}{(1-\tau)^2 r^2} + \frac{(\Tilde{v}-k)^2_+}{(1-\tau) s} \bigg)\,\dx\dt \nonumber \\
        & \quad + C \big(\rho^{2(1-\beta)} + \|f\|^2_{L^{n+2+\sigma}(Q_R)} \big) |Q_{r,s}(0,\tau_0) \cap \{\Tilde{v}>k\}|^{1-\frac{2}{n+2}+\frac{2\beta}{n+2}} \nonumber
    \end{align}
    and
    \begin{align} \label{degiorgizwei}
          \esssup\limits_{\tau\in(\tau_0  -\tau s,\tau_0]} &  \int_{B_{\tau r}\times\{\tau\}} (\Tilde{v}-k)^2_+\,\dx \\
          &\leq  \int_{B_r\times\{\tau_0-s\}} (\Tilde{v}-k)^2_+\,\dx +  C \iint_{Q_{r,s}(0,\tau_0)}\frac{(\Tilde{v}-k)^2_+}{(1-\tau)^2 r^2} \,\dx\dt \nonumber  \\
        & \quad + C \big(\rho^{2(1-\beta)} + \|f\|^2_{L^{n+2+\sigma}(Q_R)} \big) |Q_{r,s}(0,\tau_0) \cap \{\Tilde{v} >k\}|^{1-\frac{2}{n+2}+\frac{2\beta}{n+2}}. \nonumber
    \end{align}
\end{mylem}

\begin{proof}
    As explained above, we further redefine the bilinear form~$\Tilde{\B}_\epsilon$. For this matter, let us incorporate~$\B^*_\epsilon$, which is given by
    \begin{align*}
       \B^*_\epsilon & (x,t,Du_\epsilon(x_0+x,t_0+t)) \\
       &\coloneqq \begin{cases}
           \Tilde{\foo{I}}_n & \mbox{on~$\{ \partial_{e^*} u_\epsilon(x_0+x,t_0+t) \leq 1+\delta\}$}, \\
           \Tilde{\B}_\epsilon(x,t,Du_\epsilon(x_0+x,t_0+t)) & \mbox{on~$\{ \partial_{e^*}u_\epsilon(x_0+x,t_0+t)> 1+\delta\}$},
       \end{cases} 
    \end{align*}
where~$\Tilde{\foo{I}}_n(\xi,\eta)\coloneqq \langle \foo{I}_n \xi,\eta\rangle=\langle\xi,\eta\rangle$ for any~$\xi,\eta\in\R^n$. Consequently,~$\Tilde{v}$ is a weak sub-solution to
 \begin{align*}
        \iint_{Q_{2\rho}} \big( & -\Tilde{v}_\epsilon\partial_t\Tilde{\zeta} + \B^*_\epsilon(x,t,Du_\epsilon(x_0+x,t_0+t))(D\Tilde{v},D\Tilde{\zeta}) \big)\,\dx\dt \\
        &\leq C \bigg(\iint_{Q_{2\rho}} \Tilde{\zeta} (1+|\Tilde{f}|^2)\,\dx\dt + \iint_{Q_{2\rho}} |D\Tilde{\zeta}| (1+|\Tilde{f}|)\,\dx\dt \bigg)
    \end{align*}
    for any non-negative test function~$\Tilde{\zeta}\in C^1_0(Q_{2\rho})$, with constant~$C=C(\delta,M)$. Moreover, the redefined bilinear form~$\B^*_\epsilon$ is elliptic and bounded on the whole of~$\R^n$, i.e. there exists a constant~$\Tilde{C}=\Tilde{C}(C_\F(\delta),\delta)\geq 1$, such that
    \begin{align} \label{degiorgiellipticity}
        \Tilde{C}^{-1}|\xi|^2 \leq \B^*_\epsilon(x,t,Du_\epsilon(x_0+x,t_0+t))(\xi,\xi) \leq \Tilde{C}|\xi|^2
    \end{align}
    is satisfied for any~$(x,t)\in Q_{2\rho}$ and any~$\xi\in\R^n$. The preceding bound clearly holds true with constant~$\Tilde{C}=1$ in the set of points
    $$\{(x,t)\in Q_{2\rho}:\partial_{e^*} u_\epsilon(x_0+x,t_0+t)\leq 1+\delta\}.$$
    For the complementary subset
        $$\{(x,t)\in Q_{2\rho}: \partial_{e^*} u_\epsilon(x_0+x,t_0+t)> 1+\delta\},$$
    we recall the set inclusion
    $$ \{(x,t)\in Q_{2\rho}: \partial_{e^*} u_\epsilon(x_0+x,t_0+t)> 1+\delta\} \subset \{(x,t)\in Q_{2\rho}: |Du_\epsilon(x_0+x,t_0+t)|_E> 1+\delta\} $$
    and apply Lemma~\ref{bilinearelliptic} to obtain the estimate~\eqref{degiorgiellipticity}. The upper bound follows from the construction of~$\hat{\F}$, the bounds~\eqref{derivativebounds}, and the fact that~$\epsilon\leq 1$. Moreover, in the last term in~\eqref{degiorgieins} and~\eqref{degiorgizwei} involving the measure of the level sets, we exploited the identity
    $$\frac{n+\sigma}{n+2+\sigma}=1-\frac{2}{n+2+\sigma}=1-\frac{2}{n+2} + \frac{2\beta}{n+2}.$$
    In order to abbreviate this passage, we refrain from giving the details and instead refer the reader to~\cite[Chapter~10.1]{DiBenedetto2009} and also to~\cite[Proposition~7.1]{degeneratesystems}.
\end{proof}


Our final result in this section establishes the regularity~$v_\epsilon\in C^0_{\loc}\big(\Lambda_\rho ;L^2_{\loc}(B_\rho)\big)$. This property will be crucial for the discussion in Section~\ref{sec:nondegenerate}. In particular, we will exploit this fact in order to establish Proposition~\ref{lowerbound}, stating a lower bound for~$|Du_\epsilon|_E$ on~$Q_{\frac{\rho}{2}}(z_0)\Subset Q_{2\rho}(z_0)$. The proof proceeds similarly as in \cite[Chapter VIII, Section 3, Pages 230 -- 231]{dibenedetto1993degenerate}


\begin{mylem} \label{lem:stetigkeitinzeit} 
     Let~$\delta,\epsilon\in(0,1]$ and~$e^*\in \partial E^*$. Then, the function
     $$ v_\epsilon = (\partial_{e^*}u_\epsilon-(1+\delta))^2_+ $$
     satisfies
     \begin{equation} \label{est:stetigkeitinzeit}
     v_\epsilon \in C^0_{\loc} \big(\Lambda_{2\rho}(t_0);L^2_{\loc}(B_{2\rho}(x_0)) \big). 
     \end{equation}
\end{mylem}

\begin{proof}
We recall the structure conditions~\eqref{Gapprox} -- \eqref{quadraticgrowthohnekonst} for~$\hat{\F}_\epsilon$. Similarly as before, we abbreviate notation by writing~$u=u_\epsilon$ and $v=v_\epsilon$ and test the weak form~\eqref{weakformapprox} with the test function~$\partial_{e^*}\phi_{e^*}$, where~$\phi_{e^*} = 4v^{\frac32}\zeta=4(\partial_{e^*} u-(1+\delta))^3_+ \zeta $. Here,~$\zeta \in W^{1,\infty}_0(Q_{2\rho }(z_0))$ denotes an arbitrary cut-off function. Up to an approximation argument, this test function is admissible. 
Subsequently, we integrate by parts, leading to the identity 
\begin{align*} 
     \iint_{Q_{2\rho} (z_0)} & \partial_t[\partial_{e^*}u] \phi_{e^*} \,\dx\dt + \iint_{Q_{2\rho} (z_0)} \langle \partial_{e^*} [\hat{\h}_\epsilon(x,t,Du)], D\phi_{e^*}\rangle \,\dx\dt = - \iint_{Q_{2\rho} (z_0)} f \partial_{e^*}\phi_{e^*}\,\dx\dt. 
\end{align*}
This procedure can once again be justified via Steklov-averages. In the preceding integral identity, we now choose~$\zeta\in W^{1,\infty}_0(Q_{2\rho }(z_0))=\psi^2(x)\eta(t)$, such that~$\zeta$ consists of a standard smooth spatial cut-off function~$\psi\in C^1_0(B_{2\rho}(x_0))$ and a Lipschitz continuous cut-off function in time~$\eta\in W^{1,\infty}_0(\R)$ that satisfies~$\eta\equiv 0$ on~$(-\infty,\tau_1-\delta)$,~$\eta \equiv 1$ on~$(\tau_1,\tau_2)$,~$\eta \equiv 0$ on~$(\tau_2+\delta,\infty)$, and interpolates linearly on~$(\tau_1-\delta,\tau_1)$ and on~$(\tau_2,\tau_2+\delta)$. Here, we assume~$t_0-4\rho^2<\tau_1-\delta<\tau_1<\tau_2<\tau_2+\delta<t_0$ for an arbitrarily chosen parameter~$\delta>0$ small enough. In a similar manner as in the proof of Lemma~\ref{subsollemma}, by integrating by parts in the term involving the time derivative and passing to the limit~$\delta\downarrow 0$, we obtain
 \begin{align} \label{est:stetiginzeiteins}
     \int_{B_{2\rho}(x_0)\times\{\tau_2\}} & v^2 \psi^2\,\dx - \int_{B_{2\rho}(x_0)\times\{\tau_1\}} v^2 \psi^2\,\dx \\
     &= - 4 \iint_{B_{2\rho}(x_0)\times(\tau_1,\tau_2)} \langle \partial_{e^*} [\hat{\h}_\epsilon(x,t,Du)], D[(\partial_{e^*} u-(1+\delta))^3_+ \psi^2]\rangle \,\dx\dt \nonumber \\
     & \quad - 4 \iint_{B_{2\rho} (x_0)\times(\tau_1,\tau_2)} f \partial_{e^*}[(\partial_{e^*} u-(1+\delta))^3_+  \psi^2]\,\dx\dt \nonumber \\
     & \eqqcolon \foo{I} + \foo{II}. \nonumber
 \end{align}
We now argue that both quantities~$\foo{I}$ and~$\foo{II}$ are indeed finite. For the first term~$\foo{I}$, this follows from the fact that the quantity~$ \partial_{e^*} [\hat{\h}_\epsilon(x,t,Du)] $ belongs to~$L^2_{\loc}(Q_{2\rho}(z_0) \cap\{|Du|_E u>1+\delta\} )$, which can be inferred similarly as in Lemma~\ref{subsollemma} by exploiting the fact that the integrand is compactly supported within a subset of~$Q_{2\rho}(z_0)$ where~$|Du|_E\geq 1+\delta$, due to the set inclusion
$$ \{ (x,t)\in Q_{2\rho}(z_0): \partial_{e^*}u \geq 1+\delta \} \subset \{ (x,t)\in Q_{2\rho}(z_0): |Du|_E \geq 1+\delta \} ,$$ 
and by using the bounds~\eqref{derivativebounds},~\eqref{hatflipschitz},~\eqref{schrankeeins},~\eqref{hesseschrankedelta}, and also the fact that~$u_\epsilon$ admits second order weak spatial derivatives that belong to~$L^2_{\loc}(Q_{2\rho}(z_0))$ according to Lemma~\ref{approxregularityeins}. The reasoning for the second quantity~$\foo{II}$ is similar, where we recall that~$f\in L^2(\Omega_T)$. Therefore, we take absolute values in~\eqref{est:stetiginzeiteins} and pass to the limit~$|\tau_2-\tau_1|\downarrow 0$ on both sides of~\eqref{est:stetiginzeiteins} to obtain 
\begin{align*}
    \limsup\limits_{|\tau_2-\tau_1| \downarrow 0} \Bigg| \int_{B_{2\rho}(x_0)\times\{\tau_2\}} & v^2 \psi^2\,\dx - \int_{B_{2\rho}(x_0)\times\{\tau_1\}} v^2 \psi^2\,\dx \Bigg| = 0
\end{align*}
for any smooth spatial cut-off function~$\psi\in C^1_0(B_{2\rho}(x_0))$ and any~$t_0-4\rho^2<\tau_1<\tau_2<t_0$. The claimed assertion~\eqref{est:stetiginzeiteins} now follows verbatim from the arguments provided by DiBenedetto in~\cite[Chapter VIII, Section 3, Pages 230 -- 231]{dibenedetto1993degenerate}.
\end{proof}


\section{The non-degenerate regime} \label{sec:nondegenerate}
In this section, we establish the proof of Proposition~\ref{nondegenerateproposition}. Roughly speaking, the measure-theoretic information~\eqref{nondegeneratemeascond} allows to deduce a lower bound for~$|Du_\epsilon|_E$ on the smaller cylinder~$Q_{\frac{\rho}{2}}(z_0)$, provided the parameter~$\nu\in(0,\frac{1}{4}]$ incorporated in the measure-theoretic information~\eqref{nondegeneratemeascond} is chosen sufficiently small in dependence on the given data. By differentiating the weak form of~\eqref{weakformapprox}, we subsequently obtain a solution to a linear parabolic equation with elliptic and bounded coefficients, which further enables us to exploit a quantitative excess-decay estimate for parabolic De Giorgi classes that is taken from~\cite{dibenedetto2023parabolic}. We recall that in the course of this section, the set of assumptions~\eqref{schrankeeins},~\eqref{schrankezwei},~\eqref{deltamu}, 
and~\eqref{nondegeneratemeascond} is at our disposal.\,\\


We begin this section with the following expedient proposition, where we exploit the measure-theoretic information~\eqref{nondegeneratemeascond}. This allows us to derive a lower bound for~$|Du_\epsilon|_E $ on the smaller cylinder~$Q_{\frac{\rho}{2}}(z_0)$. 


\begin{myproposition} \label{lowerboundprop}
Let~$\delta,\epsilon\in(0,1]$ and~$u_\epsilon$ denote the weak solution to~\eqref{approx}. Then, there exists a parameter~$\nu=\nu(C_\F(\delta),\hat{C}_\F,\delta,\|f\|_{L^{n+2+\sigma}(Q_R)},M,N,n,R,R_E,r_E,\sigma) \in(0,\frac{1}{4}]$, such that if there holds
\begin{equation} \label{lowerboundvs}
   | Q_\rho(z_0)  \cap  \{ \partial_{e^*} u_\epsilon - (1+\delta) \leq (1-\nu) \mu \}|  \leq \nu |Q_\rho(z_0)|,
\end{equation}
for at least one~$e^*\in\partial E^*$, then there holds
\begin{equation} \label{lowerboundest}
   |D u_\epsilon|_E \geq 1+\delta+\frac{\mu}{4} \qquad\mbox{a.e. in~$Q_{\frac{\rho}{2}}(z_0)$}.
\end{equation}
\end{myproposition}


\begin{proof}
To abbreviate notation, we shall simply write~$u=u_\epsilon$ in what follows. The proof consists of three major steps and follows the approach taken in~\cite[Proposition~4.1]{elliptisch}, which is inspired by the work of Kuusi and Mingione in~\cite[Proof of Proposition 3.7]{kuusi2013gradient}. \,\\
\textbf{Step 1:} We commence by rescaling the solution~$u$ to the unit cylinder. For this matter, we introduce the rescaled mappings
\begin{align*}
\begin{cases}
    \Tilde{u}(x,t) \coloneqq \frac{u(x_0+\rho x,t_0+\rho^2t)}{\rho}  &  \mbox{for~$(x,t)\in Q_1$}, \\
    \Tilde{\F}_\epsilon(x,t,\xi) \coloneqq \hat{\F}_\epsilon(x_0+\rho x,t_0+\rho^2 t,\xi)  &  \mbox{for~$(x,t)\in Q_{1}, \xi\in\R^n$}, \\
    \Tilde{\mathcal{B}}_\epsilon(x,t,\xi) \coloneqq \hat{\mathcal{B}}_\epsilon(x_0+\rho x,t_0+\rho^2 t,\xi)  &  \mbox{for~$(x,t) \in Q_{1}, \xi\in\R^n$}, \\
    \Tilde{f} (x,t) \coloneqq \rho f(x_0+\rho x,t_0+\rho^2 t) & \mbox{for~$(x,t)\in Q_{1}$},
    \end{cases}
\end{align*}
where we set~$Q_1 \coloneqq B_{1}\times(-1,0]$. In the following we will use the abbreviation~$\Tilde{\h}_\epsilon(x,t,\xi) \coloneqq \nabla\Tilde{\F}_\epsilon(x,t,\xi)$. As~$u$ is a weak solution to~\eqref{approx} in~$Q_\rho(z_0)$, the rescaled mapping~$\Tilde{u}$ is a weak solution to
\begin{equation*}
    \partial_t \Tilde{u} - \divv \Tilde{\h}_\epsilon(x,D\Tilde{u}) = \Tilde{f} \qquad \mbox{in~$Q_{1}$}.
\end{equation*}
Moreover, the measure-theoretic assumption~\eqref{lowerboundvs} translates to
\begin{align} \label{lowerboundscaled}
   |\{ \partial_{\Tilde{e}^*} \Tilde{u} - (1+\delta) \leq (1-\nu)\mu  \}| \leq \nu |Q_1|
\end{align}
on~$Q_1$, where~$\Tilde{e}^* \in \partial E^*$ denotes one element that satisfies this measure condition. \,\\
\textbf{Step 2:} Next, we aim to establish a Caccioppoli-type estimate for~$\Tilde{u}$. Let~$\kappa\in(1+\delta+\frac{\mu}{4},1+\delta+\mu)$ and also~$e^*\in\partial E^*$ be chosen arbitrarily. In a similar manner to Lemma~\ref{subsollemma}, we abbreviate~$a=1+\delta$ and also~$\Tilde{\kappa}=1+\delta+\frac{\mu}{8}$. We test the weak form~\eqref{weakformapprox} with~$\partial_{e^*}\phi_{e^*}$, where~$\phi_{e^*}=\zeta \partial_{e^*} \Tilde{u}
 ( \partial_{e^*} \Tilde{u} - a)_+ ( \partial_{e^*} \Tilde{u} - \kappa)_- .$ Here,~$\zeta \in W^{1,\infty}_0(Q_1,[0,1])$ denotes an arbitrary cut-off function, whose choice can be justified by an approximation argument. Integrating by parts, we obtain
\begin{equation} \label{lowerbound}
     \iint_{Q_1} \big(\partial_t [\partial_{e^*}\Tilde{u}] \phi_{e^*}  + \langle \partial_{e^*}[\Tilde{\h}_\epsilon(x,t,D\Tilde{u})], D\phi_{e^*}\rangle \big)  \,\dx\dt = - \iint_{Q_1} \Tilde{f} \partial_{e^*}\phi_{e^*}\,\dx\dt.
\end{equation}
As before, this test function is admissible up to a Steklov-average procedure as in Lemma~\ref{energybound}. We now choose~$\zeta \in W^{1,\infty}_0(Q_1,[0,1])$ as~$\zeta(x,t) \coloneqq \psi^2(x) \eta^2(t) \omega(t)$, where~$\psi\in C^1_0(B_1,[0,1])$ denotes an arbitrary smooth and non-negative spatial cut-off function and~$\eta,\omega\in W^{1,\infty}(\R,[0,1])$ denote Lipschitz continuous cut-off functions in time with~$\eta(-1)=0$,~$\partial_t \eta \geq 0$ and~$\omega(0)=0$,~$\partial_t \omega \leq 0$. In particular, the cut-off function~$\omega(t)$ will be specified later on. Then, there holds
\begin{align*}
    \partial_j \phi_{e^*} &= \partial_j \zeta  (\partial_{e^*} \Tilde{u} - a)_+ (\partial_{e^*} \Tilde{u} -\kappa)_- + \zeta \partial_j \partial_{e^*} \Tilde{u} (\partial_{e^*} \Tilde{u} - a)_+ (\partial_{e^*} \Tilde{u} -\kappa)_- \\
    &\quad + \zeta \partial_{e^*}\Tilde{u} \partial_{j}\partial_{e^*}\Tilde{u}   \big( (\partial_{e^*} \Tilde{u}-\kappa)_-\bigchi_{\{\partial_{e^*} \Tilde{u}>a\}} + (\partial_{e^*} \Tilde{u}-a)_+\bigchi_{\{ \partial_{e^*} \Tilde{u}< \kappa \}} \big) 
\end{align*} 
a.e. in~$Q_1$, for any~$j=1,\ldots,n$. We begin by treating the term involving the time derivative in~\eqref{lowerbound}, where we obtain
\begin{align*}
   \iint_{Q_1} \partial_t[\partial_{e^*} \Tilde{u}] \phi_{e^*} \,\dx\dt =  \frac{1}{2} \iint_{Q_1} \partial_t(\partial_{e^*}\Tilde{u})^2 ( \partial_{e^*} \Tilde{u} - a)_+ ( \partial_{e^*} \Tilde{u} - \kappa)_- \zeta \,\dx\dt.
\end{align*}
An integration by parts thus leads to
\begin{align} \label{lowerboundtime}
  \iint_{Q_1} & \partial_t[\partial_{e^*} \Tilde{u}] \phi_{e^*} \,\dx\dt \\ 
    &= \frac{1}{2} \iint_{Q_1} \partial_t \Bigg(\int_{0}^{\partial_{e^*} \Tilde{u}^2}(\sqrt{s}-a)_+ (\sqrt{s}-\kappa)_- \,\ds \Bigg) \bigchi_{\{ a<\partial_{e^*} \Tilde{u}<\kappa \}} \zeta \,\dx\dt \nonumber \\
    &=  -\iint_{Q_1} \Bigg(\int_{0}^{\partial_{e^*} \Tilde{u}} s(s-a)_+ (s-\kappa)_-  \,\ds \Bigg) \bigchi_{\{ a<\partial_{e^*} \Tilde{u}<\kappa \}} \eta \partial_t \eta\, \omega \psi^2 \,\dx\dt \nonumber \\
    & \quad - \frac{1}{2} \iint_{Q_1} \Bigg(\int_{0}^{\partial_{e^*} \Tilde{u}} s(s-a)_+ (s-\kappa)_-  \,\ds \Bigg) \bigchi_{\{ a<\partial_{e^*} \Tilde{u}<\kappa \}} \partial_t \omega\, \eta^2 \psi^2 \,\dx\dt \nonumber.
\end{align}
We will now treat the inner one-dimensional integrals in both integral quantities on the right-hand side of the preceding identity. The integral in the first term on the right-hand side above is bounded further above, while the second term is bounded below. We begin with the treatment of the first quantity, where we exploit the fact that~$\partial_{e^*} \Tilde{u} \leq  |D \Tilde{u}|_E \leq  \frac{1}{r_E}\|D\Tilde{u}\|_{L^\infty(Q_1)}$ a.e. in~$Q_1$, which follows from~\eqref{minkowskialternativ} and~\eqref{betragminkowski}. In combination with~\eqref{schrankezwei}, this leads to
\begin{align*}
    \int_{0}^{\partial_{e^*} \Tilde{u}} s(s-a)_+ (s-\kappa)_-  \,\ds &\leq \int_{0}^{\frac{1}{r_E}\|D\Tilde{u}\|_{L^\infty(Q_1)}} s^3  \,\ds \\
    &\leq C(r_E)\|D\Tilde{u}\|^4_{L^\infty(Q_1)} \\
    &\leq C(M,r_E) 
\end{align*} 
for a.e.~$(x,t)\in Q_1 \cap \{ a<\partial_{e^*} \Tilde{u}<\kappa \}$. The second quantity is for a.e.~$(x,t)\in Q_1 \cap \{ a<\partial_{e^*} \Tilde{u}<\Tilde{\kappa} \}$ bounded below by exploiting the assumption~\eqref{deltamu} and our choices of~$a,\kappa,\Tilde{\kappa}$ as follows
\begin{align*}
    \int_{0}^{\partial_{e^*} \Tilde{u}} s(s-a)_+ (s-\kappa)_-  \,\ds  &\geq \int_{a}^{\partial_{e^*} \Tilde{u}} s(s-a)_+ (s-\kappa)_-  \,\ds  \\
    &\geq C(\delta) \int_{a}^{\partial_{e^*} \Tilde{u}} (s-a)_+  \,\ds  \\
    &= C(\delta) (\partial_{e^*} \Tilde{u}-a)^2_+ .
\end{align*}
Next, we turn our attention to the diffusion term in~\eqref{lowerbound}. There, we calculate
\begin{align*}
    \iint_{Q_1} &\langle \partial_{e^*} [\Tilde{\h}_\epsilon(x,t,D\Tilde{u})],D\phi_{e^*} \rangle\,\dx\dt \\
    &= \iint_{Q_1} \langle\Tilde{\B}_\epsilon(x,t,D\Tilde{u})(\partial_{e^*}D\Tilde{u}, D\phi_{e^*} \rangle\,\dx\dt + \iint_{Q_1} \langle\partial_{e^*}\Tilde{\h}_\epsilon
    (x,t,D\Tilde{u}), D\phi_{e^*} \rangle\,\dx\dt \\
    &\eqqcolon \foo{I} + \foo{II}. 
\end{align*}
We begin by treating the first term~$\foo{I}$, which yields
\begin{align*}
    \foo{I} &= \iint_{Q_1} \Tilde{\B}_\epsilon(x,t,D\Tilde{u})(\partial_{e^*} D\Tilde{u},D\zeta) \partial_{e^*}\Tilde{u} (\partial_{e^*}\Tilde{u} - a)_+ (\partial_{e^*}\Tilde{u} -\kappa)_- \,\dx\dt \\
    &\quad + \iint_{Q_1} \zeta \Tilde{\B}_\epsilon(x,t,D\Tilde{u})(\partial_{e^*}D\Tilde{u},\partial_{e^*}D\Tilde{u})(\partial_{e^*}\Tilde{u} - a)_+ (\partial_{e^*}\Tilde{u} -\kappa)_- \,\dx\dt \\
    &\quad + \iint_{Q_1} \zeta \Tilde{\B}_\epsilon(x,t,D\Tilde{u})(\partial_{e^*}D\Tilde{u},\partial_{e^*}\Tilde{u}) \partial_{e^*}\Tilde{u}  (\partial_{e^*}\Tilde{u} -\kappa)_- \bigchi_{\{\partial_{e^*}\Tilde{u}>a \}} \,\dx\dt \\
    &\quad + \iint_{Q_1} \zeta \Tilde{\B}_\epsilon(x,t,D\Tilde{u})(\partial_{e^*}D\Tilde{u},\partial_{e^*}D\Tilde{u}) \partial_{e^*}\Tilde{u} (\partial_{e^*}\Tilde{u} -a)_+ \bigchi_{\{\partial_{e^*}\Tilde{u}<\kappa \}} \,\dx\dt.
\end{align*}
Due to the set inclusion
 $$ \{ (x,t)\in Q_{1}: \partial_{e^*}\Tilde{u} \geq 1+\delta \} \subset \{ (x,t)\in Q_{1}: |D\Tilde{u}|_E \geq 1+\delta \}, $$
 which holds true for any~$e^*\in\partial E^*$, we employ Lemma~\ref{bilinearelliptic}, where we recall~\eqref{hatfvoraussetzung}, to obtain the lower bound
\begin{align} \label{lowerboundtermeins}
    \foo{I} &\geq - C \iint_{Q_1}  |\partial_{e^*}D\Tilde{u}| |D\zeta| \partial_{e^*}\Tilde{u} (\partial_{e^*}\Tilde{u} - a)_+ (\partial_{e^*}\Tilde{u} -\kappa)_- \,\dx\dt \\
    &\quad + (\epsilon+C_1) \iint_{Q_1} \zeta |\partial_{e^*}D\Tilde{u}|^2 (\partial_{e^*}\Tilde{u} - a)_+ (\partial_{e^*}\Tilde{u} -\kappa)_- \,\dx\dt \nonumber \\
    &\quad + (\epsilon + C_2) \iint_{Q_1} \zeta |\partial_{e^*}D\Tilde{u}|^2 \partial_{e^*}\Tilde{u} \big( (\partial_{e^*}\Tilde{u} -\kappa)_- \bigchi_{\{\partial_{e^*}\Tilde{u}>a \}} + (\partial_{e^*}\Tilde{u} -a)_+ \bigchi_{\{\partial_{e^*}\Tilde{u}<\kappa \}} \big) \,\dx\dt, \nonumber
\end{align}
where~$C=C(C_\F(\delta))$, and~$C_1=C_1(C_\F(\delta),\delta),C_2=C_2(C_\F(\delta),\delta)$ denote some positive ellipticity constants. In turn, we used the definition of~$C_\F$ according to~\eqref{schrankehessian} and also~\eqref{hesseschrankedelta}. Next, we treat the second quantity~$\foo{II}$, which we estimate further above with the Cauchy-Schwarz inequality, as well as by exploiting the Lipschitz assumption~\eqref{hatflipschitz}. By additionally using the representation formula~\eqref{representation} and the bound~$|e^*|\leq R_E$, we obtain
\begin{align*}
    \foo{II} &= \iint_{Q_1} \langle\partial_{e^*}\Tilde{\h}_\epsilon
    (x,t,D\Tilde{u}), D\phi_{e^*} \rangle\,\dx\dt\\
    &\leq C \iint_{Q_1}|D\phi_{e^*}|\,\dx\dt \\
    &\leq C \iint_{Q_1}  |D\zeta| \partial_{e^*}\Tilde{u} (\partial_{e^*}\Tilde{u} - a)_+ (\partial_{e^*}\Tilde{u} -\kappa)_-\,\dx\dt  + C \iint_{Q_1} \zeta  |\partial_{e^*}D\Tilde{u}| (\partial_{e^*}\Tilde{u} - a)_+ (\partial_{e^*}\Tilde{u} -\kappa)_-\,\dx\dt \\
    &\quad + C \iint_{Q_1} \zeta  |\partial_{e^*}D\Tilde{u}| \partial_{e^*}\Tilde{u} \big( (\partial_{e^*}\Tilde{u} -\kappa)_- \bigchi_{\{\partial_{e^*}\Tilde{u}>a \}} + (\partial_{e^*}\Tilde{u} -a)_+ \bigchi_{\{\partial_{e^*}\Tilde{u}<\kappa \}} \big) \,\dx\dt
\end{align*}
with~$C=C(\hat{C}_\F,N,n,R_E,r_E)$. Applying Young's inequality, we further estimate
\begin{align} \label{lowerboundtermzwei}
     \foo{II} &\leq C \iint_{Q_1}  |D\zeta| \partial_{e^*}\Tilde{u} (\partial_{e^*}\Tilde{u} - a)_+ (\partial_{e^*}\Tilde{u} -\kappa)_-\,\dx\dt \\
     &\quad + \mbox{$\frac{1}{4}$}C_1 \iint_{Q_1} \zeta  |\partial_{e^*}D\Tilde{u}|^2 (\partial_{e^*}\Tilde{u} - a)_+ (\partial_{e^*}\Tilde{u} -\kappa)_-\,\dx\dt + C \iint_{Q_1} \zeta (\partial_{e^*}\Tilde{u} - a)_+ (\partial_{e^*}\Tilde{u} -\kappa)_- \,\dx\dt \nonumber \\
    &\quad + \mbox{$\frac{1}{4}$}C_2 \iint_{Q_1} \zeta  |\partial_{e^*}D\Tilde{u}|^2 \partial_{e^*}\Tilde{u} \big( (\partial_{e^*}\Tilde{u} -\kappa)_- \bigchi_{\{\partial_{e^*}\Tilde{u}>a \}} + (\partial_{e^*}\Tilde{u} -a)_+ \bigchi_{\{\partial_{e^*}\Tilde{u}<\kappa \}} \big) \,\dx\dt \nonumber \\
    &\quad + C \iint_{Q_1} \zeta  \partial_{e^*}\Tilde{u} \big( (\partial_{e^*}\Tilde{u} -\kappa)_- \bigchi_{\{\partial_{e^*}\Tilde{u}>a \}} + (\partial_{e^*}\Tilde{u} -a)_+ \bigchi_{\{\partial_{e^*}\Tilde{u}<\kappa \}} \big) \,\dx\dt \nonumber 
\end{align}
with~$C=C(C_\F(\delta),\hat{C}_\F,\delta,N,n,R_E,r_E)$. Finally, we turn our attention to the term involving the datum~$\Tilde{f}$ in~\eqref{lowerbound}. In a similar way to the treatment of the second term~$\foo{II}$, we estimate
\begin{align*}
    \foo{III} &\eqqcolon \iint_{Q_1} \Tilde{f} \partial_{e^*}\phi_{e^*}\,\dx\dt \\
    &\leq \iint_{Q_1}  |\Tilde{f}| |D\zeta| \partial_{e^*}\Tilde{u} (\partial_{e^*}\Tilde{u} - a)_+ (\partial_{e^*}\Tilde{u} - \kappa)_- \,\dx\dt \\
    &\quad + \iint_{Q_1} \zeta |\Tilde{f}| |\partial_{e^*}D\Tilde{u}| (\partial_{e^*}\Tilde{u} - a)_+ (\partial_{e^*}\Tilde{u} - \kappa)_- \,\dx\dt \\
    &\quad + \iint_{Q_1} \zeta |\Tilde{f}| |\partial_{e^*}D\Tilde{u}| \partial_{e^*}\Tilde{u} \big( (\partial_{e^*}\Tilde{u} -\kappa)_- \bigchi_{\{\partial_{e^*}\Tilde{u}>a \}} + (\partial_{e^*}\Tilde{u} -a)_+ \bigchi_{\{\partial_{e^*}\Tilde{u}<\kappa \}} \big) \,\dx\dt.
\end{align*}
Similarly to before, we now employ Young's inequality to obtain
\begin{align} \label{lowerbounddatumsummed}
    \foo{III} &\leq C \iint_{Q_1}  |\Tilde{f}| |D\zeta| \partial_{e^*}\Tilde{u} (\partial_{e^*}\Tilde{u} - a)_+ (\partial_{e^*}\Tilde{u} - \kappa)_- \,\dx\dt  \\
    &\quad + \mbox{$\frac{1}{4}$} C_1 \iint_{Q_1} \zeta |\partial_{e^*}D\Tilde{u}|^2 (\partial_{e^*}\Tilde{u} - a)_+ (\partial_{e^*}\Tilde{u} - \kappa)_- \,\dx\dt \nonumber \\
    &\quad + C \iint_{Q_1} \zeta |\Tilde{f}|^2 (\partial_{e^*}\Tilde{u} - a)_+ (\partial_{e^*}\Tilde{u} - \kappa)_- \,\dx\dt \nonumber \\
    &\quad + \mbox{$\frac{1}{4}$}C_2 \iint_{Q_1} \zeta |\partial_{e^*}D\Tilde{u}|^2 \partial_{e^*}\Tilde{u} \big( (\partial_{e^*}\Tilde{u} -\kappa)_- \bigchi_{\{\partial_{e^*}\Tilde{u}>a \}} + (\partial_{e^*}\Tilde{u} -a)_+ \bigchi_{\{\partial_{e^*}\Tilde{u}<\kappa \}} \big) \,\dx\dt \nonumber \\
    &\quad + C \iint_{Q_1} \zeta |\Tilde{f}|^2 \partial_{e^*}\Tilde{u} \big( (\partial_{e^*}\Tilde{u} -\kappa)_- \bigchi_{\{\partial_{e^*}\Tilde{u}>a \}} + (\partial_{e^*}\Tilde{u} -a)_+ \bigchi_{\{\partial_{e^*}\Tilde{u}<\kappa \}} \big) \,\dx\dt \nonumber
\end{align}
with constant~$C=C(C_\F(\delta),\delta,N,n,r_E)$. Passing to the limit~$h\to 0$ in the quantity arising from the term involving the time derivative and combining our results~\eqref{lowerboundtermeins},~\eqref{lowerboundtermzwei}, and~\eqref{lowerbounddatumsummed}, we achieve
\begin{align} \label{lowerboundenergy}
&- C(\delta) \iint_{Q_1} (\partial_{e^*}\Tilde{u} - a)^2_+ \bigchi_{\{a<\partial_{e^*}\Tilde{u}<\Tilde{\kappa} \}} \partial_t \omega\, \eta^2 \psi^2 \,\dx\dt \\
    &\quad +(\epsilon+C_1) \iint_{Q_1} \psi^2\eta^2\omega |\partial_{e^*}D\Tilde{u}|^2  (\partial_{e^*}\Tilde{u}-a)_+ (\partial_{e^*}\Tilde{u}-\kappa)_- \,\dx\dt \nonumber \\
    &\quad+ (\epsilon+C_2) \iint_{Q_1} \psi^2\eta^2\omega |\partial_{e^*}D\Tilde{u}|^2 \partial_{e^*}\Tilde{u} \big( (\partial_{e^*}\Tilde{u}-\kappa)_- \bigchi_{\{ \partial_{e^*}\Tilde{u}>a \}} + (\partial_{e^*}\Tilde{u}-a)_+ \bigchi_{\{ \partial_{e^*}\Tilde{u}<\kappa \}} \big) \,\dx\dt \nonumber \\
    &\leq C \iint_{Q_1} \psi^2\eta^2\omega (1+|\Tilde{f}|^2) (\partial_{e^*}\Tilde{u}-a)_+ (\partial_{e^*}\Tilde{u}-\kappa)_- \,\dx\dt \nonumber \\
    &\quad + C \iint_{Q_1} \psi^2\eta^2\omega \partial_{e^*}\Tilde{u} (1+|\Tilde{f}|^2) \big( (\partial_{e^*}\Tilde{u}-\kappa)_- \bigchi_{\{ \partial_{e^*}\Tilde{u}>a \}} + (\partial_{e^*}\Tilde{u}-a)_+ \bigchi_{\{ \partial_{e^*}\Tilde{u}<\kappa \}} \big) \,\dx\dt \nonumber \\
    &\quad + C \iint_{Q_1} |\Tilde{f}| |D\psi| \psi \eta^2 \omega \partial_{e^*}\Tilde{u} (\partial_{e^*}\Tilde{u}-a)_+ (\partial_{e^*}\Tilde{u}-\kappa)_- \,\dx\dt \nonumber \\
    & \quad  + C \iint_{Q_1} |D\psi| \psi \eta^2 \omega (1+\partial_{e^*}\Tilde{u}) (\partial_{e^*}\Tilde{u} - a)_+ (\partial_{e^*}\Tilde{u} -\kappa)_- \,\dx\dt \nonumber \\
    &\quad + C \iint_{Q_1}  \bigchi_{\{ a<\partial_{e^*}\Tilde{u}<\kappa \}} \eta \partial_t \eta\, \omega \psi^2  \,\dx\dt + C \iint_{Q_1}  |\partial_{e^*}D\Tilde{u}| |D\zeta| \partial_{e^*}\Tilde{u} (\partial_{e^*}\Tilde{u} - a)_+ (\partial_{e^*}\Tilde{u} -\kappa)_- \,\dx\dt \nonumber    
\end{align}
with constants~$C_1=C_1(C_\F(\delta),\delta), C_2=C_2(C_\F(\delta),\delta)$ and~$C=C(C_\F(\delta),\hat{C}_\F,\delta,N,n,R_E,r_E)$. The very last quantity on the right-hand side of the preceding estimate~\eqref{lowerboundenergy} is estimated further above with an other application of Young's inequality, which yields
\begin{align*}
   C \iint_{Q_1} &  |\partial_{e^*}D\Tilde{u}| |D\zeta| \partial_{e^*}\Tilde{u} (\partial_{e^*}\Tilde{u} - a)_+ (\partial_{e^*}\Tilde{u} -\kappa)_- \,\dx\dt \\
   &\leq \mbox{$\frac{1}{2}$} C_1 \iint_{Q_1} \psi^2\eta^2\omega |\partial_{e^*} D\Tilde{u}|^2  (\partial_{e^*}\Tilde{u}-a)_+ (\partial_{e^*}\Tilde{u}-\kappa)_- \,\dx\dt\\
   &\quad + C \iint_{Q_1} |D\psi|^2 \eta^2\omega \partial_{e^*}\Tilde{u}^2 (\partial_{e^*}\Tilde{u}-a)_+ (\partial_{e^*}\Tilde{u}-\kappa)_-\,\dx\dt
\end{align*}
with constant~$C=C(C_\F(\delta),\hat{C}_\F,\delta,N,n,r_E)$. At this point, we specify the remaining cut-off function in time that is~$\omega(t)$. For an arbitrary~$\tau\in(-1,0)$ and~$\lambda\in(0,-\tau)$, let
\begin{align*}
     \omega(t) = \begin{cases}
        1  & \mbox{for~$t\in(-\infty,\tau]$ }, \\
        1-\frac{t-\tau}{\lambda}  & \mbox{for~$t\in(\tau,\tau+\lambda]$ }, \\
        0  & \mbox{for~$t\in(\tau+\lambda,\infty)$ }. 
    \end{cases}
\end{align*}
Overall, this leads us to the subsequent energy estimate
\begin{align*}
    & \mbox{ $ \frac{ 1}{\lambda}$}C(\delta) \iint_{B_1\times(\tau,\tau+\lambda]} (\partial_{e^*}\Tilde{u} - a)^2_+ \bigchi_{\{a<\partial_{e^*}\Tilde{u}<\Tilde{\kappa} \}} \eta^2 \psi^2 \,\dx\dt \\
    &\quad +(\epsilon+C_1) \iint_{Q_\tau} |\partial_{e^*}D\Tilde{u}|^2 \zeta^2 (\partial_{e^*}\Tilde{u}-a)_+ (\partial_{e^*}\Tilde{u}-\kappa)_- \,\dx\dt \nonumber \\
    &\quad+ (\epsilon+C_2) \iint_{Q_\tau} |\partial_{e^*}D\Tilde{u}|^2 \partial_{e^*}\Tilde{u} \zeta^2 \big( (\partial_{e^*}\Tilde{u}-\kappa)_- \bigchi_{\{ \partial_{e^*}\Tilde{u}>a \}} + (\partial_{e^*}\Tilde{u}-a)_+ \bigchi_{\{ \partial_{e^*}\Tilde{u}<\kappa \}} \big) \,\dx\dt \nonumber \\
    &\leq C \iint_{Q_1}  \eta^2 \omega  (\partial_{e^*}\Tilde{u}-a)_+ (\partial_{e^*}\Tilde{u}-\kappa)_- \big(\psi^2(1+|\Tilde{f}|^2) + |D\psi|^2 \partial_{e^*}\Tilde{u}^2 \big) \,\dx\dt \\
    &\quad +  C \iint_{Q_1}  \psi^2\eta^2\omega (1+|\Tilde{f}|^2) \partial_{e^*}\Tilde{u} \big( (\partial_{e^*}\Tilde{u}-\kappa)_- \bigchi_{\{ \partial_{e^*}\Tilde{u}>a \}} + (\partial_{e^*}\Tilde{u}-a)_+ \bigchi_{\{ \partial_{e^*}\Tilde{u}<\kappa \}} \big) \,\dx\dt \\
    &\quad + C \iint_{Q_1} |\Tilde{f}| |D\psi| \psi \eta^2 \omega \partial_{e^*}\Tilde{u} (\partial_{e^*}\Tilde{u}-a)_+ (\partial_{e^*}\Tilde{u}-\kappa)_- \,\dx\dt \\
    &\quad + C \iint_{Q_1}  \bigchi_{\{ a<\partial_{e^*}\Tilde{u}<\kappa \}} \eta \partial_t \eta\, \omega \psi^2  \,\dx\dt
\end{align*}
with~$C=C(C_\F(\delta),\hat{C}_\F,\delta,N,n,R_E,r_E)$, where we abbreviated notation by~$Q_\tau \coloneqq B_1\times(-1,\tau]$. The second quantity on the left-hand side involving second order weak derivatives of~$\Tilde{u}$ is discarded due to a non-negative contribution. In the third term, we exploit the lower bound~$\partial_{e^*}\Tilde{u}\geq 1+\delta$, which holds true whenever the integrand is positive. Moreover, the entire right-hand side is bounded further above by exploiting the assumptions~\eqref{schrankeeins} and~\eqref{schrankezwei} in combination with the bounds for~$\psi,\eta,\omega$. In the terms involving the datum~$\Tilde{f}$, we apply Hölder's inequality with exponents~$(n+2+\sigma,\frac{n+2+\sigma}{n+1+\sigma})$ respectively with $(\frac{n+2+\sigma}{2},\frac{n+2+\sigma}{n+\sigma})$. Passing to the limit~$\lambda\downarrow 0$ and taking the essential supremum with respect to~$\tau\in(-1,0)$, we thus obtain
\begin{align} \label{lowerboundee}
    &\esssup\limits_{t\in(-1,0)} \int_{B_1\times\{t\}} \big(\partial_{e^*}\Tilde{u}-(1+\delta) \big)^2_+ \bigchi_{\{1+\delta<\partial_{e^*}\Tilde{u}<\Tilde{\kappa} \}} \Tilde{\zeta} \,\dx \\
    &\quad + \iint_{Q_1} |\partial_{e^*}D\Tilde{u}|^2 \Tilde{\zeta} \big( (\partial_{e^*}\Tilde{u}-\kappa)_- \bigchi_{\{ \partial_{e^*}\Tilde{u}>1+\delta \}} + \big(\partial_{e^*}\Tilde{u}-(1+\delta)\big)_+ \bigchi_{\{ \partial_{e^*}\Tilde{u}<\kappa \}} \big) \,\dx\dt \nonumber \\
    &\leq C \big( \|D\Tilde{\zeta}\|_{L^\infty(Q_1)} + \|D\Tilde{\zeta}\|^2_{L^\infty(Q_1)} + \|\partial_t \Tilde{\zeta}\|_{L^\infty(Q_1)} \big) |Q_1 \cap\{1+\delta<\partial_{e^*}\Tilde{u}<\kappa\}|^{1-\frac{2}{n+2+\sigma}} \nonumber
\end{align}
with $C=C(C_\F(\delta),\hat{C}_\F,\delta,\|f\|_{L^{n+2+\sigma}(Q_R)},M,N,n,R,R_E,r_E,\sigma)$, where we also used
$$\|\Tilde{f}\|_{L^{n+2+\sigma}(Q_1)} \leq \rho^{\beta} \|f\|_{L^{n+2+\sigma}(Q_R)}\leq R^{\beta} \|f\|_{L^{n+2+\sigma}(Q_R)},$$
which is an immediate consequence of changing variables. Moreover, we abbreviated~$\Tilde{\zeta}\coloneqq\psi^2\eta^2$. In particular, the preceding estimate holds in fact true for any smooth cut-off function~$\zeta\in C^1(Q_1,[0,1])$ vanishing on the parabolic boundary of~$Q_1$, which follows from an approximation argument and the fact that both~$\psi,\eta$ were chosen arbitrarily. \,\\ 

\textbf{Step 3:} Our aim is to apply the geometric convergence Lemma~\ref{geometriclem} and exploit the Caccioppoli-type estimate from Step 2. Therefore, we consider the following sequences of levels
\begin{alignat*}{2}
  \delta_m &\coloneqq 1+\delta+\frac{\mu}{8}\Big(1-\frac{1}{2^{m}}\Big),\qquad && m\in\R_{\geq 0}, \\
    \kappa_m &\coloneqq 1+\delta+\frac{\mu}{4}\Big(1+\frac{1}{2^m}\Big),\qquad && m\in\R_{\geq 0}, 
\end{alignat*}
such that for any~$m\in\R_{\geq 0}$ we have
$$ 1+\delta \leq \delta_m \leq 1+\delta+\frac{\mu}{8} \leq 1+\delta+\frac{\mu}{4} \leq \kappa_m \leq 1+\delta+\frac{\mu}{2}.$$
The sequence of radii is given by
$$r_m \coloneqq \frac{1}{2}+\frac{1}{2^{m+1}},\qquad  m\in\N_0.$$
Next, we consider standard smooth cut-off functions~$\zeta_m \in C^1(Q_{r_m},[0,1])$ vanishing on the parabolic boundary of~$Q_1$ that 
satisfy the bounds
$$\zeta_m \equiv 1 \quad \mbox{on~$Q_{r_{m+1}}$} \qquad \& \qquad \|D\zeta_m\|_{L^\infty(B_{r_m})},\|\partial_t \zeta_m\|_{L^\infty(Q_1)} \leq 4^{m}$$
for any~$m\in\N_0$. We note that due to the rescaling procedure performed in Step 1 of the proof, the upper bounds for~$D\zeta_m$ and~$\partial_t \zeta_m$ are independent of the original radius~$\rho$. Moreover, we have
$$ \delta_{m+1} - \delta_{m} = \frac{\mu}{2^{m+4}} , \qquad 
\kappa_{m+\frac{1}{2}} - \kappa_{m+1} = \frac{\mu(\sqrt{2}-1)}{2^{m+3}}$$
for any~$m\in\R_{\geq 0}$.
Let us denote
$$Y_m \coloneqq |A_m| \coloneqq |Q_{r_m} \cap \{\delta_m < \partial_{e^*}\Tilde{u} <\kappa_m\}|$$
and further choose the test function~$\zeta = \zeta_m$ and levels~$\kappa=\kappa_m$ in the previous estimate~\eqref{lowerboundee} from Step 2, while the parameter~$\delta>0$ is replaced with~$\delta_m>0$. Additionally, we consider~$\delta_{m},\kappa_{m+\frac{1}{2}}$, where~$\delta_m$ takes the role of~$a=1+\delta$ in the estimate~\eqref{lowerboundee}, while ~$\kappa_{m+\frac{1}{2}}$ takes the role of~$\Tilde{\kappa}$ and~$\kappa_m$ the one of~$\kappa$ in~\eqref{lowerboundee}. In combination with the Poincar\'{e}-type inequality from Lemma~\ref{poincarelem}, we achieve
\begin{align*}
     \big(\delta_{m+1}  &- \delta_{m}\big)^2  \big(\kappa_{m+\frac{1}{2}}  - \kappa_{m+1}\big)^2   Y_{m+1} \\
    &\leq \big\| \big(\partial_{e^*}\Tilde{u} -\delta_{m}\big)_+ \big(\partial_{e^*}\Tilde{u}  - \kappa_{m+\frac{1}{2}}\big)_- \zeta_m \big\|^2_{L^2(A_{m+1})} \\
    &\leq  \big\| \big(\partial_{e^*}\Tilde{u} -\delta_{m}\big)_+ \big(\partial_{e^*}\Tilde{u}  - \kappa_{m+\frac{1}{2}}\big)_- \zeta_m \big\|^2_{L^2(Q_{r_m})}  \\
    &\leq C(n) |A_m|^{\frac{2}{n+2}} \bigg(\esssup\limits_{t\in(-r^2_m,0)}\int_{B_{r_m}} \big(\partial_{e^*}\Tilde{u}-\delta_{m} \big)^2_+ \big( \partial_{e^*}\Tilde{u}  -\kappa_{m+\frac{1}{2}} \big)^2_- \zeta^2_m \,\dx \\
    &\quad + \iint_{Q_{r_m}} |D[(\partial_{e^*}\Tilde{u} -\delta_m)_+ (\partial_{e^*}\Tilde{u} -\kappa_m)_- \zeta_m ]|^2\,\dx\dt \bigg) \\ 
    &\leq C 4^{m} |A_m|^{\frac{2}{n+2}} |A_m|^{1-\frac{2}{n+2+\sigma}} \\
    &= C 4^{m} Y^{1+\frac{2\beta}{n+2}}_m, 
\end{align*}
where~$C=C(C_\F(\delta),\hat{C}_\F,\delta,\|f\|_{L^{n+2+\sigma}(Q_R)},M,N,n,R,R_E,r_E,\sigma)$. In turn, we also used the fact that~$\delta_m<\delta_{m+1}$ as well as
~$\kappa_{m+1}<\kappa_{m+\frac{1}{2}}<\kappa_m$. Due to the identity for~$\delta_{m+1}-\delta_{m}$ and
~$\kappa_{m+\frac{1}{2}} - \kappa_{m+1}$ from above, and also the estimate~\eqref{deltamu}, there holds
$$Y_{m+1} \leq C 2^{6m+7} Y^{1+\frac{2\beta}{n+2}}_m \leq C 2^{6m} Y^{1+\frac{2\beta}{n+2}}_m$$
 for any~$m\in\N_0$ with~$C=C(C_\F(\delta),\hat{C}_\F,\delta,\|f\|_{L^{n+2+\sigma}(Q_R)},M,N,n,R,R_E,r_E,\sigma)$. At this point, we fix~$e^*\in\partial E^*$ through the specific choice of~$\Tilde{e}^*\in \partial E^*$, which denotes one such parameter that satisfies~\eqref{lowerboundscaled}. Finally, we apply the geometric convergence Lemma~\ref{geometriclem} with the choices~$b\coloneqq 2^6$ and~$\kappa \coloneqq \frac{2\beta}{n+2}$. The assumption of the lemma, which we recall is given by
$$Y_0 \leq C^{-\frac{1}{\kappa}}b^{-\frac{1}{\kappa2}},$$  
is satisfied, provided the parameter~$\Tilde{\nu}\in(0,\frac{1}{4}]$ is chosen small enough in dependence on the full  data~$(C_\F(\delta),\hat{C}_\F,\delta,\|f\|_{L^{n+2+\sigma}(Q_R)},N,n,R_E,M,R,r_E,\sigma)$, a matter of fact that can be inferred from the bound
\begin{align*}
    Y_0 &= |Q_1 \cap \{1+\delta < \partial_{\Tilde{e}^*}\Tilde{u} < 1+\delta+\mbox{$\frac{\mu}{2}$} \}| \\ 
    &\leq |Q_1 \cap \{ \partial_{\Tilde{e}^*}\Tilde{u}<1+\delta+\textstyle{\frac{\mu}{2}} \}| \\ 
    &\leq \nu|Q_1| 
\end{align*}
for~$\nu \in (0,\frac{1}{4}]$. An application of Lemma~\ref{geometriclem} thus yields the convergence~$Y_m \to 0$ as~$m\to\infty$, which implies
\begin{align} \label{est:meascondinterscaled}
    \big|Q_{\frac{1}{2}}\cap \{ 1+\delta+\mbox{$\frac{\mu}{8}$} \leq \partial_{\Tilde{e}^*}\Tilde{u} \leq 1+\delta+ \mbox{$\frac{\mu}{4}$} \} \big|=0,
\end{align} 
where we abbreviated~$Q_{\frac{1}{2}}\coloneqq B_{\frac{1}{2}}\times(-\frac{1}{4},0]$. By denoting
$$ v \eqqcolon \frac{(\partial_{\Tilde{e}^*}\Tilde{u}-(1+\delta))^2_+}{\mu^2}, $$
the preceding measure-theoretic information can be rewritten as
\begin{align*} 
    | Q_{\frac{1}{2}}\cap \{ \mbox{$\frac{1}{64}$} \leq v \leq \mbox{$\frac{1}{16}$} \} |=0.
\end{align*}
Next, we apply Lemma~\ref{isoperimetric} for~$\partial_{\Tilde{e}^*}\Tilde{u}$ with parameters~$l=1+\delta+\frac{\mu}{8}<1+\delta+\frac{\mu}{4}=m$ to obtain for a.e.~$t\in(-\mbox{$\frac{1}{4}$},0]$ the following 
\begin{align*}
    \mbox{$\frac{\mu}{8}$} |B_{\frac{1}{2}} & \cap \{ \partial_{\Tilde{e}^*}\Tilde{u}(\cdot,t) < 1+\delta+\mbox{$\frac{\mu}{8}$} \}| |B_{\frac{1}{2}} \cap \{ \partial_{\Tilde{e}^*}\Tilde{u}(\cdot,t) > 1+\delta+\mbox{$\frac{\mu}{4}$} \}| \\
    & \leq C(n) \int_{B_{\frac{1}{2}} \cap \{1+\delta+\frac{\mu}{8} < \partial_{\Tilde{e}^*}\Tilde{u}(\cdot,t) < 1+\delta+\frac{\mu}{4} \}}|\partial_{\Tilde{e}^*}D\Tilde{u}|\,\dx.
\end{align*}
After integrating the preceding inequality with respect to~$t\in(-\mbox{$\frac{1}{4}$},0]$ and exploiting the fact~\eqref{est:meascondinterscaled}, there holds
\begin{align*}
     \int_{(\mbox{$\frac{1}{4}$},0]} |B_{\frac{1}{2}} & \cap \{ \partial_{\Tilde{e}^*}\Tilde{u}(\cdot,t) < 1+\delta+\mbox{$\frac{\mu}{8}$} \}| |B_{\frac{1}{2}} \cap \{ \partial_{\Tilde{e}^*}\Tilde{u}(\cdot,t) > 1+\delta+\mbox{$\frac{\mu}{4}$} \}| \,\dt \\
     &\leq \mbox{$\frac{C(n)}{\mu}$} \iint_{Q_{\frac{1}{2}} \cap \{1+\delta+\frac{\mu}{8} < \partial_{\Tilde{e}^*}\Tilde{u} < 1+\delta+\frac{\mu}{4} \}}|\partial_{\Tilde{e}^*}D\Tilde{u}|\,\dx \dt \\
     &= 0. 
\end{align*}
Consequently, in combination with~\eqref{est:meascondinterscaled}, for a.e.~$t\in(-\mbox{$\frac{1}{4}$},0]$ there holds
\begin{align} \label{time-stripes}
    \partial_{\Tilde{e}^*}\Tilde{u}(\cdot,t) \leq 1+\delta+\frac{\mu}{8} \quad\mbox{a.e. in~$B_{\frac{1}{2}}$} \qquad\mbox{or} \qquad \partial_{\Tilde{e}^*}\Tilde{u}(\cdot,t) \geq 1+\delta+\frac{\mu}{4} \quad\mbox{a.e. in~$B_{\frac{1}{2}}$}. 
\end{align}
We now conclude the proof by arguing that the preceding measure-theoretic information implies that there holds~$\partial_{\Tilde{e}^*}\Tilde{u}\geq 1+\delta+\frac{\mu}{4}$ a.e. in~$Q_{\frac{1}{2}}$. Let us assume the contrary and suppose that~$\partial_{\Tilde{e}^*}\Tilde{u}\leq 1+\delta+\frac{\mu}{8}$ a.e. in some subset~$\Tilde{Q}\subsetneq Q_{\frac{1}{2}}$ with positive measure~$|\Tilde{Q}|>0$. We note that the case where~$\Tilde{Q}=Q_{\frac{1}{2}}$ can be excluded, given that the parameter~$\nu\in(0,\frac{1}{4}]$, which has been chosen above for the derivation of~\eqref{est:meascondinterscaled}, is possibly further reduced in dependence on the space dimension~$n$. Indeed, we may also assume that~$|\Tilde{Q}| \leq \frac{1}{2}|Q_{\frac{1}{2}}|$, such that~$|Q_{\frac{1}{2}}\setminus \Tilde{Q}|\geq \frac{1}{2}|Q_{\frac{1}{2}}|$, according to the very same reasoning just provided. At this point, we recall the expedient Lemma~\ref{lem:stetigkeitinzeit}, asserting that there holds 
\begin{align} \label{est:lowerboundstetiginzeit}
   v =(\partial_{e^*}\Tilde{u}-(1+\delta))^2_+
   \in C^0\big((-\mbox{$\frac{1}{4}$},0); L^2(B_{\frac{1}{2}})\big).
\end{align} 
We now consider a time-slice~$\Tilde{t}\in (-\mbox{$\frac{1}{4}$},0)$ for which there exist sequences~$(t_i)_{i\in\N}$ in~$Q_{\frac{1}{2}}\setminus \Tilde{Q}$ and~$(\Tilde{t}_i)_{i\in\N}$ in~$\Tilde{Q}$ with the property that~$t_i,\Tilde{t}_i\to\Tilde{t}$ as~$i\to\infty$. This is possible according to the assumption~$|Q_{\frac{1}{2}}\setminus \Tilde{Q}|\geq \frac{1}{2}|Q_{\frac{1}{2}}|>0$. Due to~\eqref{time-stripes}, we infer that there holds
$$ \limsup\limits_{|t_i- \Tilde{t}|\downarrow 0} \|v\|^2_{L^2(B_{\frac{1}{2}}\times\{t_i\})} \geq \mbox{$\frac{1}{16^2}$}, \qquad \liminf\limits_{|\Tilde{t}_i - \Tilde{t}|\downarrow 0} \|v\|^2_{L^2(B_{\frac{1}{2}}\times\{\Tilde{t}_i\})} \leq \mbox{$\frac{1}{64^2}$}
,$$
which further implies that at least one of the following must be true:
$$ \limsup\limits_{|t_i - \Tilde{t}|\downarrow 0} \Big | \|v\|^2_{L^2(B_{\frac{1}{2}}\times\{t_i\})} - \|v\|^2_{L^2(B_{\frac{1}{2}}\times\{\Tilde{t}\})} \Big| >0 $$
or
$$ \limsup\limits_{|\Tilde{t}_i- \Tilde{t}|\downarrow 0} \Big | \|v\|^2_{L^2(B_{\frac{1}{2}}\times\{\Tilde{t}_i\})} - \|v\|^2_{L^2(B_{\frac{1}{2}}\times\{\Tilde{t}\})} \Big| >0. $$
However, this contradicts the regularity property~\eqref{est:lowerboundstetiginzeit}. Therefore, such a subset~$\Tilde{Q}\subset Q_{\frac{1}{2}}$ with~$|\Tilde{Q}|>0$ cannot exist and there holds~$\partial_{\Tilde{e}^*}\Tilde{u}\geq 1+\delta+\frac{\mu}{4}$ a.e. in~$Q_{\frac{1}{2}}$. 
Finally, scaling back to the original solution~$u$ and also reverting to the abbreviation~$u=u_\epsilon$, we eventually obtain the claimed estimate~\eqref{lowerboundest} via
$$|Du_\epsilon|_E \geq \partial_{\Tilde{e}^*} u_\epsilon \geq 1+\delta+\frac{\mu}{4}\qquad\mbox{a.e. in~$Q_{\frac{\rho}{2}}(z_0)$},$$ 
where we used~\eqref{minkowskialternativ}. This finishes the proof of the proposition. 
\end{proof}


The lower bound~\eqref{lowerboundest} from the preceding Proposition~\ref{lowerboundprop} implies that for any~$i=1,\ldots,n$ the function~$D_i u_\epsilon$ is a weak solution to a linear elliptic equation on~$Q_{\frac{\rho}{2}}(z_0)$ whose ellipticity constant only depends on~$C_\F(\delta)$ and the parameter~$\delta\in(0,1]$, but is independent of~$\epsilon\in(0,1]$. This yields the subsequent quantitative~$L^\infty$- and oscillation estimate.


\begin{myproposition} \label{dibenedetto}
     Let~$\delta,\epsilon\in(0,1]$ and~$u_\epsilon$ denote the weak solution to~\eqref{approx}. Assume there exists~$\nu\in(0,\frac{1}{4}]$, such that~\eqref{lowerboundvs} holds true for at least one~$e^*\in \partial E^*$. Then, there exists a constant~$C\geq 1$ and a parameter~$\alpha\in(0,1)$, depending on
     \begin{align*}
         C &=C(C_\F(\delta),\hat{C}_\F,\delta,\|f\|_{L^{n+2+\sigma}(Q_R)},M,N,n,R,r_E,\sigma), \\
         \alpha&=\alpha(C_\F(\delta),\hat{C}_\F,\delta,\|f\|_{L^{n+2+\sigma}(Q_R)},M,N,n,r_E,\sigma),
     \end{align*}
     such that for any~$i\in\{1,\ldots,n\}$ there holds the quantitative local~$L^\infty$-estimate
    \begin{align} \label{excesssupbound}
    \esssup\limits_{Q_{\frac{\rho}{2}}(z_0)} |D_i u_\epsilon
    -(D_i u_\epsilon)_{z_0,\rho}|
    &\leq C\Bigg(\bigg( \fiint_{Q_{\rho}(z_0)}|D_i u_\epsilon
    -(D_i u_\epsilon)_{z_0,\rho}|^2 \,\dx\dt \bigg)^{\frac{1}{2}} + \rho^{\beta} \Bigg).
\end{align}
    Moreover, for any~$\theta\in(0,\frac{1}{2}]$ there holds the quantitative oscillation estimate
    \begin{align} \label{excessquantitative}
    \essosc\limits_{Q_{\theta \frac{\rho}{2}}(z_0)} \big(D_i u_\epsilon
    -(D_i u_\epsilon)_{z_0,\rho}\big) \leq C \bigg(\theta^{\alpha} \essosc\limits_{Q_{\frac{\rho}{2}}(z_0)} \big( D_i u_\epsilon - (D_i u_\epsilon)_{z_0,\rho} \big) +\rho^{\beta}\bigg).
\end{align}
\end{myproposition}

\begin{proof}

We differentiate the weak form~\eqref{weakformapprox} of~$u_\epsilon$ by testing the latter with~$D_i\phi$, where~$\phi\in C^1_0(Q_{\frac{\rho}{2}}(z_0))$ denotes an arbitrary smooth test function. We note that this procedure is indeed admissible due to the previous Proposition~\ref{lowerboundprop}, which establishes that~$|Du_\epsilon|_E\geq 1+\delta+\mbox{$\frac{\mu}{4}$}$ a.e. in~$Q_{\frac{\rho}{2}}(z_0)$, such that, according to~$\eqref{fregularity}_3$,~$\hat{\F}_\epsilon(x,t,Du_\epsilon(x,t))$ is of class~$C^2$ with respect to the gradient variable for any~$(x,t)\in Q_{\frac{\rho}{2}}(z_0)$. Subsequently, we integrate by parts in the evolution term and also in the diffusion term, which yields that for any~$i=1,\ldots,n$, the function~$v\coloneqq D_i u_\epsilon$ is a weak solution to
\begin{equation} \label{sollinell}
      \iint_{Q_{\frac{\rho}{2}}(z_0)} \big(-v\, \partial_t \phi +  \hat{\B}_\epsilon(x,t,Du_\epsilon)(D v,D\phi) \big) \,\dx\dt =  - \iint_{Q_{\frac{\rho}{2}}(z_0)} \bigg(f D_i\phi + \sum\limits_{j=1}^n g^i_j D_j \phi \bigg) \,\dx\dt,
\end{equation}
where the bounded coefficients~$g^i \colon Q_{\frac{\rho}{2}}(z_0) \to \R^n$ are given by
$$ |g^i(x,t)| \coloneqq |\partial_{x_i} \nabla\hat{\F}(x,t,Du_\epsilon)| \leq C(\hat{C}_\F,N,r_E),$$
with the bound being an immediate consequence of the Lipschitz assumption~\eqref{hatflipschitz}. Moreover, the bilinear form~$\hat{\mathcal{B}}_\epsilon(x,t,Du_\epsilon)$ is elliptic and bounded for any~$(x,t)\in Q_{\frac{\rho}{2}}(z_0)$, with an ellipticity constant that is independent of the parameter~$\epsilon\in(0,1]$, i.e. there exists a positive constant~$C=C(C_\F(\delta),\delta)\geq 1$ with
\begin{align} \label{ellipticityohneeps}
  C^{-1}|\xi|^2 \leq \hat{\mathcal{B}}_\epsilon(x,t,Du_\epsilon)(\xi,\xi) \leq C|\xi|^2  
\end{align}
for any~$(x,t)\in Q_{\frac{\rho}{2}}(z_0)$ and any~$\xi\in\R^n$. The ellipticity follows from Lemma~\ref{bilinearelliptic} and from the fact that there holds~$|Du_\epsilon|_E \geq 1+\delta+\frac{\mu}{4} \geq 1+\frac{5\delta}{4}$ a.e. in~$Q_{\frac{\rho}{2}}(z_0)$, a consequence of Proposition~\ref{lowerboundprop} and assumption~\eqref{deltamu}. 
The upper bound in~\eqref{ellipticityohneeps} is an immediate consequence of the very definition of~\eqref{bilinear}. Moreover, the constant~$C$ can be chosen to depend on~$C_\F(\delta)$ rather than~$C_\F$, where we recall~\eqref{hatfvoraussetzung}, which is possible due to~$\delta\in(0,1]$ and the set inclusion
$$ \{ \xi\in\R^n:|\xi|_E\geq \mbox{$\frac{5}{4}$}\delta\} \subset \R^n\setminus E_\delta. $$
Due to the fact that~$v$ is a weak solution to the equation~\eqref{sollinell} that is linear and has bounded and elliptic coefficients, we infer that~$v$ belongs to a parabolic De Giorgi class, in the sense that for any level~$k>0$ and any~$\sigma>0$ there holds
  \begin{align*}
        \esssup\limits_{ t\in \Lambda_{\sigma \frac{\rho}{2} }(t_0)}  \int_{B_{\sigma \frac{\rho}{2} }(x_0) \times\{t \}} (v &-k)^2_{\pm}\,\dx + \iint_{Q_{\sigma \frac{\rho}{2}}(z_0)} |D(v - k)_{\pm}|^2\,\dx\dt \\
        & \leq C \Bigg( \iint_{Q_{\frac{\rho}{2}}(z_0) } \bigg(\frac{(v-k)^2_{\pm}}{(1-\sigma)^2 \rho^2} + \frac{v-k)^2_{\pm}}{(1-\sigma) \rho } \bigg)\,\dx\dt \\
        &\quad + C\big(\rho^{2(1-\beta)}+\|f\|^2_{L^{n+2+\sigma}(Q_{\frac{\rho}{2}}(z_0))} \big) |Q_{\frac{\rho}{2}}(z_0) \cap \{v\gtrless k\}|^{1-\frac{2(1-\beta)}{n+2}} \Bigg)
    \end{align*}
with a constant~$C=C(C_\F(\delta),\hat{C}_\F,\delta,\|f\|_{L^{n+2+\sigma}(Q_R)},M,N,n,r_E,\sigma)$. Due to similarity, we omit the details and refer the reader to~\cite[Chapter~12,~(1.8)]{dibenedetto2023parabolic} and also to~\cite[Lemma~7.2]{degeneratesystems} for the proof. Since~$v=D_i u_\epsilon$ satisfies the identity~\eqref{sollinell}, the same holds true for the translated function~$D_i u_\epsilon-(D_i u_\epsilon)_{z_0,\rho}$. An application of~\cite[Chapter 12,~Theorem~2.1]{dibenedetto2023parabolic} with the choice~$\delta = \frac{\beta}{n+2}$ thus yields for the function~$D_i u_\epsilon-(D_i u_\epsilon)_{z_0,\rho}$ the quantitative~$L^\infty$-estimate
\begin{align*} 
    \esssup\limits_{Q_{\frac{\rho}{2}}(z_0)} |D_i u_\epsilon
    -(D_i u_\epsilon)_{z_0,\rho}|
    &\leq C\Bigg(\bigg( \fiint_{Q_{\rho}(z_0)}|D_i u_\epsilon
    -(D_i u_\epsilon)_{z_0,\rho}|^2 \,\dx\dt \bigg)^{\frac{1}{2}} + \rho^{\beta} \Bigg)
\end{align*}
with a constant~$C=C(C_\F(\delta),\hat{C}_\F,\delta,\|f\|_{L^{n+2+\sigma}(Q_R)},M,N,n,R,r_E,\sigma)$. Moreover, we employ the quantitative excess-decay estimate for parabolic De Giorgi classes, which can be found in~\cite[Chapter~12,~Theorem 3.1]{dibenedetto2023parabolic}. An application of the latter yields for any~$\theta\in(0,1]$ the excess-decay estimate
\begin{align*} 
    \essosc\limits_{Q_{\theta \frac{\rho}{2}}(z_0)} \big(D_i u_\epsilon
    -(D_i u_\epsilon)_{z_0,\rho}\big) \leq C \bigg(\theta^{\alpha} \essosc\limits_{Q_{\frac{\rho}{2}}(z_0)} \big( D_i u_\epsilon - (D_i u_\epsilon)_{z_0,\rho} \big) +\rho^{\beta}\bigg)
\end{align*}
for some constant~$C >1$ and an exponent~$\alpha\in(0,1)$ that both depend on the data as above.
We note that a similar excess-decay result has already been given in~\cite{ladyzenskaia1968linear}. However, for convenience and the quantitative form, we employ the former result of~\cite{dibenedetto2023parabolic}.
\end{proof}


The previous Proposition~\ref{dibenedetto} allows us to obtain a quantitative excess-decay estimate for the excess~\eqref{excess}, constituting the major component for the proof of Proposition~\ref{nondegenerateproposition}.


\begin{mylem} \label{excessdecay}
Let~$\delta,\epsilon\in(0,1]$ and~$u_\epsilon$ denote the weak solution to~\eqref{approx}. Assume there exists~$\nu\in(0,\frac{1}{4}]$, such that~\eqref{lowerboundvs} holds true for at least one~$e^*\in \partial E^*$. Then, there exists a constant~$\Tilde{C}\geq  1$ and an exponent~$\Tilde{\alpha}\in(0,1)$, both depending on
\begin{align*}
    \Tilde{C} &= \Tilde{C}(C_\F(\delta),\hat{C}_\F,\delta,\|f\|_{L^{n+2+\sigma}(Q_R)},M,N,n,R,R_E,r_E,\sigma), \\
    \Tilde{\alpha} &= \Tilde{\alpha}(C_\F(\delta),\hat{C}_\F,\delta,\|f\|_{L^{n+2+\sigma}(Q_R)},M,N,n,r_E,\sigma),
\end{align*}
such that the quantitative excess-decay estimate
\begin{equation} \label{excessdecayest}
   \Phi(z_0,\theta\rho)\leq \Tilde{C} \big( \theta^{2\Tilde{\alpha}} \Phi(z_0,\rho)+\rho^{\Tilde{\alpha}}\mu^2\big)
\end{equation}
holds true for any~$\theta\in(0,1]$, provided that~$\rho\in(0,1]$.
\end{mylem}


\begin{proof}
In the case where~$\theta\in(\frac{1}{2},1]$, the claimed excess-decay estimate~\eqref{excessdecayest} readily follows by an application of the~$L^2$-minimality of the mean value
\begin{equation*}
    \Phi(z_0,\theta\rho) \leq C(n)\Phi(z_0,\rho) \leq 2^{2\Tilde{\alpha}}C(n)\theta^{2\Tilde{\alpha}} \Phi(z_0,\rho) \leq C(n) \theta^{2\Tilde{\alpha}} \Phi(z_0,\rho)
\end{equation*}
for any~$\Tilde{\alpha}\in(0,1)$. Thus, let us treat the case where~$\theta\in(0,\frac{1}{2}]$, such that Proposition~\ref{dibenedetto} is applicable. The estimates~\eqref{excesssupbound},~\eqref{excessquantitative}, and also~\eqref{deltamu}, altogether yield
\begin{align*}
    \fiint_{Q_{\theta \rho}(z_0)}| & D_i u_\epsilon - (D_i u_\epsilon)_{z_0,\theta\frac{\rho}{2}}|^2\,\dx\dt \\
    &\leq \bigg(\essosc\limits_{Q_{\theta\rho}(z_0)} D_i u_\epsilon \bigg)^2 \\
    &= \bigg(\essosc\limits_{Q_{\theta\rho}(z_0)} [D_i u_\epsilon - (D_i u_\epsilon)_{z_0,\rho}] \bigg)^2 \\
    &\leq C \Bigg(\theta^{2\alpha} \essosc\limits_{Q_{\frac{\rho}{2}}(z_0)} |D_i u_\epsilon - (D_i u_\epsilon)_{z_0,\rho}|^2
    +\rho^{2\beta}\Bigg) \\
    &\leq \Tilde{C} \Bigg(\theta^{2\Tilde{\alpha}} \fiint_{Q_{\rho}(z_0)}|D_i u_\epsilon - (D_i u_\epsilon)_{z_0,\rho}|^2\,\dx\dt +  \rho^{2\Tilde{\alpha}} \mu^2 \Bigg)
\end{align*}
with constant~$C\geq 1$ and exponent~$\Tilde{\alpha}\in(0,1)$ depending on the given data as stated in the Proposition, where we set~$\Tilde{\alpha}\coloneqq\min\{\alpha,\frac{\beta}{2} \}$ for~$\beta\in(0,1)$ given in~\eqref{beta}. In the very last step, we estimated the essential oscillation further above through the essential supremum and also exploited the fact that~$\rho\in(0,1]$. Since $i\in\{1,\ldots,n\}$ was chosen arbitrarily, we sum up with respect to $i=1,\ldots,n$ and obtain the claimed estimate~\eqref{excessdecayest}. 
\end{proof}


At this point, all ingredients are in place to treat the proof of Proposition~\ref{nondegenerateproposition}.


\begin{proof}[\textbf{\upshape Proof of Proposition~\ref{nondegenerateproposition}}]
By~$\Tilde{\alpha}\in(0,1)$ and~$\Tilde{C}>1$ we denote the exponent and the constant from Lemma~\ref{excessdecay}. Let us choose the parameter~$\theta\in(0,1]$ as follows
$$\theta\coloneqq \min \Big\{\frac{4M^2r^2_E+\delta^2}{\Tilde{C}\delta^2},(2\Tilde{C})^{-\frac{1}{2(\Tilde{\alpha}-\alpha)}},2^{-\frac{1}{\alpha}} \Big\},$$
where~$\alpha\in(0,\Tilde{\alpha})$ denotes an arbitrarily chosen exponent that is strictly smaller than~$\Tilde{\alpha}$. Moreover, we consider the radius~$\Tilde{\rho}\coloneqq \min\{\theta,1\}.$ Accordingly, both~$\theta$ and~$\Tilde{\rho}$ depend on the full data 
$$(C_\F(\delta),\hat{C}_\F,\delta,\|f\|_{L^{n+2+\sigma}(Q_R)},M,N,n,R,R_E,r_E,\sigma).$$
For the following reasoning, we consider a radius~$\rho\in(0,\Tilde{\rho}]$ and a cylinder~$Q_{2\rho}(z_0) \Subset Q_R $. The bounds~\eqref{minkowskischranke} and~\eqref{schrankezwei} enable us to estimate the excess further above by
\begin{align} \label{decayabsch}
    \Phi(z_0,\rho) \leq 4 \Tilde{M}^2 = 4 M^2 r^2_E.
\end{align}
We will now show inductively that for any~$i\in\N$ there holds
\begin{equation} \label{excessiterated}
    \Phi(z_0,\theta^i \rho) \leq \theta^{2\alpha i}\mu^2.
\end{equation}
To abbreviate notation, we refer to~$\eqref{excessiterated}_i$ in the~$i$-th step of the iteration. The case~$i=1$ is treated first. Due to our choices of~$\theta$ and~$\Tilde{\rho}$, applying Lemma~\ref{excessdecay} and exploiting the bounds~\eqref{deltamu} and~\eqref{decayabsch}, we achieve
\begin{align*}
    \Phi(z_0,\theta\rho) \leq \Tilde{C} \big(\theta^{2\Tilde{\alpha}}\Phi(z_0,\rho)+\rho^{2 \Tilde{\alpha}}\mu^2\big) \leq \Tilde{C}\theta^{2\Tilde{\alpha}} \bigg(\frac{4M^2}{\delta^2}+1\bigg) \mu^2 \leq \theta^{2\alpha}\mu^2,
\end{align*}
verifying that~$\eqref{excessiterated}_1$ holds indeed true. Next, we treat the case where~$i>1$ under the assumption that~$\eqref{excessiterated}_{i-1}$ is satisfied. By additionally employing Lemma~\ref{excessdecay} and utilizing the choices of~$\theta$ and~$\Tilde{\rho}$ once more, we obtain
\begin{align*}
    \Phi(z_0,\theta^{i}\rho) &\leq \Tilde{C} \big(\theta^{2\Tilde{\alpha}}\Phi(z_0,\theta^{i-1}\rho)+ (\theta^{i-1}\rho)^{2\Tilde{\alpha}}\mu^2\big) \\
    &\leq \Tilde{C} \big(\theta^{2\Tilde{\alpha}} \theta^{2\alpha(i-1)}\mu^2+ \theta^{2\alpha (i-1)} \rho^{2\Tilde{\alpha}} \mu^2\big) \\
    &\leq \Tilde{C} \theta^{2\alpha i} \big(\theta^{2(\Tilde{\alpha}-\alpha)} + \theta^{-2\alpha i} \Tilde{\rho}^{2\Tilde{\alpha}} \mu^2\big) \mu^2 \\
    & \leq  2\Tilde{C} \theta^{2(\Tilde{\alpha}-\alpha)} \theta^{2\alpha i} \mu^2 \\
    & \leq \theta^{2\alpha i} \mu^2.
\end{align*}
Therefore, also~$\eqref{excessiterated}_i$ holds true. Next, we aim to establish both assertions~\eqref{lebesguerepresentant} and~\eqref{excessgdelta}. Lemma~\ref{gdeltalem} and~$\eqref{excessiterated}_{i-1}$ yield for any~$i\geq 2$ the following estimate
\begin{align} \label{intmedstep}
    \Tilde{\Phi}(z_0,\theta^{i-1} \rho) &\coloneqq \fiint_{Q_{\theta^{i-1} \rho}(z_0)} |\G_{2\delta}(Du_\epsilon)-(\G_{2\delta}(Du_\epsilon))_{z_0,\theta^{i-1}}\rho|^2\,\dx\dt \\
    &\leq  \fiint_{Q_{\theta^{i-1} \rho}(z_0)} |\G_{2\delta}(Du_\epsilon)-\G_{2\delta}((Du_\epsilon)_{z_0,\theta^{i-1}\rho})|^2\,\dx \dt \nonumber \\
    &\leq C \Phi(z_0,\theta^{i-1}\rho) \leq C \theta^{2\alpha(i-1)}\mu^2, \nonumber
\end{align}
with~$C=C(R_E,r_E)$. This further implies
\begin{align*}
    |(\G_{2\delta}(Du_\epsilon))_{z_0,\theta^{i-1}\rho}   - (\G_{2\delta}(Du_\epsilon))_{z_0,\theta^i\rho}|^2 
    &\leq \fiint_{Q_{\theta^i \rho}(z_0)} |\G_{2\delta}(Du_\epsilon)-(\G_{2\delta}(Du_\epsilon))_{z_0,\theta^{i-1}\rho}|^2\,\dx\dt  \\ 
    &\leq \frac{1}{\theta^{n+2}} \Tilde{\Phi}(z_0,\theta^{i-1} \rho) \leq \frac{C}{\theta^{n+2}} \theta^{2\alpha(i-1)}\mu^2. 
\end{align*}
Considering two arbitrarily chosen natural numbers~$l<m$ and taking roots in the preceding estimate, we obtain, due to the assumption that~$\theta^\alpha \leq \frac{1}{2}$, the following
\begin{align*}
    |(\G_{2\delta}(Du_\epsilon))_{z_0,\theta^{l}\rho} - (\G_{2\delta}(Du_\epsilon))_{z_0,\theta^m\rho}|
    &\leq \sum\limits_{i=l+1}^{m} |(\G_{2\delta}(Du_\epsilon))_{z_0,\theta^{i-1}\rho}  -(\G_{2\delta}(Du_\epsilon))_{z_0,\theta^i\rho}| \\
    &\leq C \theta^{-\frac{n+2}{2}} \sum\limits_{i=l+1}^{m} \theta^{\alpha (i-1)} \mu \leq C \theta^{-\frac{n+2}{2}} \frac{\theta^{\alpha l}}{1-\theta^\alpha} \mu \leq C \theta^{-\frac{n+2}{2}} \theta^{\alpha l}\mu.
\end{align*}
 Consequently, we have established that the sequence~$((\G_{2\delta}(Du_\epsilon))_{z_0,\theta^{i}\rho} )_{i\in\N}$ is a Cauchy sequence in~$\R^n$. We denote its limit by
$$\Gamma_{z_0} \coloneqq \lim\limits_{i\to\infty} (\G_{2\delta}(Du_\epsilon))_{z_0,\theta^{i}\rho}.$$
Passing to the limit~$m\to\infty$ in the preceding estimate, we derive
\begin{align*}
    |(\G_{2\delta}(Du_\epsilon))_{z_0,\theta^{l}\rho}-\Gamma_{z_0}| \leq C \theta^{-\frac{n+2}{2}} \theta^{\alpha l}\mu \qquad \mbox{for any~$l\in\N$.}
\end{align*}
By combining this result with estimate~\eqref{intmedstep}, there holds
\begin{align*}
    \fiint_{Q_{\theta^l \rho}(z_0)} |\G_{2\delta}(Du_\epsilon)-\Gamma_{z_0}|^2\,\dx\dt &\leq 2\Tilde{\Phi}(z_0,\theta^{l}\rho) + 2 |(\G_{2\delta}(Du_\epsilon))_{z_0,\theta^{l}\rho}-\Gamma_{z_0}|^2 \\
    &\leq C \theta^{2\alpha l}\mu^2 + C \theta^{-(n+2)} \theta^{\alpha 2l}\mu^2 \\
    &\leq C \theta^{-(n+2)} \theta^{2\alpha l}\mu^2.
\end{align*}
Eventually, we convert the preceding estimate into the excess-decay estimate~\eqref{excessgdelta}. Given~$r\in(0,\rho]$, we choose~$l\in\N$, such that~$\theta^{l+1}\rho <r \leq \theta^{l}\rho$ is satisfied. Then, our choice of~$\theta$ and the preceding inequality yield 
\begin{align*}
    \fiint_{Q_r(z_0)} |\G_{2\delta}(Du_\epsilon) - \Gamma_{z_0}|^2\,\dx\dt &\leq \frac{1}{\theta^{n+2}} \fiint_{Q_{\theta^{l}\rho}(z_0)} |\G_{2\delta}(Du_\epsilon) - \Gamma_{z_0}|^2\,\dx\dt \\
    &\leq C \theta^{-2(n+2)} \theta^{2\alpha l}\mu^2 \\
    &\leq C\Big(\frac{r}{\rho} \Big)^{2\alpha}\mu^2
\end{align*}
with~$C=C(C_\F(\delta),\hat{C}_\F,\delta,\|f\|_{L^{n+2+\sigma}(Q_R)},M,N,n,R,R_E,r_E,\sigma)$. Finally, we apply Jensen's inequality, such that there holds
\begin{align*}
    |(\G_{2\delta}(Du_\epsilon))_{z_0,r} - \Gamma_{z_0}|^2 \leq \fiint_{Q_r(z_0)} |\G_{2\delta}(Du_\epsilon) - \Gamma_{z_0}|^2 \,\dx\dt \leq C\Big(\frac{r}{\rho} \Big)^{2\alpha}\mu^2,
\end{align*}
which furthermore establishes
$$\Gamma_{z_0} = \lim\limits_{r\downarrow 0} (\G_{2\delta}(Du_\epsilon))_{z_0,r}.$$
The bound~\eqref{lebesguebound} follows from estimate~\eqref{schrankeeins} and the fact that there holds~$|(\G_{2\delta}(Du))_{z_0,r}|\leq R_E\mu$ for any~$r\in(0,\rho]$.
\end{proof}


\section{The degenerate regime} \label{sec:degenerate}
In this section we devote ourselves to the treatment of the degenerate regime. In particular, we aim to give the proof of Proposition~\ref{degenerateproposition}. For this matter, the function 
$$v_\epsilon \coloneqq ( \partial_{e^*} u_\epsilon - (1+\delta))^2_+.$$
and also its rescaled version~$\Tilde{v}_\epsilon$~\eqref{tildevrescaled}, where~$\delta,\epsilon\in(0,1]$ and~$e^*\in \partial E^*$, turn out expedient. In fact, we have already established that~$ v_\epsilon$ is a weak sub-solution to a linear parabolic equation with elliptic and bounded coefficients, as stated in Lemma~\ref{subsollemma}, and further deduced two classical De Giorgi class-type estimates in form of Lemma~\ref{degiorgilem}. They will serve as a key tool in the subsequent argument. More precisely, we elaborate a tool that is commonly referred to as~\textit{expansion of positivity}, which, together with the measure-theoretic information~\eqref{degeneratemeascond}, yields a quantitative reduction of the supremum of~$|\G_\delta(Du_\epsilon)|_E$ within a smaller cylinder. We recall that throughout this section the set of assumptions~\eqref{schrankeeins},~\eqref{schrankezwei},~\eqref{deltamu}, and also~\eqref{degeneratemeascond} is at our disposal. \,\\


Our next goal is to derive the aforementioned tool that converts measure-theoretic information of~$\Tilde{v}_\epsilon$ at a certain time-slice into pointwise estimates at later times. This tool is commonly known as~\textit{expansion of positivity}. 


\begin{myproposition} [Expansion of positivity] \label{expansion}
    Let~$\Tilde{v}_\epsilon \leq 1$ denote the weak sub-solution defined in~\eqref{tildevrescaled}. Further, let~$\rho\in(0,1]$,~$\Tilde{t}\in(-\rho^2,-\frac{1}{2}\rho^2)$, and~$\Gamma>0$, such that~$\Tilde{t}+\Gamma\rho^2< 0$ holds true. For any~$\alpha\in(0,1)$ and~$L\in(0,1)$, there exists a parameter~$\eta\in(0,1)$ depending on the data
    $$(\alpha,C_\F(\delta),\hat{C}_\F,\delta,\|f\|_{L^{n+2+\sigma}(Q_R)},\Gamma,M,N,n,R,r_E,\sigma),$$
    such that whenever the assumption
    \begin{align} \label{expansionmeascond}
        |B_\rho\cap\{1-\Tilde{v}_{\epsilon}(\cdot,\Tilde{t})>L\}| \geq \alpha |Q_\rho|
    \end{align}
    is satisfied, then either for all times
    $$\Tilde{t}+\Gamma\rho^2<t\leq 0$$
    there holds
    $$1-\Tilde{v}_{\epsilon}(\cdot,t)\geq \eta L \qquad \mbox{a.e. in~$B_\rho$},$$
    or we have~$\eta L \leq C\rho^\beta$ with a constant
    $$C=C(C_\F(\delta),\hat{C}_\F,\delta,\|f\|_{L^{n+2+\sigma}(Q_R)},M,N,n,R,r_E,\sigma).$$
\end{myproposition}


\begin{proof}
The result follows directly from the De Giorgi class-type energy estimates from Lemma~\ref{degiorgilem} in Section~\ref{sec:energyestimates} and the arguments presented in~\cite[Chapter~12, Section~10.1 Proof of Theorem 10.1. Preliminaries]{dibenedetto2023parabolic} and~\cite[Chapter~12, Section~10.2 Proof of Theorem 10.1. Expansion of Positivity]{dibenedetto2023parabolic}. These sections develop the framework within the context of a Harnack inequality for parabolic De Giorgi classes. As in the proof of Proposition~\ref{dibenedetto}, we note that the structural assumptions outlined in~\cite[Chapter~12, Section~1 Quasi-Linear Equations and De Giorgi Classes,~(1.2)]{dibenedetto2023parabolic} are satisfied in our setting. Consequently, the expansion of positivity stated in Proposition~\ref{expansion} immediately follows from these references.

However, it is important to highlight that, in~\cite[Chapter~12, Section~10.2. Proof of Theorem 10.1. Expansion of Positivity, Proposition~10.1]{dibenedetto2023parabolic}, the quantity~$\Gamma_*$ corresponds to~$C\|f\|_{L^{n+2+\sigma}(Q_R)} \rho^\beta$ in our context, where the constant~$C$ depends only on the full data set~$(C_\F(\delta), \hat{C}_\F, \delta, M, N, n, R, r_E, \sigma)$. Furthermore, as noted in~\cite[Chapter~12, Section~10.1 Proof of Theorem 10.1. Preliminaries,~(10.2)]{dibenedetto2023parabolic}, the additional term~$\frac{1}{u(x_0,t_0)}$ appears solely due to the rescaling performed immediately prior to that step. For parabolic De Giorgi classes as considered in~\cite[Chapter~12, Section~1 Quasi-Linear Equations and De Giorgi Classes]{dibenedetto2023parabolic}, this term does not naturally arise. The claimed expansion of positivity follows.
\end{proof}


Finally, we prove the main result of the degenerate regime, that is Proposition~\ref{degenerateproposition}.


\begin{proof}[\textbf{\upshape Proof of Proposition~\ref{degenerateproposition}}]
  As before, let~$\Tilde{v}_\epsilon \leq 1$ denote be the weak sub-solution defined in~\eqref{tildevrescaled}. We begin by expressing the measure-theoretic information~\eqref{degeneratemeascond} in terms of~$\Tilde{v}_\epsilon$. We recall that for~$u_\epsilon$ the latter is given by
  \begin{equation*}
      |Q_\rho(z_0)\cap\{ \partial_{e^*}u_\epsilon \leq 1+\delta + (1-\nu)\mu\}| \geq \nu |Q_\rho(z_0)|.
  \end{equation*}
For the re-scaled function~$\Tilde{v}_\epsilon$, this translates to the measure-theoretic information
 \begin{align} \label{tildevmeasure}
        |Q_\rho \cap \{\Tilde{v}_\epsilon \leq (1-\nu)^2 \} | \geq \nu |Q_\rho|.
    \end{align}
The preceding upper bound for~$\Tilde{v}_\epsilon$ in~\eqref{tildevmeasure} can be estimated further above by 
\begin{align*}
    \Tilde{v}_\epsilon \leq  (1-\nu)^2 \leq 1-\nu^2,
\end{align*}
due to~$\nu\in(0,1)$. The measure condition~\eqref{tildevmeasure} thus implies that
\begin{align*}
     |Q_\rho \cap \{\Tilde{v}_\epsilon \leq 1-\nu^2\} | \geq \nu |Q_\rho|.
\end{align*}
 This measure-theoretic information in particular implies the existence of a time~$t_0\in\big(-\rho^2,-\frac{1}{2}\nu\rho^2 \big]$, such that there holds
\begin{align} \label{measureball}
   | B_\rho\cap\{\Tilde{v}_\epsilon(\cdot,t_0)\leq 1-\nu^2 \}| \geq \mbox{$\frac{1}{2}$}\nu |B_{\rho}|,
\end{align}
since otherwise a contradiction is obtained as follows
\begin{align*}
    \nu |Q_\rho| &\leq \displaystyle\int_{-\rho^2}^{0} | B_\rho\cap\{\Tilde{v}_\epsilon (\cdot,t)\leq 1-\nu^2 \}|\,\dt \\
    &= \displaystyle\int_{-\rho^2}^{-\frac{1}{2}\nu\rho^2} | B_\rho\cap\{\Tilde{v}_\epsilon(\cdot,t)\leq 1-\nu^2 \}|\,\dt + \displaystyle\int_{-\frac{1}{2}\nu\rho^2}^{0} | B_\rho\cap\{\Tilde{v}_\epsilon(\cdot,t)\leq 1-\nu^2 \}|\,\dt \\
    &\leq \mbox{$\frac{1}{2}$} \nu|Q_\rho|\big(1-\mbox{$\frac{1}{2}$}\nu \big) + {\textstyle\frac{1}{2}}\nu|Q_\rho| \\
    &= \nu{\textstyle \big(1-\frac{1}{4}\nu\big)} |Q_\rho|.
\end{align*}
Consequently, the measure-theoretic information~\eqref{measureball} allows an application of the expansion of positivity stated in Proposition~\ref{expansion} with the choice~$\alpha \coloneqq \frac{1}{2}\nu$,~$\Gamma\coloneqq \frac{1}{4}\nu$, and~$L\coloneqq \nu^2$. This way, we obtain the existence of a parameter~$\eta\in(0,1)$ that depends on 
$$\eta=\eta(C_\F(\delta),\hat{C}_\F,\delta,\|f\|_{L^{n+2+\sigma}(Q_R)},M,N,n,\nu,R,r_E,\sigma),$$
such that either there holds~$\eta \nu^2 \leq C\rho^\beta$, or we have
\begin{equation} \label{tildevshrink}
    \Tilde{v}_\epsilon \leq 1-\eta\nu^2 \qquad \mbox{a.e. in~$B_\rho \times \big(-\frac{1}{4}\nu\rho^2,0 \big]$}.
\end{equation}
The bound~\eqref{tildevshrink} can be translated to the original solution~$u_\epsilon$ itself, which yields
\begin{equation*} 
    \partial_{e^*} u_\epsilon-(1+\delta) \leq \sqrt{1-\eta\nu^2}\mu \qquad \mbox{a.e. in~$B_\rho(x_0) \times \big(t_0-\frac{1}{4}\nu\rho^2,t_0 \big]$}.
\end{equation*}
Due to the fact that~$\nu\in(0,\frac{1}{4}]$, there holds~$\sqrt{1-\eta\nu^2}\in[\frac{1}{2},1)$, ensuring a lower bound for the given quantity. In particular, way may choose~$\sqrt{1-\eta\nu^2}\in\big[2^{-\frac{\Tilde{\beta}}{2}},1\big)$, as stated in Proposition~\ref{degenerateproposition}. Moreover, by choosing a radius~$\hat{\rho}>0$ small enough in dependence on the given data 
$$(C_\F(\delta),\hat{C}_\F,\delta,\|f\|_{L^{n+2+\sigma}(Q_R)},M,N,n,\nu,R,r_E,\sigma),$$
such that~$C\Tilde{\rho}^\beta<\eta\nu^2$, the first alternative of Proposition~\ref{expansion} never applies. Eventually, since~$e^*\in\partial E^*$ remains chosen arbitrarily up to this point, we pass to the supremum over all~$e^*\in\partial E^*$, and use~\eqref{minkowskialternativ} and the definition of~$\G_\delta$, to obtain
$$ |\G_\delta(Du_\epsilon)|_E \leq \kappa\mu \qquad\mbox{a.e. in~$Q_{\Tilde{\nu}\rho}(z_0)$}, $$
 with~$\Tilde{\nu}\coloneqq\frac{\sqrt{\nu}}{2}$ and~$\kappa\coloneqq \sqrt{1-\eta\nu^2} \in \big[2^{-\frac{\Tilde{\beta}}{2}},1\big)$. This finishes the proof of the proposition.
\end{proof}


 \nocite{*}
 \,\\
\bibliographystyle{plain}
\bibliography{Literature.bib}


\end{document}